\documentclass[12pt, reqno]{amsart}
\usepackage[english]{babel}
\usepackage[utf8]{inputenc}
\usepackage{lipsum}
\usepackage{dsfont}
\usepackage{mathrsfs}
\usepackage{stix}
\usepackage{fullpage}
\usepackage{mathtools}
\usepackage{amsmath}
\usepackage{bbm}
\usepackage{amsfonts}
\usepackage{tikz}
\usepackage{tikz-cd}
\usepackage{graphicx}
\usetikzlibrary{decorations.pathreplacing,angles,quotes}
\usetikzlibrary{arrows,chains,positioning,scopes,quotes}
\newtheorem{defi}{Definition}[section]

\newtheorem{lema}[defi]{Lemma}
\newtheorem{teo}[defi]{Theorem}
\newtheorem{rem}[defi]{Remark}

\newtheorem{coro}[defi]{Corollary}
\newtheorem{pro}[defi]{Proposition}
\newtheorem*{rem*}{Remark}

\newcommand{\D}{\mathbb{D}}
\newcommand{\C}{\mathbb{C}}
\newcommand{\Q}{\mathbb{Q}}

\newcommand{\R}{\mathbb{R}}
\newcommand{\N}{\mathbb{N}}
\newcommand{\Z}{\mathbb{Z}}

\newcommand{\1}{\mathds{1}}
\newcommand{\G}{\mathbb{G}}

\newcommand{\smallstar}{\scalebox{0.8}{$\, \star \, $}}
\newcommand{\g}{\mathfrak{g}}

\newcommand{\interior}[1]{%
  {\kern0pt#1}^{\mathrm{o}}%
}

\usepackage[colorlinks=false]{hyperref}
\renewcommand\eqref[1]{(\ref{#1})} 
\begin{document}

\title[$L^r$-Multipliers on compact $p$-adic Lie groups. ]
 {$L^r$-Multipliers on compact $p$-adic Lie groups }

\author{
  J.P. Velasquez-Rodriguez
}

\newcommand{\Addresses}{{
  \bigskip
  \footnotesize

 J.P. Velasquez-Rodriguez, \textsc{Departamento de ciencias naturales y matematicas, Pontificia Universidad Javeriana, Cali-Colombia}\par\nopagebreak
  \textit{E-mail address:} \texttt{juanpablovr2025@gmail.com / juanpablo.velasquez@javerianacali.edu.co}}

}




\date{\today}
\begin{abstract}
Let $p$ be a prime number, and let $\mathbb{G}$ be a compact $p$-adic Lie group. This work provides multiplier theorems for invariant operators on $\mathbb{G}$ acting on $L^r_\alpha(\mathbb{G})$, $1<r<\infty$, $\alpha>0$, in terms of the Ruzhansky-Turunen difference operators and Saloff-Coste's condition. As an application, a Littlewood-Paley decomposition is proven, together with the $L^r$-boundedness of bounded functions of the Vladimirov-Taibleson operator on compact Vilenkin groups.  
\end{abstract}
\maketitle
\tableofcontents
\section{Introduction}
The problem of $L^r$-multipliers is a fascinating and challenging open question that has captivated functional analysts for decades. Initially explored within the framework of Euclidean spaces, this problem has been the focus of extensive research aimed at establishing sufficient conditions (necessary and sufficient would be cooler!) for the $L^r$-boundedness of operators. More precisely, the problem considers a special class of linear operators defined through Fourier analysis, often called \emph{Fourier multipliers}. In $\mathbb{R}^d$, these operators can be expressed as 
\[
T_\sigma f(x) = \int_{\mathbb{R}^d} \sigma(\xi) \hat{f}(\xi) e^{2 \pi i x \cdot \xi} \, d\xi, \quad f \in L^1(\mathbb{R}^d) \cap L^2(\mathbb{R}^d),
\] where $\widehat{f}$ is the Fourier transform of $f$, and the function $\sigma: \mathbb{R}^d \to \mathbb{C}$ is often referred to as the \emph{symbol} of the operator. While the continuity of these operators is straightforward to characterize in terms of the symbol for the Hilbert space $L^2(\mathbb{R}^d)$, determining their boundedness on the Banach space $L^r(\mathbb{R}^d)$, $r \neq 2$, poses a significant challenge. So, as it is natural for any interesting problem, numerous authors have investigated this issue, providing partial answers by imposing various conditions on the behaviour of the symbol. Over time, with advancements in the theory of smooth manifolds and topological groups, these results have been extended to a variety of contexts where some notion of multipliers exists, particularly in spaces that support a version of the Fourier transform. 

In this paper, we focus on the framework of totally disconnected topological spaces, with particular emphasis on the compact case, exploring in detail compact $p$-adic Lie groups, where $p$ is some prime number. Our goal is to establish several multiplier theorems in this setting, and discuss along the way some related problems. This line of research forms part of a broader effort to extend the theory of Vladimirov-type operators \cite{Bendikov2014, Dragovich2017} beyond its classical origins on the $p$-adic numbers $\Q_p$, toward $p$-adic domains \cite{zuniga1, zuniga2, Zuniga3} and $p$-adic manifolds \cite{bradley1, bradley2, bradley3}. Despite the substantial advances achieved in recent years, even the fundamental question of what should constitute an appropriate notion of “Laplacian” on a general $p$-adic manifold or domain remains unsettled. The situation is even more subtle in the case of directional Vladimirov operators, where the lack of a canonical geometric structure makes the formulation of a satisfactory definition particularly elusive.

Nevertheless, when the underlying space is endowed with additional algebraic structure, like that of a locally compact group, one gains significant analytical advantages over the general manifold setting, like how it becomes possible to study the hypoellypticity of differential operators on certain Lie groups through their representation theory \cite{Rockland1978}. In this spirit, we follow a strategy similar to that developed in \cite{Kitada1987, Onneweer1985, Onneweer1989}, where the authors exploit the structure of certain classes of locally compact abelian groups to investigate the boundedness of Fourier multipliers. In particular, several Vladimirov-type operators can be realized as Fourier multipliers with respect to the standard group Fourier transform, just like in \cite{Rockland1978, Ruzhansky2010}, a fact that enables us to derive the multiplier theorems presented in this work. This is a first step towards a better study of the properties of operators like the Vladimirov sub-Laplacian \cite{VelasquezRodriguez2025}, and more general translation-invariant operators on (non-commutative) $p$-adic Lie groups.

The $p$-adic side of the theory of pseudo-differential operators can be regarded as a natural counterpart to the more classical theory developed for real groups. Structural results for compact groups (see \cite{Hofmann2020}) show that every compact group is either isomorphic to, or isomorphic to something close to, the product of a profinite group and a locally connected group, the latter often being a Lie group. In this sense, the study of analysis on compact groups naturally splits into two complementary directions: one governed by the geometry of Lie groups and real manifolds, and another driven by totally disconnected, profinite structures. This dichotomy resonates with Ostrowski's theorem, which classifies all absolute values on $\mathbb{Q}$ into the Archimedean one and the non-Archimedean $p$-adic ones. While the analytical theory associated with real Lie groups forms the mainstream setting \cite{Rockland1978, Ruzhansky2010}, the $p$-adic world provides a parallel framework with markedly different geometric and analytic features. One of the most striking differences, and the one most relevant for our purposes, concerns the structure of the unitary dual \cite{Boyarchenko2008, Howe1977}. On the $p$-adic side, the duals of certain groups admit parametrizations by special trees, which naturally carry the structure of an ultrametric space. Consequently, once our multiplier theorems are formulated in terms of difference operators, they can be interpreted as estimates for differences of mappings (symbols) defined on ultrametric trees. These trees come as sub-trees of ($d$-dimensional products of) the familiar tree depiction of the unitary dual $\widehat{\mathbb{Z}}_p$ of $\mathbb{Z}_p$, often realized as the Pr{\"u}fer group.
\begin{figure}[h]
\caption{The dual group $\widehat{\Z}_p$, $p=3$, as an infinite tree. Here we can see the first 3 ultrameric balls, which produce finite trees.       }
\centering
\includegraphics[width=0.8\textwidth]{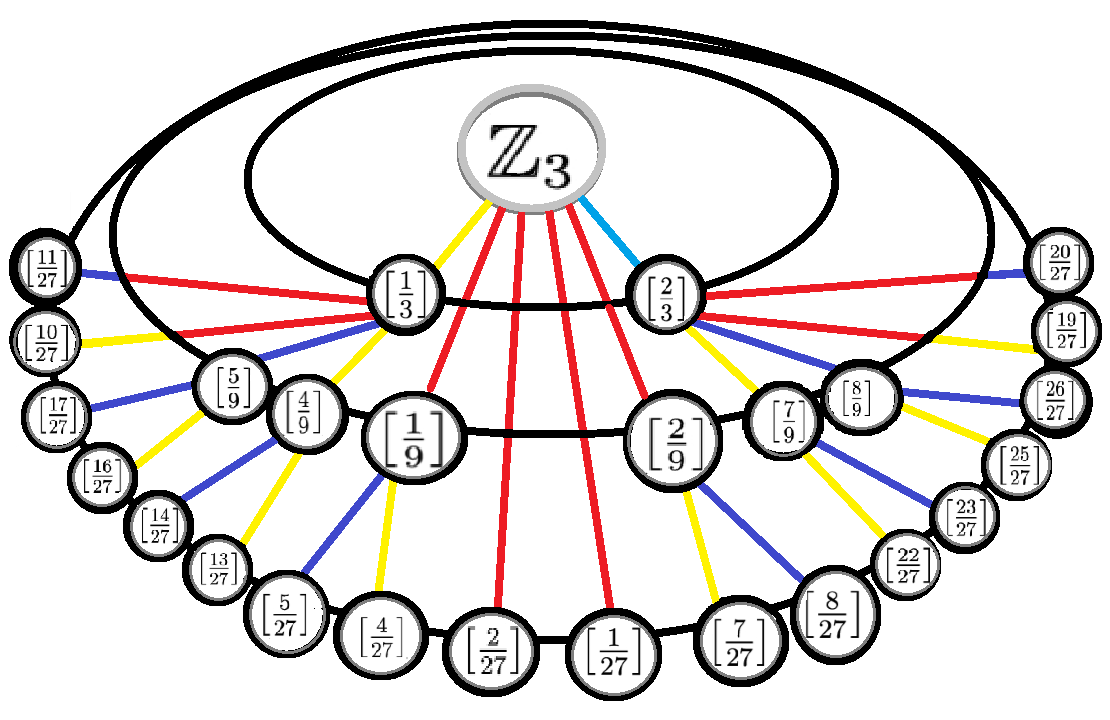}
\end{figure}

\begin{rem}
    The line of reasoning we will be following here \cite{Kitada1987, Onneweer1985, Onneweer1989} is kind of old, but we will mix it up with some Ruzhansky-Turunen theory \cite{Ruzhansky2010, Ruzhansky2015}, and some of Fischer's ideas  \cite{Fischer2015, Fischer2020} to give it and update. The result is meant to be a collection of tools to study Vladimirov-type operators on (non-commutative) $p$-adic Lie groups, appearing here only in their compact form. 
\end{rem}

\begin{rem}
    For the rest of this work $p$ will denote a (usually big enough!) prime number. 
\end{rem}
\subsection{Multipliers on Abelian Vilenkin Groups}
This work is motivated by two main sources. The first is the theory of pseudo-differential operators on Lie groups, as developed for example in \cite{Rockland1978, Ruzhansky2010}. The second originates in the study of Fourier analysis on abelian Vilenkin groups, following the ideas of \cite{SaloffCoste1986} and the references therein. Throughout this paper we will continuously combine techniques from both contexts, and the multiplier theorems we establish here come as a result of mixing them together, while keeping as our primary inspiration the works of Onnewer, Taibleson, and Saloff-Coste on abelian Vilenkin groups \cite{Onneweer1985, SaloffCoste1986}. 

\begin{defi}\normalfont
We say that a l.c.a. group $G$ is an \emph{abelian Vilenkin group} if $(G, \smallstar)$ is endowed with a strictly decreasing sequence of compact open subgroups $\mathcal{G} := \{G_n\}_{n \in \mathbb{Z}}$ such that
\begin{enumerate}
    \item It holds:
    \[
    2 \leq m_n := |G_n / G_{n+1}| < \infty,
    \]
    for every $n \in \mathbb{Z}$.
    \item 
    \[
    G = \bigcup_{n \in \mathbb{Z}} G_n, \quad \text{and} \quad \bigcap_{n \in \mathbb{Z}} G_n = \{e\}.
    \]
    \item The sequence $\{G_n\}_{n \in \mathbb{Z}}$ forms a basis of neighborhoods at $e \in G$.
\end{enumerate}
\end{defi}

Vilenkin groups are metrizable $0$-dimensional topological groups whose dual $\widehat{G}$, in the sense of Pontryagin, is a metrizable topological group too, and also a Vilenkin group with the sequence of sub-groups $$ G_n^\bot:= \{ \chi \in \widehat{G} \, : \, \chi(x)=1, \, \, \text{for all} \, x \in G_n  \}.$$ 
The Fourier transform in this setting maps functions on \(G\) into functions on its dual group \(\widehat{G}\), and the Plancherel theorem gives a connection between $L^2$-norms in both sides. As shown by Taibleson and Saloff-Coste~\cite{SaloffCoste1986, Taibleson1970}, this viewpoint allows one to exploit equivalent norms defined on the Fourier side, and we want to show how the same works on non-abelian groups, by extending to the noncommutative setting two key arguments arising in the commutative theory:

\begin{itemize}
    \item[(i)] The weighted $L^2$-space $L^2 (G, |\cdot|_{\mathcal{G}}^\alpha)$, where $|\cdot|_{\mathcal{G}}$ is the ultrametric of the group associated to $\mathcal{G}$, possess and equivalent norm in terms of difference operators applied to the Fourier transform. Using that and some sort of “product rule” we can get multiplier theorems on $L^2 (G, |\cdot|_{\mathcal{G}}^\alpha)$.
    \item[(ii)] Assume we are given a suitable Calderón--Zygmund decomposition, and a method to establish the weak-type \((1,1)\) boundedness of a certain operator. Then, by applying interpolation arguments, we can extend this boundedness to all the intermediate weighted \(L^{r}\)-spaces.
\end{itemize}

So, for the first part, we have the following equivalence proved in \cite{SaloffCoste1986}:

$$\|f\|_{L^2_\alpha(G)} := \Big( \int_{G} |f(x)|^2 |x|^\alpha_{\mathcal{G}} dx \Big)^{1/2} \asymp \Big( \int_{\widehat{G}} \int_{\widehat{G}} \frac{|\triangle_\eta \widehat{f}(\xi)|^2}{|\eta|^{1+\alpha}_{\widehat{G}}}  d\eta d \xi  \Big)^{1/2},$$where the Pontryagin dual $\widehat{G}$ of the abelian Vilenkin group $G$ is a metrizable group, with ultrametric $|\cdot|_{\widehat{G}}$, and the difference operator is the usual difference
\[
\triangle_\eta \sigma(\xi) = \sigma(\eta + \xi) - \sigma(\xi).
\]This is simply an equivalence of $L^2_\alpha$-norms, which turns out to be quite convenient for understanding multipliers,
and naturally leads to the following theorem: 
 
\begin{teo}[Saloff-Coste, \cite{SaloffCoste1986}]\label{teosaloff}Let $G$ be a locally compact abelian Vilenkin group with an
order bounded sequence of compact open subgroups $\{G_n\}_{n \in \mathbb{Z}}$. Let $\widehat{G}$ be its unitary dual endowed
with the sequence of compact open subgroups $\{\widehat{G}_n\}_{n \in \mathbb{Z}}$ associated to $\{G_n\}_{n \in \mathbb{Z}}$. Let $\sigma \in L^\infty(\widehat{G})$.
Suppose there exists $B, \varepsilon > 0$ such that for all $n \in \mathbb{Z}$
\[
\int_{|\xi|_{\widehat{G}} < |G_n^\bot|} \int_{G_n^\bot \setminus G_{n-1}^\bot}
\frac{|\sigma(\eta + \xi) - \sigma(\eta)|^2}{|\xi|^{2+\varepsilon}_{\widehat{G}}} \, d\xi \, d\eta \leq B^2 |G_n^\bot|^{-\varepsilon}.
\]
Then $T_\sigma$ extends to a bounded operator on $L^r(G)$ for $1 < r < \infty$.
\end{teo}

\begin{rem}
The condition of Theorem 1.13 can be written as
\[
\int_{|\xi|_{\widehat{G}} < |G_n^\bot|} \int_{G_n^\bot \setminus G_{n-1}^\bot}
|\triangle^{1+\varepsilon/2}_\eta \sigma(\xi)|^2 \, d\eta \, d\xi \leq B^2 |G_n^\bot|^{-\varepsilon},
\]
where $\triangle^\beta_\eta$ is the Saloff-Coste difference operator
\[
\triangle^\beta_\eta \sigma(\xi) := \frac{\sigma(\eta + \xi) - \sigma(\xi)}{|\eta|^\beta_{\widehat{G}}}.
\]
\end{rem}

For the second part, we employ an argument typically used for the spaces \(L^{r}\) with \(r \neq 2\): an interpolation method. We do it because in the weighted spaces \(L^{r}_{\alpha}(G)\) becomes considerably more difficult to identify a useful equivalent norm that expresses the boundedness of invariant operators in terms of their symbols. There is probably none. So, we are forced to adopt a less direct approach using the convolution kernel of the operator instead of the symbol: a reasoning  similar to that explored in the commutative setting by Onneweer and Kitada (see \cite{Kitada1987, Onneweer1985, Onneweer1989}). Our aim here is to obtain a noncommutative analogue of the following theorem of Onneweer, which relies in a certain condition called by the authors \textbf{Condition $H(t)$}.

\begin{teo}[Onneweer]\label{teomiltionnewer}
Let $G$ be a locally compact abelian Vilenkin group with an order-bounded sequence $\{G_n\}_{n \in \mathbb{Z}}$ of compact open subgroups. Let $|\cdot|_{\mathcal{G}}$ be the ultrametric associated to the sequence, and for $1 \le r < \infty$ let $L^r_\alpha(G) := L^r(G, |\cdot|^\alpha_{\mathcal{G}})$.
\begin{enumerate}
    \item If $\sigma \in L^\infty(\hat{G})$ and condition $H(t)$ holds for some $1 < t < \infty$, then if $T_\sigma$ is a multiplier on $L^2_{\alpha_0}(G)$ for some $-1/r' < \alpha_0 < 1/r'$, it is a multiplier on $L^r_\alpha(G)$ for all $r, \alpha$ such that $1 < r < \infty$ and $-|\alpha_0| \le \alpha \le (r-1)|\alpha_0|$.
    \item If $\sigma \in L^\infty(\hat{G})$ and condition $H(1)$ holds, then $T_\sigma$ is a multiplier on $L^r(G)$ for $1 < r < \infty$.
\end{enumerate}
\end{teo}

See \cite{Onneweer1985} for the definition of \emph{Condition} $H(t)$ in the abelian case. Finally, as an application of Theorem \ref{teomiltionnewer}, Onneweer and Quek proved the following Littlewood-Paley decomposition:

\begin{teo}[ Quek]
Let $G$ be a locally compact order-bounded abelian Vilenkin group, and let $1 < r < \infty$, $-1 < \alpha < r-1$. Then for all $f \in L^r_\alpha(G)$, it holds
\[
\|f\|_{L^r_\alpha(G)} \asymp \|S f\|_{L^r_\alpha(G)},
\]
where
\[
S f(x) = \left( \sum_{n \in \mathbb{Z}} |\mathcal{F}_G^{-1}(\widehat{f} \mathbb{1}_{G_{n+1}^\bot \setminus G_n^\bot})(x)|^2 \right)^{1/2}.
\]
\end{teo}

Lets turn now to the challenge of extending the previous results to the noncommutative setting, beginning with compact $p$-adic Lie groups. In that case, the dual object of the groups under consideration is no longer a group, which introduces additional complications: its elements are matrix representations, and thus one must carefully define the appropriate difference operators to be used \cite{Fischer2015, Fischer2020}. In addressing these issues, we are led to work with the noncommutative group Fourier transform \cite{Ruzhansky2010}, the representation theory of compact Vilenkin groups \cite{Boyarchenko2008, Corwin2004-kz, Howe1977, VelasquezRodriguez2025}, and the extension of several ideas originating from the theory of pseudo-differential operators on compact Lie groups \cite{Ruzhansky2015}.

\begin{rem}
    I should warn the reader that we will be using a fair amount of non-standard notation, although a lot of notation will be borrowed from \cite{Fischer2015, Fischer2020, Ruzhansky2010, Ruzhansky2015}. For this reason, it may be helpful to start by outlining the organization of the paper, so as to keep track of the meaning of each step along the way. I assume the reader has some basic familiarity with the representation theory of compact groups (just a bit!) and a working notion of $p$-adic numbers. Some background in $p$-adic analysis would certainly be useful! 
\end{rem}

\subsection{Organization of the Paper}
We aim to establish multiplier theorems on compact $p$-adic Lie groups, using as our main example the compact Heisenberg group. To achieve this, we must proceed through several preparatory steps:

\begin{itemize}
    \item We need to understand well enough the representations of these groups, and compute (if possible!) the matrix coefficients. This information should also help us to understand tensor products of representations.
    \item With the information about the representation theory of the group we can talk about \textbf{the group Fourier transform}. This will be our main tool, because left-invariant operators (which are the ones we study here!) are precisely those that become diagonal when we consider the Fourier series representation of functions.     
    \item We need to introduce a class of operators worth considering. In our setting, the operators of interest are (arguably) analogues of differential operators, which we refer to as Vladimirov-type operators. As it is customary, multiplier theorems provide a tool for studying the inverses of these pseudo-differential operators, which is one of our motivations, specially for Vladimirov-type operators.
    \item As usual, our multiplier theorems are given in terms of some control of the symbol and its “derivatives”, which in discrete sets are simply difference operators. In our context, it will be useful to define such operators in a nice way, so that it is possible to apply to them interpolations arguments.  
    \item Along the way, we will have to estimate many norms (analysis after all!) and prove equivalences between some of them. Obviously!   
\end{itemize}

Keeping these goals in mind, the paper has been organized in the following order:

\begin{enumerate}
    \item For the rest of this first section, we’ll focus on introducing all the necessary ideas in order to formally state the results we’ve obtained, starting by motivating them with the results obtained in the abelian case. This is intended to provide a clear and comprehensive exposition for readers who are not familiar with the representation theory of compact groups, particularly in the $p$-adic setting. A reader well informed on the topic can simply jump to Section 2. Most of the notation about representation theory is taken from \cite{Ruzhansky2010}, which is also a very handy reference for the basics on compact Lie groups. We also borrow some notation from \cite{Onneweer1985, Onneweer1989}, which are very important references for multiplier theorems on abelian Vilenkin groups.  
    \item In Section 2, we state our main results.
    \item In Section 3, we recall all relevant definitions and terminology related to compact Vilenkin groups. We highlight the particularities of the representation theory in this setting, comparing it with the well-studied case of (real) compact Lie groups.
    \item In Section 4, we introduce equivalent norms for the Hilbert spaces $L^2_\alpha(\mathbb{G})$, $\alpha>0$, and use them to provide conditions on the symbol for the $L^2_\alpha$-boundedness of a Fourier multiplier.
    \item In Section 5, we deal with the more abstract case of compact Vilenkin groups, and the Banach spaces $L^r_\alpha (G)$. We extend the arguments of \cite{Onneweer1985} to the noncommutative setting, exploiting the parallels between the topological structures of abelian and non-abelian Vilenkin groups.
    \item In Section 6, we show how theorems on compact Vilenkin groups take a particular form when $G = \mathbb{G}$ is a compact $p$-adic Lie group, using our understanding of the unitary dual and the difference operators to adapt Taibleson’s and Saloff-Coste conditions \cite{Taibleson1970, SaloffCoste1986} to compact noncommutative $p$-adic Lie groups.
    \item In Section 7 we discuss some implications and possible generalizations of the results obtained.
\end{enumerate}

\subsection{Compact Vilenkin groups}
It was Taibleson who gave one of the first results about $L^r$-multipliers on local fields, in the spirit of the celebrated Hörmander multiplier theorem. In \cite{Taibleson1970}, he used the ideas of Calderón and Zygmund to give some multiplier theorems which imply analogues to Mihlin and Hörmander results in the Euclidean space. Subsequently, in \cite{Onneweer1985}, Onneweer took Taibleson's ideas to a higher level of generality, initiating a series of papers on multipliers and function spaces on abelian Vilenkin groups. See \cite{Kitada1987, Onneweer1985, Onneweer1989, Yueping2002} and the references therein. 

Vilenkin groups are totally disconnected groups, endowed with a strictly decreasing sequence of compact open subgroups that form a basis of neighborhoods at the identity. Such sequence naturally induces an ultrametric on the group and, conversely, for any metrizable totally disconnected group, there is a basis of metric balls which is also a strictly decreasing sequence of compact open subgroups. In other words, Vilenkin groups are just metrizable totally disconnected groups, like the additive group of a local field, that we choose to see as a group with a suitable sequence of subgroups, mostly because of the advantages of such an approach when doing analysis. Next, we recall for the reader the formal definition.

\begin{defi}\normalfont
We say that a topological group $G$ is a \emph{compact Vilenkin group} if $(G, \smallstar)$ is a profinite group endowed with a strictly decreasing sequence of compact open subgroups $\mathcal{G} := \{G_n\}_{n \in \mathbb{N}_0}$ such that
\begin{enumerate}
    \item It holds:
    \[
    2 \leq m_n := |G_n / G_{n+1}| < \infty,
    \]
    for every $n \in \mathbb{N}_0$.
    \item 
    \[
    G = \bigcup_{n \in \mathbb{N}_0} G_n, \quad \text{and} \quad \bigcap_{n \in \mathbb{N}_0} G_n = \{e\}.
    \]
    \item The sequence $\{G_n\}_{n \in \mathbb{N}_0}$ forms a basis of neighborhoods at $e \in G$.
\end{enumerate}

If we also have
\[
\sup_{n \in \mathbb{N}_0} |G_n / G_{n+1}| < \infty,
\]
we say that $G$ is a \emph{bounded-order Vilenkin group}. If $|G_n / G_{n+1}|$ is constant, we say that $G$ is a \emph{constant-order Vilenkin group}.    
\end{defi}

\begin{rem}\label{remdefimetric}
The above definition is linked to the following definition of ultrametric: for $x, y \in G$ their associated distance is defined as
\[
\varrho_{\mathcal{G}}(x,y) = |x \smallstar y^{-1}|_{\mathcal{G}} :=
\begin{cases}
0, & \text{if } x = y,\\[2mm]
|G_n|, & \text{if } \, \,   x \smallstar y^{-1} \in G_n \setminus G_{n+1}.
\end{cases}
\]
Here, for a measurable set $A \subset G$, the symbol $|A|$ denotes the measure of $A$. Notice how this distance function depends on the choice of sequence of compact open subgroups.

The study of Fourier analysis on abelian Vilenkin groups has been an active area of research in the last years, especially in the framework of Vilenkin systems. See \cite{Persson2022} for a recent survey on the topic. However, the literature on the noncommutative case is rather limited, even though the theory is still a very interesting one, especially because of its interplay with representation theory.
\end{rem}
Operators and functions on a compact Vilenkin group \(G\) can be analyzed using the \textbf{group Fourier transform}, a particularly powerful tool for studying \emph{translation--invariant operators}, which are defined through the (generally non-commutative) translations of the group. If an operator \(T\) is translation invariant, then it admits the Fourier--series representation
\[
Tf(x)= \sum_{[\pi] \in \widehat{G}} d_{\pi}\,\mathrm{Tr}\big[\,\pi(x)\,\sigma_T(\pi)\,\widehat{f}(\pi)\,\big],
\]
where the mapping \(\sigma_T\) is called the \emph{symbol} of the operator \(T\). 
In what follows, our main objective will be to study this symbol and to determine which conditions on \(\sigma_T\) ensure that \(T\) defines a bounded operator on weighted \(L^{r}(G)\)-spaces.

\begin{defi}
Let $G$ be a compact topological group. We will denote by $\mathrm{Rep}(G)$ the collection of all (equivalence classes of) continuous, unitary, finite-dimensional representations of $G$. The symbol $\widehat{G}$ will be used to denote the unitary dual of $G$, that is, the collection of all (equivalence classes of) continuous unitary irreducible representations of $G$. As it is well known, any element $[\pi]$ of $\mathrm{Rep}(G)$ can be written as a direct sum of finitely many elements of $\widehat{G}$.
\end{defi}

\begin{rem}\label{remnotationinequalities}
We will use the symbols $a \lesssim b$ and $a \gtrsim b$ to indicate that the quantity $a$ is, respectively, less or equal, or greater or equal than a constant times the quantity $b$ ($a \leq C b, \,\, C> 0$). We use the notation $a \asymp b$ to indicate that $a \lesssim b$ and $a \gtrsim b$ hold simultaneously.
\end{rem}

\subsection{The compact Heisenberg group}
Let $d \geq 1$. Probably, the simplest non-trivial example of a noncommutative profinite group is the Heisenberg group over the $p$-adic integers $\mathbb{H}_d(\Z_p)$, here defined as $\mathbb{Z}_p^{2d+1}$ endowed with the noncommutative operation
\[
\mathbf{x} \smallstar \mathbf{y} := (x_1 + y_1, x_2 + y_2, x_3 + y_3 + x_1 \cdot y_2), \quad x_1, x_2, y_1, y_2 \in \mathbb{Z}_p^d, \; x_3, y_3 \in \mathbb{Z}_p.
\]
This is actually a pro-$p$ group, and a compact $p$-adic Lie group, which we want to use as a basic model to understand more general classes of compact Vilenkin groups. It is a pro-finite group, since it is compact and totally disconnected, and it can be endowed with the special sequence of compact open subgroups $\mathcal{G} = \{\G_n\}_{n \in \mathbb{N}_0}$ defined as
\[
\G_n := \mathbb{H}_d(p^n \mathbb{Z}_p)= \{ \mathbf{x} \in \mathbb{H}_d \, : \, \| \mathbf{x} \|_p \leq p^{-n}  \} , \quad n \in \mathbb{N}_0.
\]

With this basis of neighbourhoods, $\mathbb{H}_d$ becomes a compact Vilenkin group, and its associated Vilenkin distance, see Remark \ref{remdefimetric}, satisfies
\[
\|\mathbf{x}\|_{p}^{2d+1} = |\mathbf{x}|_{\mathcal{G}}, \quad \text{for any } \mathbf{x} \in \mathbb{G}= \mathbb{H}_d.
\]
Since $\mathbb{H}_d$ is a compact group, it naturally carries a Fourier analysis through its representation theory, as long as we know enough about its unitary dual. Unfortunately, to the best of the author’s knowledge, there are no results in the literature that compute the matrix coefficients of representations of general $\Z_p$-Lie groups, the kind of information one would ideally hope for. What we do have instead are works such as \cite{Boyarchenko2008, Howe1977}, which study the characters and dimensions of these representations, sometimes even in the broader setting of profinite groups. The missing piece in \cite{Boyarchenko2008, Howe1977} is a general induction procedure to find explicit forms of the representations from the co-adjoint orbits, similar to the Kirillov orbit method described in \cite{Corwin2004-kz}.

To address (partially) this gap in the literature, the author has recently carried out explicit computations of the unitary dual of compact nilpotent $\Z_p$-Lie groups up to dimension five in \cite{velasquezrodriguez2024heisenberg, 2024arXiv240706289V, VelasquezRodriguez2025}, including the corresponding matrix coefficients of the representations. This detailed understanding enables a better analysis of certain linear operators defined on the group, similar to the approach taken by Rockland for the Heisenberg group over the real numbers \cite{Rockland1978}. In his work, Rockland makes significant contributions to the field of hypoelliptic operators within the context of noncommutative groups, by establishing representation-theoretic criteria that provide a robust framework for determining the hypoellipticity of operators defined on the Heisenberg group. This is particularly important because the Heisenberg group serves as a model for various phenomena in analysis and physics, especially in quantum mechanics and sub-Riemannian geometry.

Here in this work we want to use the Heisenberg group, and our understanding of its representation theory, to guide our study of Fourier multipliers on compact $p$-adic Lie groups, called here sometimes $\Z_p$-Lie groups, and more generally on compact Vilenkin groups. The idea is to extend several of the techniques developed for compact Lie groups in \cite{Ruzhansky2010} to the Heisenberg group, moving next to a deeper exploration on more general compact $p$-adic Lie groups. In the same way as for the real case, we will see how we can assure the boundedness of a given invariant operator by requiring certain control on the "derivatives" of the symbol, which are in this case difference operators \cite{Fischer2015, Fischer2020}. For instance, on (real) compact Lie groups, the authors of \cite{Ruzhansky2015} showed how the control on the Fourier variable can be done through the use of the RT-difference operators, which generalize the difference operators from the abelian setting, and are defined in terms of the matrix coefficients of some realization of the unitary irreducible representations of the group. See \cite{Fischer2015, Fischer2020}.

The study of compact \(p\)-adic Lie groups may be viewed as a complementary side of the mainstream theory on real groups, specially because, according to several structure theorems, an arbitrary compact group is always close to some kind of product between a locally connected group, often a Lie group, and a profinite group. See for instance \cite[Chapter 9]{Hofmann2020}. This is some kind of extension of the idea that $\Q$ can only be endowed with two different metric structures: the usual locally connected real one, corresponding to the archimedean absolute value $| \cdot |$, and the locally profinite one, corresponding to the $p$-adic absolute values $| \cdot |_p$. So, for the $p$-adic setting, we aim to show that, although some form of Ruzhansky--Turunen (RT) difference operators \cite{Fischer2015, Fischer2020, Ruzhansky2015} will still be required (something in common with the real case!) the control over the Fourier variable developed in \cite{Ruzhansky2015} must be replaced. Indeed, the appropriate substitute is a suitable generalization of the conditions introduced by Saloff-Coste in \cite{SaloffCoste1986} for abelian Vilenkin groups. This development uses tools somehow similar to the intrinsic calculus studied by Fischer in \cite{Fischer2015, Fischer2020}, but it will consider the subtle variations that we need to introduce in the totally disconnected case. For starters, we want whatever definition of control on the Fourier variable to include the behaviour of the Vladimirov–Taibleson operator, which is defined on general compact $\mathbb{K}$-Lie groups as:

\begin{defi}[Vladimirov–Taibleson operator]\label{defivladimirov}
Let $\mathbb{K}$ be a non-archimedean local field with ring of integers $\mathcal{O}_{\mathbb{K}}$, prime ideal $\mathfrak{p} = p \mathcal{O}_{\mathbb{K}}$ and residue field $\mathbb{F}_q \simeq \mathcal{O}_{\mathbb{K}}/p \mathcal{O}_{\mathbb{K}}$. Let $\mathbb{G} \le \mathrm{GL}_m(\mathcal{O}_{\mathbb{K}})$ be a compact $d$-dimensional $\mathbb{K}$-Lie group, and let $\mathcal{M}_m(\mathcal{O}_{\mathbb{K}})$ the Lie algebra of all $m \times m$ matrices over $\mathcal{O}_{\mathbb{K}}$. We define the Vladimirov–Taibleson operator on $\mathbb{G}$ via
\[
D^\alpha f(\mathbf{x}) := \frac{1 - q^\alpha}{1 - q^{-(\alpha+d)}} \int_{\mathbb{G}} \frac{f(\mathbf{x} \smallstar \mathbf{y}^{-1}) - f(\mathbf{x})}{\|\mathbf{y}\|_{\mathbb{K}}^{\alpha+d}} \, d\mathbf{y},
\]
where
\[
\|\mathbf{y}\|_{\mathbb{K}} :=
\begin{cases}
1, & \mathbf{y} \in \mathbb{G} \setminus I_m + \mathcal{M}_m(p\mathcal{O}_{\mathbb{K}}),\\
|\G / \G_1|^{\frac{-1}{d}}q^{-(n-1)}, & \mathbf{y} \in  \G \cap \big( I_m + \mathcal{M}_m(p^n \mathcal{O}_{\mathbb{K}}) \setminus I_m + \mathcal{M}_m(p^{n+1} \mathcal{O}_{\mathbb{K}}) \big).
\end{cases}
\]Here $d\mathbf{y}$ denotes the unique normalized Haar measure on $\mathbb{G}$. Sometimes it will be convenient to consider the operator
\[
\mathbb{D}^\alpha f(\mathbf{x}) := \frac{1 - q^{\alpha}}{1 - q^{-(\alpha+d)}} \Big( 1 - \frac{1}{|\G/\G_1|} - \frac{1-p^{-d}}{p^\alpha -1 } |\G / \G_1|^{\alpha/d} \Big) f(\mathbf{x}) + D^\alpha f (\mathbf{x}),
\]which in particular for nilpotent groups can be written as \[
\mathbb{D}^\alpha f(\mathbf{x}) := \frac{1 - q^{-d}}{1 - q^{-(\alpha+d)}} f(\mathbf{x}) + \frac{1 - q^\alpha}{1 - q^{-(\alpha+d)}} \int_{\mathbb{G}} \frac{f(\mathbf{x} \smallstar \mathbf{y}^{-1}) - f(\mathbf{x})}{\|\mathbf{y}\|_{\mathbb{K}}^{\alpha+d}} \, d\mathbf{y}.
\]
\end{defi}

\begin{rem}
    Compact $p$-adic Lie groups are always isomorphic to some closed subgroup of $\mathrm{GL}_m (\Z_p)$, so we will assume, without losing generality, that our groups are always matrix sub-groups of a certain $\mathrm{GL}_m (\Z_p)$, for an appropriate $m \in \N$.  
\end{rem}

The study of analysis on pro-finite groups seems to be quite underdeveloped, and there seems to be no literature available, except for a few attempts to establish some properties of the Fourier transform, but it always makes use of the characters of the representations, rather than the matrix coefficients (see e.g. \cite{Gt2006}). Here we take a different approach, following a path somewhat similar to Ruzhansky, Turunen, and Wirth in \cite{Ruzhansky2010, Ruzhansky2015}, but including the obvious modifications necessary to be compatible with the ultrametric structure, in a similar fashion as Saloff-Coste did in \cite{SaloffCoste1986}. Our idea is simple, even though it requires working out some technical details. A key argument will be the interpolation between weighted spaces
\[
L^r_\alpha (\mathbb{G}) := \Big\{ f : \mathbb{G} \to \mathbb{C} \ \Big| \ \int_{\mathbb{G}} |f(\mathbf{x})|^r \|\mathbf{x}\|_p^\alpha \, d\mathbf{x} < \infty \Big\}, \quad 1 \le r < \infty, \ \alpha \in \mathbb{R},
\]
and the weak weighted spaces
\[
L^{r,w}_{\alpha}(\mathbb{G}) := \Big\{ f : \mathbb{G} \to \mathbb{C} \ \Big| \ \sup_{t>0} t \big(\mu_\alpha\{ \mathbf{x} \in \mathbb{G} : |f(\mathbf{x})| > t \}\big)^{1/r} < \infty \Big\}, \quad 1 \le r < \infty, \ \alpha \in \mathbb{R},
\]
where $$ \mu_\alpha (A):= \int_A \| \mathbf{x} \|_p^\alpha d \mathbf{x} .$$

\subsection{Dual of a compact $p$-adic Lie group} 
Let $\G$ be a compact $p$-adic Lie group. A \emph{\textbf{Fourier multiplier on}} $\G$ is a lineal operator $T_\sigma$ with the form \[
T_\sigma f(\mathbf{x})= \sum_{[\pi] \in \widehat{\G}} d_{\pi}\,\mathrm{Tr}\big[\,\pi(\mathbf{x})\,\sigma(\pi)\,\widehat{f}(\pi)\,\big],
\]where the mapping $$\sigma : \G \times \mathrm{Rep}(\G) \to \bigcup_{[\pi] \in \mathrm{Rep}(\G)} \C^{d_\pi \times d_\pi},$$is called the symbol of the operator. An equivalent definition would be: a Fourier multiplier is a left-invariant linear operator on $\G$. In any case, to understand multipliers on a compact group it is convenient, perhaps necessary, to have some understanding of the representation theory of the group in question. In general this can be a very difficult endeavour, but there are some classes of $p$-adic groups where there is enough information, like for compact nilpotent $p$-adic Lie groups.

\begin{rem}\label{remnotationdualofZp}
We identify each equivalence class $\lambda \in \widehat{\mathbb{Z}}_p \cong \mathbb{Q}_p / \mathbb{Z}_p$ with its representative in the complete system
\[
\{1\} \cup \Big\{ \sum_{k=1}^{\infty} \lambda_k p^{-k} : \text{only finitely many } \lambda_k \neq 0 \Big\}.
\]
Similarly, elements of the quotients $\mathbb{Q}_p / p^{-n} \mathbb{Z}_p$ are chosen from
\[
\{1\} \cup \Big\{ \sum_{k=n+1}^{\infty} \lambda_k p^{-k} : \text{only finitely many } \lambda_k \neq 0 \Big\}.
\]
\end{rem}

\begin{rem}
    $\widehat{\Z}_p^d$ is also a tree. To see it, identify any $\vec{\lambda} \in \widehat{\Z}_p^d$ with \[
\{(1,...,1)\} \cup \Big\{ \sum_{k=1}^{\infty} \vec{\lambda}_k p^{-k} : \text{only for finitely many } \, k, \, \,  (0,...,0) \neq  \vec{\lambda}_k \in \{0,1,...,p-1 \}^d  \Big\}.
\]Define $P_l (\vec{\lambda}) := \sum_{k=1}^l \vec{\lambda}_k p^{-k}$. We say that $\xi$ is a descendant of $\eta$ if $P_l(\xi) = \eta$ for some $l$. And to draw an edge between two $\xi, \eta \in \widehat{\Z}_p^d$ we follow the next instructions: 
\begin{enumerate}
    \item The root of the tree is $\vec{0}=(0,...,0)$. All the elements of the form $ \vec{\lambda}_1 p^{-1}$ are connected to $\vec{0}$. This forms level one. 
    \item In level two we have the elements $\vec{\lambda}$ with the form $\vec{\lambda}_1 p^{-1} + \vec{\lambda}_2 p^{-2}.$ We will draw an edge between one of this, and any of the previous edges $\eta$ if and only if $P_1( \vec{\lambda}) = \eta$. 
    \item Following the same idea, once we have the points and edges up to level 2, we consider all the elements $\vec{\lambda}_1 p^{-1} + \vec{\lambda}_2 p^{-2}+ \vec{\lambda}_3 p^{-3}$, $\vec{\lambda}_3 \neq 0$. One of these connects with an edge $\eta$ of a previous level if and only if $P_2 (\vec{\lambda}) = \eta.$
    \item In general, if $\|  \xi \|_p = p^n$, then $\xi$ is connected with an edge to $\eta$ in a previous level if and only if $P_{n-1} (\xi) = \eta .$
\end{enumerate}
\end{rem}

The information we would like to have (ideally!) about the unitary dual is some sort of parametrization for it, ideally by some subset of $\widehat{\Z}_p^d$, like in the case we describe next. The following ideas are developed in detail in \cite{Boyarchenko2008} and \cite{Howe1977}. 

Let $\mathfrak{g}$ be a nilpotent $\Z_p$-Lie algebra, and let $\G = \exp(\mathfrak{g})$ denote the associated compact nilpotent $p$-adic Lie group. For simplicity we will consider $\g$ as a sub-algebra of $\mathcal{M}_{m} (\Z_p)$, for an appropriate $m \in \N$. Throughout, we assume that $p$ is greater than the nilpotency class of $\g$. For such groups one can carry out the following construction:

\begin{enumerate}
    \item Assume $\g \cong \Z_p^d$, so that $\g$ and $\G$ have dimension $d$. 
    \item Consider the dual of $\g$: $$\g^*:=\{ \chi_\ell (X)=e^{2 \pi i \{ \ell \cdot X  \}_p } \, : \, \ell \in \widehat{\Z}_p^d \cong \Q_p^d / \Z_p^d   \}.$$
    \item The \textbf{co-adjoint representation} in this case has the form $$ \langle \mathrm{Ad}^*_{\mathbf{g}} \ell , X \rangle := \langle  \ell , \mathrm{Ad}_{\mathbf{g}^{-1}} X \rangle = e^{2 \pi i \{ \ell \cdot \mathrm{Ad}_{\mathbf{g}^{-1}}  X  \}_p },$$where the \textbf{adjoint representation} is, as usual, defined as the map $$\mathrm{Ad}: \G \to \mathrm{Aut}(\g), \, \, \, \mathbf{g} \mapsto \Psi_\mathbf{g} (X) = \mathbf{g} X \mathbf{g}^{-1}.$$
    \item Denote by $\mathcal{O}_\ell$ the \textbf{orbit of the co-adjoint representation} associated to $\ell$. A \textbf{polarization} of $\g$ at $\ell$ is a Lie subalgebra $\mathfrak{p} \subset \g$ which has the property that $\ell$ is trivial on $[ \mathfrak{p}, \mathfrak{p} ]$ ( $e^{2 \pi i \{ \ell \cdot X \}_p} =1$, for all $X \in \mathfrak{p}$)  and which is maximal among all additive subgroups of $\g$ with this property.
    \item According to \cite[ Lemma 1.4]{Howe1977}, a polarizations always
    exist, and if $P:= \mathbf{exp}(\mathfrak{p})$ is the subgroup of $\G$ corresponding to a polarization $\mathfrak{p}$, then $\ell$ induces
    a 1-dimensional character $\chi_\ell$ of $P$, and we can form the induced representation $\pi_\ell = \mathrm{Ind}_\G (P, \chi_\ell)$. 
    \item It was proven in \cite[Lemma 1.2]{Howe1977}, and later in \cite{Boyarchenko2008} for more general profinite groups, how the resulting representation is always irreducible; its isomorphism
    class depends only on the $\G$-orbit of $\ell$ ; and, finally, every $[\pi] \in \widehat{\G}$ arises in this way from a unique $\G$-orbit $\mathcal{O}_\ell \subset \g^*$.
    \item According to \cite[Lemma 1.2]{Howe1977}, the \textbf{character of the representation} associated to the $\G$-orbit $\mathcal{O}_\ell$ has the formula $$\chi_{\pi_\ell} (\mathbf{exp}(X)) = |\mathcal{O}_\ell|^{-1/2} \sum_{\xi \in \mathcal{O}_\ell} e^{2 \pi i \{ \xi \cdot X \}_p}.$$  
    \item In conclusion \textbf{the unitary dual of a compact $p$-adic Lie group $\G$ can be identified with a certain special subset of} $\widehat{\Z}_p^d \cong \Q_p^d / \Z_p^d$.
\end{enumerate} 
In particular, the Kirillov-Howe formula \cite{Howe1977} gives us for $\mathbb{H}_d$:
\[
\chi_{\pi_\xi}(\mathbf{x}) = |\xi_3|_p^d e^{2\pi i \{ \xi_1 \cdot x_1 + \xi_2 \cdot x_2 + \xi_3 x_3 \}_p} \mathbb{1}_{p^{-\vartheta(\xi_3)} \mathbb{Z}_p^{2d}}(x_1, x_2)
.\]Here the function $\mathbb{1}_A$ denotes the indicator function of the measurable set $A$, and $\vartheta(\xi)$ is the usual $p$-adic valuation. 9
When the exponential map of our Lie group is not an homeomorphism, the arguments in \cite{Boyarchenko2008,Howe1977} do not work any more, and we cannot get our irreducible representations from co-adjoint orbits like in \cite{Corwin2004-kz}. However, for any $\G \leq \mathrm{GL}_m (\Z_p)$ and any $[\pi] \in \widehat{\G}$, we can define the following ultrametric: 

$$|\pi|_{\widehat{\G}}:= |\G_{n_{[\pi]}}|^{-1}, \, \| \pi \|_p:=|\pi|_{\widehat{\G}}^{1/d}, \, \text{where} \, n_{[\pi]}:= \min\{ n \in \N_0 \, : \, \pi(\mathbf{x}) = I_{d_\pi}  \}.$$

The key takeaway here is that $\widehat{\G}$ and $\mathbf{Rep}(\G)$ can be endowed with a certain ultrametric, meaning that on the abelian monoid $(\mathrm{Rep}(\G), \otimes)$ it holds $$| [\pi] \otimes [\pi'] |_{\widehat{\G}} \leq  \max \{| [\pi] |_{\widehat{\G}}, \, | [\pi'] |_{\widehat{\G}}\}.$$ In particular, when we $\G$ is nilpotent or the exponential map is an homeomorphism, the arguments in \cite{Boyarchenko2008} let us conclude that $\widehat{\G}$ can be identified with a certain subset of $\widehat{\Z}_p^d \cong \Q_p^d/\Z_p^d$, and the previous definition of $\| \cdot \|_p$ coincides with the usual $p$-adic norm on $\Q_p^d/\Z_p^d$. This means that $L^2$-functions on a compact $p$-adic Lie group $\G$ can be written in its Fourier series representation 
\begin{equation}
f(\mathbf{x}) := \sum_{n \in \N_0 } \sum_{ [ \pi] \in \widehat{\G}, \, \| \pi \|_p = p^n } d_\pi \, \mathrm{Tr}[\pi(\mathbf{x}) \widehat{f}(\pi)], \quad f \in L^2(\mathbb{G}),
\end{equation}
where $\widehat{f}$ denotes the group Fourier transform of $f$, defined for $\mathbb{G}$ as
$$\mathcal{F}_{\mathbb{G}}[f](\pi) := \int_{\mathbb{G}} f(\mathbf{x}) \pi^*(\mathbf{x})  \, d\mathbf{x}
.$$In particular, for the Heisenberg group, the Fourier series representation takes the form 

\begin{equation}
f(\mathbf{x}) := \sum_{\xi_3 \in \widehat{\mathbb{Z}_p}} \sum_{(\xi_1, \xi_2) \in \mathbb{Q}_p^{2d} / p^{\vartheta(\xi_3)} \mathbb{Z}_p^{2d}} |\xi_3|_p^d \, \mathrm{Tr}[\pi_\xi(\mathbf{x}) \widehat{f}(\xi)], \quad f \in L^2(\mathbb{G}).
\end{equation}
See \cite{velasquezrodriguez2024heisenberg} for the complete calculations.

\subsection{Equivalent norms on $L^2_\alpha (\mathbb{G}) $}
As we said before, in this paper we use mainly two arguments:

\begin{itemize}
    \item[(i)] A multiplier theorem between $L^2_\alpha$-spaces,
    \item[(ii)] An interpolation argument between $L^{1,w}_\alpha (\G)$ and $L^2_\alpha(\G)$. 
\end{itemize}

For the treatment of the endpoint spaces $L^2_\alpha(\mathbb{G})$ in our interpolation argument, we incorporate Saloff-Coste difference operators to the unitary dual, by means of the following equivalence of norms:
\begin{equation}\label{eq1}
\boxed{\| f \|_{L^2_\alpha(\mathbb{G})}^2 \asymp \sum_{[\xi] \in \widehat{\mathbb{G}}} \sum_{[\eta] \in \widehat{\mathbb{G}}} d_\xi d_\eta  \, \| \eta \|_p^{-(\alpha + d)} \| \Delta_\eta \widehat{f}(\xi) \|_{HS}^2,}
\end{equation}
which is a simple consequence of the following equivalence \begin{equation}
    \boxed{ \|\mathbf{x}\|_{p}^\alpha \asymp I_\alpha(x) := 
    \sum_{[\eta] \in \widehat{\mathbb{G}}} 
    d_\eta \, \|\eta(\mathbf{x}) - I_{d_\eta} \|^2_{HS} \, 
    \|\eta\|_p^{-(\alpha+d)}.}
\end{equation}
The idea actually comes from Taibleson and Saloff-Coste, who worked out the details on local fields and abelian Vilenkin groups \cite{SaloffCoste1986}. In particular, for the Heisenberg group, the equivalence of norms takes the following form:

\begin{equation}\label{eq1Heis}
\boxed{\| f \|_{L^2_\alpha(\mathbb{H}_d)}^2 \asymp \sum_{[\xi] \in \widehat{\mathbb{H}}_d} \sum_{[\eta] \in \widehat{\mathbb{H}}_d} |\xi_3|_p^d |\eta_3|_p^d \, \| \eta \|_p^{-(\alpha + (2d + 1))} \| \Delta_\eta \widehat{f}(\xi) \|_{HS}^2.}
\end{equation}

The above norm uses the fact that the unitary dual of $\mathbb{H}_d$, as it was proven in \cite{velasquezrodriguez2024heisenberg}, can be identified with the following sub-set of $\widehat{\mathbb{Z}}_p^{2d+1}$:
\[
\widehat{\mathbb{H}}_d := \{ \xi \in \widehat{\mathbb{Z}}_p^{2d+1} : (\xi_1, \xi_2) \in \mathbb{Q}_p^{2d} / p^{\vartheta(\xi_3)} \mathbb{Z}_p^{2d} \},
\]
where for any equivalence class $[\pi] \in \widehat{\mathbb{H}}_d$, we choose a representative $\pi_\xi$ defined on the Hilbert space
\[
\mathcal{H}_\xi := \mathrm{span}_\mathbb{C} \Big\{ |\xi_3|_p^{d/2} \mathbb{1}_{h + p^{-\vartheta(\xi_3)} \mathbb{Z}_p^d} : h \in \mathbb{Z}_p^d / p^{-\vartheta(\xi_3)} \mathbb{Z}_p^d \Big\} \subset L^2(\mathbb{Z}_p^d),
\]
via the formula
\[
(\pi_\xi(\mathbf{x}) \varphi)(u) := e^{2\pi i \{ \xi_1 \cdot x_1 + \xi_2 \cdot x_2 + \xi_3(x_3 + x_2 \cdot u) \}_p} \varphi(u + x_1).
\]Here, the $p$-adic valuation of $\xi_3 \in \mathbb{Q}_p$ is denoted by $\vartheta(\xi_3)$.

\subsection{Difference operators on $\mathbb{G}$}

In this paper, we are concerned with the study of left-invariant linear operators on $\mathbb{G}$, 
which are precisely those operators that can be written as block-diagonal operators using Fourier series. The \emph{symbol} of the operator can be written in terms of a given realization $\pi$ of $[\pi]$ as
\[
\sigma_T(\pi) = \pi^*(\mathbf{x}) T \pi(\mathbf{x}),
\]
and it can be used to establish conditions that guarantee the boundedness of the operators between 
certain function spaces. For instance, for the Hilbert space $L^2(\mathbb{G})$, the condition $T \in \mathcal{L}(L^2(\mathbb{G}))$ (meaning $T$ is a bounded operator on $L^2(\mathbb{G})$) is equivalent to 
the symbol condition
\[ \sigma \in L^\infty_{op} (\widehat{\G}) := \{ \sigma \, : \,  \sup_{[\xi] \in \widehat{\mathbb{G}}} \|\sigma(\xi)\|_{\text{op}} < \infty, 
\quad 
\|\sigma(\xi)\|_{\text{op}} := \sup\{\|\sigma(\xi)v\|_{\mathcal{H}_\xi} : \|v\|_{\mathcal{H}_\xi}=1\}\} 
.
\]

However, when $r \neq 2$, there is no obvious way to guarantee boundedness in terms of the symbol alone, and it remains an open problem to find necessary and sufficient conditions, even in simple cases. 
A typical first attempt is to control the behaviour of the symbol using its derivatives or certain 
\emph{difference operators}, analogous to the classical Mikhlin multiplier theorem. For locally connected compact Lie groups the reader can find in \cite{Ruzhansky2010, Ruzhansky2015} the right definition of difference operators, 
and in \cite{Ruzhansky2015} the corresponding Mikhlin-type theorem. See also \cite{Fischer2015, Fischer2020}.

\begin{defi}[Difference operator]\label{def:diff_op}\normalfont
Let $G$ be a compact group. The difference operator $\Delta_\eta$, $[\eta] \in \widehat{G}$, 
acts on symbols via
\[
\Delta_\eta \sigma(\xi) := \sigma(\eta \otimes \xi) - \sigma(\mathrm{I}_{d_\eta} \otimes \xi),
\]where $\mathrm{I}_{d_\eta}$ denotes the identity representation of dimension $d_\eta$.
\end{defi}

\begin{rem}
Let $G$ be a compact group. Any continuous unitary representation $[\pi] \in \mathrm{Rep}(G)$ can be written as a direct sum of irreducible 
representations:
\[
[\pi] \cong \bigoplus_{[\xi] \in \widehat{G}} m_{[\xi],[\pi]} [\xi],
\]
so symbols initially defined on the unitary dual $\widehat{G}$ can always be extended to the whole $\mathrm{Rep}(G)$, and any mapping 
on $\mathrm{Rep}(G)$ restricts to a mapping on $\widehat{G}$.
\end{rem}

There is a natural connection between Definition~\ref{def:diff_op} and the so-called RT difference operators. 
Identifying our finite-dimensional representations with their matrix realizations, and their tensor products with Kronecker products, 
we may think of $[\eta] \otimes [\xi] \cong [\eta \otimes \xi]$ as a block matrix of size $d_\eta \times d_\xi$ with block entries $\eta_{ij}(x)\xi(x)$. 
Hence, for $f \in L^2(G)$, the difference operator acts as
\[
\Delta_\eta\widehat{ f}(\xi) = \widehat{f}(\eta \otimes \xi) - \widehat{f}(\mathrm{I}_{d_\eta} \otimes \xi)
= \int_G f(x) \big( (\eta(x) - \mathrm{I}_{d_\eta}) \otimes \xi(x) \big)^* \, dx,
\]
and the $(i,j)$-th block-entry is
\[
[\Delta_\eta \widehat{f}(\xi)]_{ij} := \int_G f(x) (\overline{\eta_{ji}}(x) - \delta_{ij}) \xi^*(x),
\]
which coincides with the RT difference operator associated with the function $q^\eta_{ij}(x) := \overline{\eta_{ji}}(x) - \delta_{ij}$.

\begin{defi}[RT difference operator]
Let $\G$ be a compact Lie group and $q \in C^\infty(\G)$. The RT difference operator acting on mappings 
defined on the unitary dual is
\[
\Delta_q \sigma(\xi) := \mathcal{F}_\G \Big[ q \, \mathcal{F}_\G^{-1}[\sigma] \Big](\xi)
= \int_\G q(x) \mathcal{F}_\G^{-1}[\sigma](x) \, dx.
\]
\end{defi}

The RT difference operators can be used to prove an analogue of Mikhlin condition on compact Lie
groups. Here we are interested in finding its right analogue for a compact $p$-adic Lie group $\mathbb{G}$, starting with our first example $\mathbb{H}_d$, which will serve us to explore the behaviour of the symbol of interesting
operators on a noncommutative setting. In light of this, considering how the tensor product of representations will be frequently used, it is convenient to recall how for the Heisenberg group it was proven in
\cite{velasquezrodriguez2024heisenberg} that the character of the tensor product of representations is
\[
\chi_{\pi_\eta \otimes \pi_\xi}(\mathbf{x}) = \chi_{\pi_\eta}(\mathbf{x}) \chi_{\pi_\xi}(\mathbf{x}) 
= |\eta_3|_p^d |\xi_3|_p^d \, e^{2\pi i \{\mathbf{x} \cdot (\xi + \eta)\}_p} \, \mathbb{1}_{\max\{|\xi_3|_p, |\eta_3|_p\}\mathbb{Z}_p^{2d}}(x_1,x_2),
\]
which provides a convenient way to compute the symbol of tensor products. To see it, consider the set $$\widehat{\mathbb{H}}_d^{\N, \, finite} := \{ f: \N \to \widehat{\mathbb{H}}_d \subseteq \widehat{\Z}_p^{2d+1} \, : \, f(n)=0 \,\, \text{for large enough} \, n  \},$$which is the set of finite sequences on $\widehat{\mathbb{H}}_d$. Define $S_\infty$ as the union of all symmetric groups $S_n$. Then $$\Gamma(\mathbb{H}_d):= \widehat{\mathbb{H}}_d^{\N, \, finite}/ S_\infty,$$is the collection of all finite unordered sequences of elements of $\widehat{\mathbb{H}}_d$, which is isomorphic to $(\mathrm{Rep}(\mathbb{H}_d), \otimes)$ if we endow it with the following operation $\boxtimes$.   

\begin{defi}[Congruence and $\boxtimes$ operation]\normalfont
\,
\begin{enumerate}
    \item Let $\gamma, \zeta \in \widehat{\mathbb{Z}}_p^d$. We say that $\gamma$ is congruent to $\zeta$ modulo $p^n$, and we write $\gamma \equiv \zeta \mod p^n$, if
    \[
    \gamma - \zeta \in  p^{-n} \mathbb{Z}_p^d.
    \]
    In other words, $\gamma, \zeta \in \widehat{\mathbb{Z}}_p^d$ are congruent if they have the same digits in their $p$-adic expansion, except maybe for the first $n$ digits. This defines, for every fixed $n$, an equivalence relation in $\widehat{\mathbb{Z}}_p^d$, and we denote the class of $\gamma$ in $\Q_p^d / p^{-n}\Z_p^d$ by $[\gamma]_n$.
    \item For $\xi, \eta \in \widehat{\mathbb{Z}}_p^{2d+1}$, let $N(\eta,\xi) \in \mathbb{N}_0$ be the natural number such that
    \[
    p^{N(\eta,\xi)} = \max \{ |\eta_3|_p, |\xi_3|_p \}.
    \]
    We define the function
    \[
    \boxtimes : \widehat{\mathbb{H}}_d \times \widehat{\mathbb{H}}_d \to \Gamma(\widehat{\mathbb{H}}_d)
    \]
    through the formula $\xi \boxtimes \eta :=$
    \[
    \begin{cases} \big([ \xi_1 + \eta_1 ]_{N(\eta,\xi)}, [ \xi_2 + \eta_2 ]_{N(\eta,\xi)}, \xi_3 + \eta_3 \big), & \text{if } \, |\xi_3 + \eta_3|_p = \max \{ |\xi_3|_p, |\eta_3|_p \}  \, \,, \\
\bigoplus_{\gamma \in p^{-N(\eta, \xi)} \mathbb{Z}_p^{2d} / p^{\vartheta(\xi_3+\eta_3)} \mathbb{Z}_p^{2d}} \big([ \xi_1 + \eta_1 + \gamma_1 ]_{N(\eta,\xi)}, [ \xi_2 + \eta_2 + \gamma_2 ]_{N(\eta,\xi)}, \xi_3 + \eta_3 \big) , & \text{if } \, |\xi_3 + \eta_3|_p<\max \{ |\xi_3|_p, |\eta_3|_p \} .
\end{cases} 
    \]
\end{enumerate}
\end{defi}

As we will show later, the tensor products of representations decompose as
\[
\pi_\eta \otimes \pi_\xi \cong  \bigoplus_{j=1}^{|\xi_3 + \eta_3|_p^d} \bigoplus_{\gamma \in p^{-N(\eta, \xi)} \mathbb{Z}_p^{2d} / p^{\vartheta(\xi_3+\eta_3)} \mathbb{Z}_p^{2d}} \pi_{[\eta + \xi + (\gamma_1, \gamma_2, 1)]},
\quad \dim_\mathbb{C}(\mathcal{H}_{[\eta + \xi + (\gamma_1, \gamma_2, 1)]}) = |\xi_3 + \eta_3|_p^d,
\]
which implies the following formula for symbols on $\mathbb{G}$:
\[
\Delta_\eta \sigma(\xi) =
\begin{cases}
\displaystyle
\bigoplus_{j=1}^{|\xi_3 + \eta_3|_p^d} \bigoplus_{\gamma \in p^{-N(\eta, \xi)} \mathbb{Z}_p^{2d} / p^{\vartheta(\xi_3+\eta_3)} \mathbb{Z}_p^{2d}} \sigma([\eta + \xi + (\gamma_1, \gamma_2, 1)]) - \mathrm{I}_{d_\eta} \otimes \sigma(\xi), & |\xi_3 + \eta_3|_p < \|(\xi_3, \eta_3)\|_p, \\
\displaystyle
\bigoplus_{j=1}^{|\xi_3|_p^d |\eta_3|_p^d / |\xi_3+\eta_3|_p^d} \sigma(\eta \boxtimes \xi) - \mathrm{I}_{d_\eta} \otimes \sigma(\xi), & |\xi_3 + \eta_3|_p = \|(\xi_3, \eta_3)\|_p.
\end{cases}
\]

In particular, if $|\xi_3 + \eta_3|_p = \|(\xi_3, \eta_3)\|_p$:
\[
\Delta_\eta \sigma(\xi) :=
\begin{cases}
\mathrm{I}_{d_\eta} \otimes \big(\sigma(\xi \boxtimes \eta) - \sigma(\xi)\big), & |\eta_3|_p \le |\xi_3|_p, \\
\sigma(\xi \boxtimes \eta) \otimes \mathrm{I}_{d_\xi} - \mathrm{I}_{d_\eta} \otimes \sigma(\xi), & |\eta_3|_p > |\xi_3|_p,
\end{cases}
\]
where $\eta = (\eta_1, \eta_2, \eta_3), \xi = (\xi_1, \xi_2, \xi_3) \in \widehat{\mathbb{H}}_d \subset \widehat{\mathbb{Z}}_p^{2d+1}$.

\subsection{The symbol of the Vladimirov-Taibleson operator}

In order to understand the kind of difference operator that we need on $\mathbb{G}$, we have to consider at least one example of an interesting Fourier multiplier. The most natural one is the Vladimirov-Taibleson operator, which we introduced before in Definition \ref{defivladimirov}. The symbol of this operator can be easily calculated, and we have for $\alpha \in \mathbb{R}$:
\[
\sigma_{\mathbb{D}^\alpha}(\xi) =
\begin{cases}
\dfrac{1-p^{-(2d+1)}}{1-p^{-(\alpha+(2d+1))}}, & \text{if } \|\xi\|_p = 1, \\[1.2em]
\|\xi\|_p^\alpha \, \mathrm{I}_{d_\xi}, & \text{if } \|\xi\|_p > 1.
\end{cases}
\]

For any compact Vilenkin group one can notice how the symbol of the VT operator posses some sort of local constancy property. For instance, on compact $p$-adic Lie groups we can check how for $\|\eta\|_p < \|\xi\|_p$ it holds
\[
\sigma_{\mathbb{D}^\alpha} (\eta \otimes \xi) = \sigma_{\mathbb{D}^\alpha} ( \mathrm{I}_{d_\eta} \otimes \xi),
\]
so that
\[
\Delta_\eta \sigma_{\mathbb{D}^\alpha} (\xi) = 0,
\]
and this behaviour of the VT operator is not special to $\mathbb{G}$, but the same occurs in many other classes of compact Vilenkin groups. A generalization of this local constancy property would be rapid decaying of the difference operator for $\|\eta\|_p < \|\xi\|_p$, that is
\[
\|\Delta^\beta_\eta \sigma(\xi)\|_{\mathrm{op}} := \frac{1}{\|\eta\|_p^\beta}
\|\Delta_\eta \sigma(\xi)\|_{\mathrm{op}} \leq C \|\xi\|_p^{-\beta}, \quad \text{for some } C > 0, \ \text{and any } \beta \geq 0,
\]
which is actually a similar condition to the one given by Saloff-Coste in \cite{SaloffCoste1986}. Notice how the condition in \cite{SaloffCoste1986} is basically the same condition Taibleson gave for local fields in \cite{Taibleson1970}, but in the more general setting of abelian Vilenkin groups. One of the main results we will prove here is a generalization of Theorem \ref{teosaloff} to the noncommutative
group $\mathbb{G}$, which provides a $p$-adic analogue of the Mikhlin condition provided in \cite{Ruzhansky2015} for compact Lie groups. However, this will come as a more general result about interpolation between weighted
spaces on compact Vilenkin groups so, in order to introduce better what comes next, we want to
emphasise in our first main underlying argument which will be the alternative choice of norm for $L^2_\alpha(\mathbb{G})$, proven in Lemma \ref{lemma3.1}. This time, unlike for the abelian case, we actually have more than one definition of difference operator. The first kind was given in Definition \ref{def:diff_op}, and now we have the following second definition:

\begin{defi}[RT-difference operators]\label{RT-difference operators}\normalfont
Given $\xi \in \widehat{\mathbb{G}}$, $\eta \in \widehat{\mathbb{Z}}_p^{2d+1}$, and $\beta \ge 0$, the \emph{RT-difference operator} $\triangletimes_\eta^\beta$ is defined as
\[
\triangletimes_\eta^\beta \sigma(\xi) := \frac{1}{\|\eta\|_p^\beta} \mathcal{F}_G \Big[ \big( e^{2 \pi i \{\eta \cdot x\}_p} - 1 \big) \mathcal{F}_G^{-1}[\sigma] \Big](\xi).
\]
We write $\triangletimes_{\eta} := \triangletimes^0_{\eta} \neq \triangletimes^1_{\eta}$.
\end{defi}

\begin{rem}
    In the above definition the abbreviature RT stand for Ruzhasnky-Turunen, who defined together with J. Wirth the following difference operators: $$\triangletimes_q \sigma (\xi) := \mathcal{F}_G \Big[ q (\cdot) \mathcal{F}_G^{-1}[\sigma] \Big](\xi) ,$$for a suitable differentiable function $q$. Here we use the same definition for the particular case of the functions $$q_\eta^\beta (\mathbf{x}) := \frac{e^{2 \pi i \{\eta \cdot x\}_p} - 1}{ \| \eta \|_p^\beta}.$$
\end{rem}

An advantage of these operators is that they do not change the dimension of the matrices, and for $f \in L^2(\mathbb{G})$, we have
\begin{align*}
    \hat{f}(\xi)_{hh'} &:= \int_{\mathbb{Z}_p^{2d+1}} f(\mathbf{x}) \mathbb{1}_{h-h'+p^{-\vartheta(\xi_3)} \mathbb{Z}_p^d}(x_1) e^{2 \pi i \{ \xi \cdot \mathbf{x} + (1, \xi_3 h, 1) \cdot \mathbf{x} \}_p} dx \\&= \mathcal{F}_{\mathbb{Z}_p^{2d+1}}[\mathbb{1}_{h-h'+p^{-\vartheta(\xi_3)} \mathbb{Z}_p^d}(x_1)] \ast_{\widehat{\mathbb{Z}}_p^{2d+1}}  \, \mathcal{F}_{\mathbb{Z}_p^{2d+1}}[f](\xi + (1, \xi_3 h,1)),
\end{align*}
where $\ast_{\widehat{\mathbb{Z}}_p^{2d+1}} $ denotes the discrete convolution on $\widehat{\mathbb{Z}}_p^{2d+1}$.
\[
\mathcal{F}_{\mathbb{Z}_p^{2d+1}}[f] \ast_{\hat{\mathbb{Z}}_p^{2d+1}} \mathcal{F}_{\mathbb{Z}_p^{2d+1}}[g](\xi) := \sum_{\gamma \in \hat{\mathbb{Z}}_p^{2d+1}} \mathcal{F}_{\mathbb{Z}_p^{2d+1}}[f](\xi - \gamma)\, \mathcal{F}_{\mathbb{Z}_p^{2d+1}}[g](\gamma).
\]

In this way, the RT difference operators can be written as
\[
\triangletimes^\beta_{\eta} \widehat{f}(\xi)_{hh'} = \mathcal{F}_{\mathbb{Z}_p^{2d+1}}\Big[\mathbb{1}_{h-h' + p^{-\vartheta(\xi^3)}\mathbb{Z}_p^d}(x_1)\Big] \ast_{\hat{\mathbb{Z}}_p^{2d+1}} \triangle^\beta_{\eta} \mathcal{F}_{\mathbb{Z}_p^{2d+1}}[f](\xi + (1, \xi_3 h, 1)),
\]

where $\triangle$ is the usual difference operator
\[
\triangle^\beta_{\eta} \mathcal{F}_{\mathbb{Z}_p^{2d+1}}[f](\xi) := \frac{\mathcal{F}_{\mathbb{Z}_p^{2d+1}}[f](\xi + \eta) - \mathcal{F}_{\mathbb{Z}_p^{2d+1}}[f](\xi)}{\|\eta\|_p^\beta},
\]

and hence it becomes easy to check that for the VT operator it holds:
\[
\triangletimes_{\eta} \sigma_{\mathbb{D}^\alpha}(\xi) = 0, \quad \text{for } \|\eta\|_p < \|\xi\|_p,
\]

which is a local-constancy property for $\sigma_{ \mathbb{D}^\alpha}$. See the final remarks for more details.

\section{Main results}
In this work, we intend to contribute to the literature by studying noncommutative Vilenkin groups, and dealing with the extensions of several ideas explored before in the narrower class of compact (abelian) Vilenkin groups. Abelian and noncommutative Vilenkin groups are very close in terms of their topological structure, and there are many instances where one can exploit this similarity. For instance, the Calderón-Zygmund decomposition, a very useful tool in analysis, works the same in both cases without major modifications, because it only uses the associated sequence of compact-open subgroups. Nonetheless, there are many situations where the noncommutativity of the group plays an important role, like when talking about representation theory and the difference operators. The representations of abelian groups are one-dimensional and their Pontryagin dual is also a metrizable group, which simplifies many things, but the representations of noncommutative groups are in general operators acting on some Hilbert space. We will focus here on compact groups where the representations are finite-dimensional and therefore can be realized as matrices \cite{Fischer2015, Fischer2020}. In this setting, Fourier multipliers on the group, i.e., linear invariant operators, can be expressed in terms of a matrix-valued map which is commonly called the \emph{symbol} of the operator \cite{Ruzhansky2010}. This symbol is very important, because many properties of the operator can be put in terms of it, and it is one of our main purposes to exploit this principle to exhibit some sufficient conditions for the $L^r_\alpha$-boundedness of the operator.

As we anticipated, we will be following the line of reasoning behind the multiplier theorems obtained by Onnewer, Queck, Kidata, among others, in \cite{Onneweer1985, Onneweer1989, Kitada1987, Yueping2002}. A common denominator in these works is the assumption of a special condition, which the author calls here condition $H(t)$, in order to guarantee the boundedness of operators. Under the assumption of such condition, as the main results of this paper, we extend the multiplier theorems of Onnewer and Quek to compact noncommutative Vilenkin groups, and use them to establish conditions for $L_\alpha^r$-boundedness based on the operator’s symbol, together with a suitable version of the Littlewood-Paley decomposition for our setting. In contrast with the abelian case, the main challenge in our work is the particularities of the unitary dual of non-abelian Vilenkin groups, which requires us to formulate appropriate conditions for a symbol $\sigma$ defined on $\mathrm{Rep}(G)$ to act as a multiplier on $L^r_\alpha(G)$. Our conditions, at lest in the compact case, extend the conditions given by Taibleson \cite{Taibleson1970} and Saloff-Coste \cite{SaloffCoste1986} to compact noncommutative groups and provide analogues of the Hörmander-Mikhlin multiplier theorem on compact $p$-adic Lie groups $\mathbb{G}$.

\subsection{Multipliers on $L^2_\alpha(\G)$}
The first main tool we will employ is a convenient equivalence of norms in the space 
$L^2_\alpha(\G)$, which will later serve as the endpoint space in certain interpolation arguments. 
A common theme in the works of Taibleson and Onnewer \cite{Taibleson1970, SaloffCoste1986} 
is the use of the Fourier transform on an abelian Vilenkin group to define a convenient $L^2$-norm 
on the Fourier side, which in our setting takes the form given in Equation~(\ref{eq1}). 
The underlying idea is quite simple: one can observe in a group like $\mathbb{H}_d$ the equivalence

\begin{equation}\label{EqEquivalenultrametric1}
    \|\mathbf{x}\|_{p}^\alpha \asymp I_\alpha(x) := 
    \sum_{[\eta] \in \widehat{\mathbb{H}}_d} 
    |\eta_3|_p^d \, \|\pi_\eta(\mathbf{x}) - \mathrm{I}_{ d_\eta} \|^2_{HS} \, 
    \|\eta\|_p^{-(\alpha+(2d+1))}.
\end{equation}This equivalence is not difficult to establish, and it actually remains valid 
for more general $d$-dimensional nilpotent $\Z_p$-Lie groups:
\[
    \|\mathbf{x}\|_{p}^\alpha \asymp I_\alpha(x) := 
    \sum_{[\eta] \in \widehat{\G}} 
    d_\eta \, \|\eta(\mathbf{x}) - \mathrm{I}_{ d_\eta} \|^2_{HS} \, 
    \|\eta\|_p^{-(\alpha+d)}.
\]
It is even plausible that the same statement extends to (all) compact Vilenkin groups, 
although in the most abstract setting the required estimates are not immediate. 
The difficulty lies in obtaining a suitable lower bound for 
$\|\eta(\mathbf{x}) - \mathrm{I}_{d_\eta} \|_{HS}$, which is not always guaranteed, 
since the representation $[\eta]$ may have too many eigenvalues equal to $1$. 
We will return to this sub-section 3.5.  

In the cases where the equivalence above holds, as for $\G = \mathbb{H}_d$, 
one can combine it with the Plancherel theorem to obtain Equation~(\ref{eq1}), 
from which the following result readily follows.

\begin{teo}\label{teomultL2v1}
    Let $\sigma \in L^\infty_{op}(\widehat{\G})$ be a symbol and le $\alpha>0$. Assume there is a $C = C(\sigma,\alpha) > 0$ such that for $\|\eta\|_p < \|\xi\|_p$,
    \[
    \left\| \Delta_{\eta}^{\frac{\alpha + d}{2}} \sigma(\xi) \right\|_{\mathrm{op}} \leq C_2 \|\xi\|_p^\frac{-(\alpha + d)}{2}, \quad \forall \xi \in \widehat{\mathbb{G}}.
    \]

Then $T_\sigma$ extends to a bounded operator on $L^2_\alpha(\mathbb{G})$.
\end{teo}

\subsection{Multipliers on $L^r_\alpha (\G)$, $r \neq 2$} When \(r \neq 2\), the problem of proving continuity between \(L^{r}\)-spaces becomes considerably more subtle. The approach we take here is rather well known, so the novelty of our contribution lies in the application of such ideas to a new setting. One of our initial steps, in the same path as \cite{Onneweer1985, Onneweer1989}, is to recall the standard Calderón--Zygmund decomposition and use it to study the boundedness of operators on \(L^{1,w}_{\alpha}(\mathbb{G})\), which is the second endpoint in our interpolation argument. The key assumption ensuring continuity in these spaces will be referred to as \emph{condition \(H(t)\)}, following the terminology in \cite{Onneweer1985, Onneweer1989}. For a compact Vilenkin group, we formulate it using the notation introduced in the following definition: 

\begin{defi}\normalfont
Given a measurable set $A$, let $\1_A$ denote the characteristic function of $A$. The \emph{normalized characteristic function} or \emph{normalized indicator function} is defined as $\epsilon_A(x) := |A|^{-1} \mathbb{1}_A(x)$. In particular, when $A=G_k$ we write $\epsilon_k:= \epsilon_{G_k}$ and it holds that $$\widehat{\epsilon}_k(\xi):=  \mathcal{F}_G [\epsilon_{G_k}] (\xi) = \delta_{G_k^\bot} (\xi ) \,  \mathrm{I}_{d_\xi},$$where $$G_k^\bot:= \{ [\pi] \in \mathrm{Rep}(G) \, : \, \pi(x)= \mathrm{I}_{d_\pi}, \,\, \text{for all} \, x \in G_k\}.$$      
\end{defi}

By using the normalized indicator functions, condition $H(t)$ can be formulated as:

\begin{defi}[Condition $H(t)$]\label{conditionH(t)}
Let $G$ be a compact noncommutative Vilenkin group with an order-bounded sequence of compact open subgroups $\{G_n\}_{n \in \mathbb{N}_0}$. We say that $\sigma \in L^\infty_{\text{op}}(\hat{G})$ satisfies condition $H(t)$, $1 \le t < \infty$, if for all $k \in \mathbb{N}_0$ and some $\epsilon > 0$ it holds
\[
\sup \left\{
\left(\int_{G_n \setminus G_{n+1}} |\mathcal{F}_G^{-1}[\hat{\epsilon}_{G_k} \sigma](x \smallstar y^{-1}) - \mathcal{F}_G^{-1}[\hat{\epsilon}_{G_k} \sigma](x)|^t dx\right)^{1/t} : y \in G_l
\right\} \lesssim |G_n|^{-(\epsilon+1/t')} |G_l|^\epsilon,
\]when $1 < t < \infty$, 
and
\[
\sup \left\{
\int_{G \setminus G_l} |\mathcal{F}_G^{-1}[\hat{\epsilon}_{G_k} \sigma](x \smallstar y^{-1}) - \mathcal{F}_G^{-1}[\hat{\epsilon}_{G_k} \sigma](x)| dx : y \in G_l
\right\} \lesssim 1, \quad t=1.
\]
\end{defi}

Now we are in position to state our main results. The first one is a lemma telling us how boundedness on $L^\alpha_2(\mathbb{G})$ extends to boundedness on the weak space $L_\alpha^{1,w}(\mathbb{G})$.

\begin{lema}\label{LemaH(t)L2alpha}
Let $\sigma \in L^\infty_{\mathrm{op}}(\widehat{\mathbb{G}})$ and assume that $\sigma$ satisfies the condition $H(t)$ for some $t \ge 1$. If $T_\sigma$ is of type $(2,2)$ on $L^\alpha_2(\mathbb{G})$ for some $\alpha$ with $-d/t' < \alpha \le 0$, then $T_\sigma$ is of weak type $(1,1)$ on $L^\alpha_1(\mathbb{G})$.
\end{lema}

As a consequence of Lemma \ref{LemaH(t)L2alpha}, by using interpolation methods, we obtain the following theorem about the boundedness of Fourier multipliers on $L^r_\alpha (\G)$.

\begin{teo}\label{teoL2impliesLr}
Let $\sigma \in L^\infty_{\mathrm{op}}(\widehat{\mathbb{G}})$. Assume that $\sigma$ satisfies condition $H(t)$ for some $t \ge 1$, and that $T_\sigma$ is bounded on $L_\alpha^{2}(\mathbb{G})$ for some $-d/t' < \alpha_0 < d/t'$. Then the Fourier multiplier $T_\sigma$ is bounded on $L_\alpha^r(\mathbb{G})$ for $1<r<\infty$ and $-|\alpha_0|d \le \alpha \le (r-1)|\alpha_0|d$. If the condition $H(1)$ holds, then $T_\sigma$ is bounded on $L^r(\mathbb{G})$ for $1<r<\infty$.
\end{teo}

\begin{rem}
The above result holds in the more general setting of compact Vilenkin groups. We will write the proof of Theorem \ref{teoL2impliesLr} in that setting, limiting ourselves here to the particular case of $\mathbb{G}$ to simplify our statements.
\end{rem}

Applying together Theorem \ref{teomultL2v1} and Theorem \ref{teoL2impliesLr}, we get the following noncommutative generalization of the results of Taibleson \cite{Taibleson1970} and Saloff-Coste \cite{SaloffCoste1986}:

\begin{teo}\label{TeoMultiplierLr}
Let $\sigma \in L^\infty (\widehat{\G})$ be a symbol, al let $1 <t <2$ be a real number.
\begin{enumerate}
    \item If there is a certain $\alpha_0> \frac{d}{2}$ and a certain constant  $C = C(\sigma,\alpha_0) > 0$ such that for $\|\eta\|_p < \|\xi\|_p$,
    \[
    \left\| \Delta_{\eta}^{\alpha_0 +\frac{ d}{2}} \sigma(\xi) \right\|_{\mathrm{op}} \leq C_2 \|\xi\|_p^{-(\alpha_0+\frac{ d}{2})}, \quad \forall \xi \in \widehat{\mathbb{G}},
    \]Then $T_\sigma$ extends to a bounded operator on $L^r(\G)$ with $1 <r<\infty$.
    \item If there is a certain $\alpha_0> d \big( \frac{1}{t} - \frac{1}{2} \big)$ and a certain constant  $C = C(\sigma,\alpha_0,t) > 0$ such that for $\|\eta\|_p < \|\xi\|_p$,
    \[
    \left\| \Delta_{\eta}^{\alpha_0 +\frac{ d}{2}} \sigma(\xi) \right\|_{\mathrm{op}} \leq C_2 \|\xi\|_p^{-(\alpha_0+d(
    \frac{3}{2}-\frac{1}{t}))}, \quad \forall \xi \in \widehat{\mathbb{G}},
    \]Then $T_\sigma$ extends to a bounded operator on $L^r_\alpha (\G)$ with $1 <r<\infty$ and any $$ -d \leq \alpha \leq (r-1) d.$$
\end{enumerate}
\end{teo}

Finally, as an application of Theorem \ref{TeoMultiplierLr}  we will prove the existence of the following Littlewood-Paley decomposition on $L^r_\alpha(\mathbb{G})$:

\begin{teo}\label{teolittlewood}
Let $\G$ be a compact $p$-adic Lie group with dimension $d$. Let $1<r<\infty$ and $-d < \alpha < (r-1)d$. Then for all $f \in L^r_\alpha(\mathbb{G})$ we have the equivalence of norms
\[
\|f\|_{L^r_\alpha(\mathbb{G})} \asymp \|Sf\|_{L^r_\alpha(\mathbb{G})},
\]
where
\[
Sf(x) = \left( \sum_{n \in \mathbb{N}_0} \Big| \sum_{\| \pi \|_p = p^n} d_\pi \mathrm{Tr}\big[ \pi(\mathbf{x}) \hat{f}(\xi) \big] \Big|^2 \right)^{1/2}.
\]
\end{teo}

\section{Preliminaries}

\subsection{Fourier analysis on compact groups}
Let $G$ be a profinite group. Let $\mathrm{Rep}(G)$ be the collection of all unitary finite-dimensional representations of $G$, and let $\widehat{G}$ denote the unitary dual of $G$. We will be working within the framework of Fourier analysis on compact groups, where the group Fourier transform $\mathcal{F}_G$ and the inverse Fourier transform $\mathcal{F}_G^{-1}$ are defined as
\[
\mathcal{F}_G[f](\xi) = \widehat{f}(\xi) := \int_G f(x)\, \xi^*(x)\, dx, 
\qquad 
\mathcal{F}_G^{-1} \varphi(x) := \sum_{[\xi] \in \widehat{G}} d_\xi \,\mathrm{Tr} \big[ \xi(x)\varphi(\xi)\big],
\]
respectively. Here $dx$ denotes the unique normalized Haar measure on $G$. Moreover, the following version of the Plancherel identity holds:
\[
\int_G |f(x)|^2\, dx = \sum_{[\xi] \in \widehat{G}} d_\xi \|\widehat{f}(\xi)\|_{HS}^2.
\]

Let $T$ be a densely defined linear operator acting on some dense subset of $L^2(G)$. By using the Fourier transform on $G$, and the Fourier series representation of $L^2$-functions, the action of the operator $T$, at least on trygonometric polynomials, can be expressed in terms of a certain map
\[
\sigma_T : G \times \mathrm{Rep}(G) \to \bigcup_{[\pi] \in \mathrm{Rep}(G)} \mathcal{L}(\mathcal{H}_\pi), 
\quad \sigma_T(x, \pi) := \pi^*(x) T \pi(x),
\]
commonly called the \emph{symbol} of the operator. In particular, when $T$ is left-invariant, i.e., when $T$ commutes with the left translations of $G$, the symbol does not depend on $x \in G$, and we call $T$ a \emph{Fourier multiplier}. In other words, Fourier multipliers are linear operators of the form
\[
T_\sigma f(x) = \sum_{[\pi] \in \widehat{G}} d_\pi \,\mathrm{Tr} \big[ \pi(x) \sigma(\pi) \widehat{f}(\pi)\big],
\]
where the symbol $\sigma$ is a mapping
\[
\sigma : \mathrm{Rep}(G) \to \bigcup_{[\pi] \in \mathrm{Rep}(G)} \mathcal{L}(\mathcal{H}_\pi),
\qquad \sigma(\pi) \in \mathcal{L}(\mathcal{H}_\pi), \,  \text{ for every } [\pi] \in \mathrm{Rep}(G).
\]
In this paper, we will always assume $\sigma \in L^\infty_{\mathrm{op}}(\widehat{G})$, where
\[
L^\infty_{\mathrm{op}}(\widehat{G}) := \Big\{\sigma : \mathrm{Rep}(G) \to \bigcup_{[\pi] \in \mathrm{Rep}(G)} \mathcal{L}(\mathcal{H}_\pi) \ : \ \sup_{[\xi] \in \widehat{G}} \|\sigma(\xi)\|_{\mathrm{op}} < \infty \Big\}.
\]
An immediate corollary of this symbolic representation of left-invariant operators is that $T_\sigma$ defines a bounded operator on $L^2(G)$ if and only if $\sigma \in L^\infty_{\mathrm{op}}(\widehat{G})$. For the $L^r$ case, the situation is more complicated, and it is an important open problem to determine necessary and sufficient conditions on $\sigma$ so that $T_\sigma$ extends to a bounded operator on $L^r(G)$, $r \neq 2$.

\subsection{Dual of a compact Vilenkin group}
When $G$ is a compact Vilenkin group, we can partition the unitary dual in a special way. To do this, we fix a sequence of compact open subgroups $\mathcal{G} =\{G_n\}_{n \in \mathbb{N}_0}$. We introduce the notation:
\[
G_n^\perp := \{[\pi] \in \mathrm{Rep}(G) : \pi|_{G_n} = \mathrm{I}_{d_\pi} \}, 
\qquad 
\mathrm{Rep}_n(G) := G_n^\perp \setminus G_{n-1}^\perp, 
\qquad 
\widehat{G}(n) := \mathrm{Rep}_n(G) \cap \widehat{G}.
\]
Every matrix representation $\pi : G \to \mathrm{GL}_{d_\pi}(\mathbb{C})$, $[\pi] \in \mathrm{Rep}(G)$, is a continuous function, and $\mathrm{GL}_{d_\pi}(\mathbb{C})$ does not contain small subgroups. Thus, we can find a neighborhood $V$ of $\mathrm{I}_{d_\pi}$ such that the only subgroup of $\mathrm{GL}_{d_\pi}(\mathbb{C})$ contained in $V$ is $\{\mathrm{I}_{d_\pi}\}$. Let $U := \pi^{-1}(V)$. Clearly, $U$ is a neighborhood of $e \in G$. Hence some $G_n$ is contained in $U$ and $\pi(U) \subset \{\mathrm{I}_{d_\pi}\}$. This shows that the kernel of every $\pi$ with $[\pi] \in \mathrm{Rep}(G)$ is a compact open subgroup of $G$ that must contain one of the $G_n$. 

\begin{defi}\label{defiultrametricdual}\normalfont
Let $G$ be a compact Vilenkin group, and let $[\pi] \in \mathrm{Rep}(G)$ be a continuous unitary representation. We define the natural number $n_{[\pi]}$ to be:
\[
n_{[\pi]} := \min \{ n \in \mathbb{N}_0 : \pi|_{G_n} = \mathrm{I}_{d_\pi} \}.
\]
Now let us consider the monoid $(\mathrm{Rep}(G), \otimes)$, with the one-dimensional identity representation as the identity element, which is abelian because $$[ \pi \otimes \pi' ] \cong [\pi] \otimes [\pi'] \cong [\pi'] \otimes [\pi] \cong [ \pi' \otimes \pi ] .$$With this notation, we define the \emph{dual ultrametric associated to} $\mathcal{G} =\{G_n\}_{n \in \mathbb{N}_0}$ on $\mathrm{Rep}(G)$, which we denote by $| \cdot |_{\widehat{G}}$, as
\[
|[\pi]|_{\widehat{G}} :=
\begin{cases}
0, & [\pi] \, \text{ is the identity representation} \,   ( \text{case } \, n_{[\pi]}=0),\\
|G_{n_{[\pi]}}^\bot \cap \widehat{G}| = |G_{n_{[\pi]}}|^{-1}, & \text{if} \, n_{[\pi]} > 0.
\end{cases}
\]
In particular, if $G$ is some kind of topologial manifold with dimension $d$, like for a compact $p$-adic Lie group $\G$, then we write $$ \| [\pi] \|_p:=|[\pi]|_{\widehat{\G}}^{1/d}.$$
\end{defi}

As we mentioned before, the above definition endows $(\mathrm{Rep}(G), \otimes)$ with a certain ``ultrametric'', meaning that $$| [\pi] \otimes [\pi'] |_{\widehat{\G}} \leq  \max \{| [\pi] |_{\widehat{\G}}, \, | [\pi'] |_{\widehat{\G}}\}.$$ 
The balls in this ultrametric are precisely the anihilators $$B_{ \mathrm{Rep}(G)} (n):= \{[\pi ] \in \mathrm{Rep}(G) \, : \, | [\pi]  |_{\widehat{\G}} \leq |G/G_n| \}= G_n^\bot,$$and the spheres would be $$S_{ \mathrm{Rep}(G)} (n):= \{[\pi ] \in \mathrm{Rep}(G) \, : \, | [\pi] |_{\widehat{\G}} = |G/G_n| \}= \mathrm{Rep}_n (G).$$Similarly, for the unitary dual we have$$B_{ \widehat{G}} (n):= \{[\pi ] \in \widehat{G} \, : \, | [\pi] |_{\widehat{\G}} \leq |G/G_n| \}= G_n^\bot \cap \widehat{G},$$and the spheres would be $$S_{ \widehat{G}} (n):= \{[\pi ] \in \widehat{G} \, : \, | [\pi]  |_{\widehat{\G}} = |G/G_n| \}= \mathrm{Rep}_n (G) \cap \widehat{G} = \widehat{G}(n).$$
We have thus proven the decomposition of $\mathrm{Rep}(G)$ and $\widehat{G}$ as disjoint union of spheres:
\[
\mathrm{Rep}(G) = \bigcup_{n \in \mathbb{N}_0} \mathrm{Rep}_n(G), 
\qquad 
\widehat{G} = \bigcup_{n \in \mathbb{N}_0} \widehat{G}(n),
\]
and in this way any $f \in L^2(G)$ can be written as
\[
f(x) = \sum_{n \in \mathbb{N}_0} \sum_{[\pi] \in \widehat{G}(n)} d_\pi \,\mathrm{Tr} \big[ \pi(x) \widehat{f}(\pi) \big].
\]

In particular, for any compact $p$-adic Lie group $\G$, the sets $\G_n^\perp \cap \widehat{\mathbb{G}}$ and $\widehat{\mathbb{G}}(n)$ have nice descriptions in terms of the ultrametric structure of $\widehat{\G}$ as the corresponding balls and spheres:
\[
\G_n^\perp \cap \widehat{\mathbb{G}} = \{ [\pi] \in \widehat{\mathbb{G}} : \|\pi\|_p \le p^n \} =: B_{\widehat{\mathbb{G}}}(n), 
\quad 
\widehat{\mathbb{G}}(n) = \{ [\pi] \in \widehat{\mathbb{G}} : \|\pi\|_p = p^n \} =: S_{\widehat{\mathbb{G}}}(n).
\]

\begin{rem}\label{remultrametricisnotdistance}
Perhaps the term ``ultrametric'' should be replaced by ``ultrametric semi-norm''. The reason is that the function $\rho : \mathrm{Rep}(G) \to \R_0^+$, $\rho(\pi):= |[\pi]|_{\widehat{G}}$, rather than being a distance function satisfies the following conditions: 
\begin{itemize}
    \item $\rho(1)=0$,
    \item $\rho([ \pi] \otimes [\pi' ]) \leq \max \{  \rho ([ \pi]) , \, \rho ([ \pi'])\},$
\end{itemize}
which are the conditions that define a semi-norm (with strong triangle inequality) on a monoid. To properly make $\rho$ a distance function, and therefore an ultrametric, we would have to restrict ourselves to the quotient $\mathrm{Rep}(G)/ \sim$ where $$[\pi] \sim [\xi] \, \text{if and only if} \,  \, \pi \otimes \xi^* \in \mathrm{Rep}^{ab}(G),$$where $\mathrm{Rep}^{ab}(G)$ is the collection of elements of $\mathrm{Rep}(G)$ which are trivial on $[G,G]$. 
\end{rem}

\subsection{Comparison with the real case}

It is known that
\[
|G/G_n| = \sum_{[\xi] \in \widehat{G/G_n}} d_\xi^2 = \sum_{k \le n} \sum_{[\xi] \in \widehat{G}(k)} d_\xi^2,
\]
and also
\[
\sum_{[\xi] \in \widehat{G}(n)} d_\xi^2 = |G/G_n| - |G/G_{n-1}| = |G/G_n|\Big(1 - \frac{1}{|G_{n-1}/G_n|}\Big),
\]
so that we always have
\[
\sum_{[\xi] \in \widehat{G}(n)} d_\xi^2 \leq |G/G_n|,
\]
and 
\[
\sum_{[\xi] \in \widehat{G}(n)} d_\xi^2 \asymp |G/G_n|.
\]
We can think of this property as an analogue of the following estimation on real Lie groups. Let $G$ be a compact Lie group of dimension $d= \mathrm{dim}_\R (G)$, and define
\[
\langle \xi \rangle := (1+\lambda_\xi)^{1/2}, \quad [\xi] \in \widehat{G}.
\]It is proven in \cite{AkylzhanovNets} that 
\begin{equation}
\sum_{\substack{[\xi] \in \widehat{G} \\ \langle \xi \rangle \leq \langle \pi \rangle}} d_\xi^2
\langle \xi \rangle^{\alpha d}   \asymp \langle \pi \rangle^{(\alpha+1)d}, \quad \text{for } \alpha > -1,
\tag{4.4}
\end{equation}
and
\begin{equation}
\sum_{\substack{[\xi] \in \widehat{G} \\ \langle \xi \rangle \geq \langle \pi \rangle}}d_\xi^2
\langle \xi \rangle^{\alpha d}    \asymp \langle \pi \rangle^{(\alpha+1)d}, \quad \text{for } \alpha < -1.
\tag{4.5}
\end{equation}

For a compact $p$-adic Lie group $\G$ we would write instead \begin{align*}
    \sum_{\substack{[\xi] \in \widehat{\G} \\ \| \xi\|_p \leq \| \pi \|_p}} d_\xi^2 \| \xi \|_p^{\alpha d} \asymp \| \pi \|_p^{(\alpha + 1 )d}, \quad \text{for } \alpha > -1,
\end{align*}
and
\begin{align*}
    \sum_{\substack{[\xi] \in \widehat{\G} \\ \| \xi \|_p \geq \| \pi \|_p}}d_\xi^2
\| \xi \|_p^{\alpha d}    \asymp \| \pi \|_p^{(\alpha+1)d}, \quad \text{for } \alpha < -1.
\end{align*}

\subsection{Smooth functions and Sobolev spaces}

Our purpose here is to study Fourier multipliers. Specifically, we want to find sufficient conditions to extend a Fourier multiplier $T_\sigma$ from some dense subset of $L^r_\alpha(G)$, $\alpha > 0$, to a bounded operator on $L^r_\alpha(G)$. Usually, our starting dense subset will be the space of trigonometric polynomials on $G$, here denoted by $\mathcal{P}(G)$, which coincides with the space $\mathcal{D}(G)$ of smooth functions on $G$:
\[
\mathcal{D}(G) := \{ f : G \to \mathbb{C} \mid \text{there exists } l \in \mathbb{N}_0 \text{ such that } f(x \smallstar y) = f(x) \text{ for all } y \in G_l \}.
\]

\begin{pro}
The spaces of smooth functions on $G$ equal the space of trigonometric polynomials on $G$, i.e.,
\[
\mathcal{P}(G) = \mathcal{D}(G).
\]
\end{pro}

\begin{proof}
In one direction, the matrix entries of the representation of a profinite group are locally constant functions, so $\mathcal{P}(G) \subset \mathcal{D}(G)$. On the other side, the Fourier transform $\widehat{f}(\xi)$ of a locally constant function $f \in \mathcal{D}(G)$ is zero for $[\xi] \in \widehat{G}_n$ with $n$ large enough. This fact and the Fourier inversion formula show that $\mathcal{D}(G) \subset \mathcal{P}(G)$. In conclusion, $\mathcal{P}(G) = \mathcal{D}(G)$.
\end{proof}

Fourier multipliers, as defined before, are left-invariant operators on $G$. Thus, such operators can also be thought of as convolution operators given in terms of a certain right convolution kernel, which in terms of the symbol can be expressed as
\[
\mathfrak{r}_\sigma(x) = \sum_{[\pi] \in \widehat{G}} d_\pi \, \mathrm{Tr}\big[ \pi(x) \sigma(\pi) \big].
\]

Considering this, we have two approaches to study the boundedness of a certain invariant operator $T_\sigma$: either we focus on its associated symbol, or we work with its convolution kernel. We will exploit both approaches here, but we are particularly interested in finding explicit conditions on the symbols to guarantee the $L^r$-boundedness of the operator. In the spirit of classical Euclidean results, we show that some ``regularity'' on the symbol is enough to conclude its boundedness on $L^r(G)$, where the regularity of the symbol is measured via the difference operators described in Definition \ref{def:diff_op}. This idea was explored in \cite{SaloffCoste1986}, where Saloff-Coste defines a pseudo-differential calculus in terms of the Vladimirov–Taibleson operator and the weighted difference operators
\[
\Delta_\eta^\beta \sigma(\xi) := \frac{\sigma(\xi+\eta)-\sigma(\xi)}{|\eta|^\beta}.
\]

\begin{defi}
Given a sequence of compact open subgroups $ \mathcal{G}=\{ G_n \}_{n \in \N_0}$, we define the function $\langle \cdot \rangle_{\mathcal{G}} : \mathrm{Rep}(G) \to \mathbb{R}$ as
\[
\langle \pi \rangle_{\mathcal{G}} :=
\begin{cases}
1, & \text{if } \pi \text{ is the identity representation} \, \, (\text{case} \, n_{[\pi]} =0), \\
|G/G_n|, & \text{if } \pi \in \mathrm{Rep}_n(G),\ n \in \mathbb{N}_0, \, \, (\text{case} \, n_{[\pi]} >0).
\end{cases}
\]
For any pair of nontrivial representations, it holds that
\[
\langle \pi \otimes \xi \rangle_{\mathcal{G}} \le \max \{ \langle \xi \rangle_{\mathcal{G}}, \langle \pi \rangle_{\mathcal{G}} \}, \quad [\pi], [\xi] \in \mathrm{Rep}(G).
\]
\end{defi}

Since the Fourier transform diagonalizes the VT operator in Definition \ref{defivladimirov}, we can naturally define Sobolev spaces and characterize them via Fourier transform. These are a scale of Banach spaces measuring the ``smoothness'' of functions and distributions, compatible with the VT operator $D_\mathcal{G}^\alpha$, meaning these Sobolev spaces are the natural domains of $D_\mathcal{G}^\alpha$.

\begin{defi}\normalfont
Let $G$ be a compact Vilenkin group with a bounded filtration $\mathcal{G} = \{ G_n \}_{n \in \mathbb{N}_0}$.
\begin{enumerate}
\item For $1 < r < \infty$ and $\beta \in \mathbb{R}$, the inhomogeneous Sobolev spaces $H_r^\beta(G)$ are the metric completion of $\mathcal{D}(G)$ with the norm
\[
\| f \|_{H_r^\beta(G)} := \Bigg( \int_G \Big| \sum_{n \in \mathbb{N}_0} \sum_{[\xi] \in \widehat{G}(n)} d_\xi \, \mathrm{Tr}\big[ \xi(x) \langle \xi \rangle_\mathcal{G}^\beta \widehat{f}(\xi) \big] \Big|^r \, d\mu_G(x) \Bigg)^{1/r}.
\]
The homogeneous Sobolev space $\dot{H}_r^\beta(G)$ is the metric completion of $\mathcal{D}(G)$ with the norm
\[
\| f \|_{\dot{H}_r^\beta(G)} := \Bigg( \int_G \Big| \sum_{n \in \mathbb{N}} \sum_{[\xi] \in \widehat{G}(n)} d_\xi \, \mathrm{Tr}\big[ \xi(x) \langle \xi \rangle_\mathcal{G}^\beta \widehat{f}(\xi) \big] \Big|^r \, d\mu_G(x) \Bigg)^{1/r}.
\]

\item For $\beta \in \mathbb{C}$, define the Fourier multiplier $\mathbb{D}_\mathcal{G}^\beta$ by its symbol
\[
\sigma_{\mathbb{D}_\mathcal{G}^\beta}(\pi) = \langle \pi \rangle_\mathcal{G}^\beta \, \mathrm{I}_{d_\pi}, \quad [\pi] \in \mathrm{Rep}(G).
\]
Then, for $\beta \in \mathbb{R}$, $f \in H_r^\beta(G)$ if and only if $\mathbb{D}_\mathcal{G}^\beta f \in L^r(G)$.
\end{enumerate}
\end{defi}

\begin{rem}
In the case where $G = \mathbb{G}$ is a compact $p$-dic Lie group, the Sobolev norm can be written as
\[
\| f \|_{H_r^\beta(\mathbb{G})} := \Bigg( \int_{\mathbb{G}} \Big| \sum_{n \in \mathbb{N}_0} \sum_{[\xi] \in \widehat{\mathbb{H}}_d(n)} |\xi_3|_p^d \, \mathrm{Tr}\big[ \pi_\xi(x) \| \xi \|_p^\beta \widehat{f}(\xi) \big] \Big|^r \, dx \Bigg)^{1/r}.
\]
\end{rem}

\subsection{Tensor products of representations}
As the reader might have noticed, it will be essential for us in each step to work with some matrix realization of our representations, and most of the time we will be talking about a certain matrix representation $\pi$ and its corresponding equivalence class $[\pi] \in \mathrm{Rep}(\G)$ as the same thing. This can be done without loss of consistency because most of our analysis is independent on the choice of representative for the class $[\pi]$. With this convention, the tensor product $\pi \otimes \xi$ can be identified with the Kronecker product of the matrix forms of the representations. If $A$ is an $m \times n$ matrix and $B$ is a $p \times q$ matrix, the Kronecker product $A \otimes B$ is
\[
A \otimes B =
\begin{bmatrix}
a_{11} B & \cdots & a_{1n} B \\
\vdots & \ddots & \vdots \\
a_{m1} B & \cdots & a_{mn} B
\end{bmatrix}, \quad \xi \otimes \pi (\mathbf{x}) =
\begin{bmatrix}
\xi_{11} (\mathbf{x}) \pi (\mathbf{x}) & \cdots & \xi_{1d_\xi} (\mathbf{x}) \pi (\mathbf{x}) \\
\vdots & \ddots & \vdots \\
\xi_{d_\xi1} (\mathbf{x}) \pi (\mathbf{x}) & \cdots & \xi_{d_\xi d_\xi} (\mathbf{x}) \pi (\mathbf{x})
\end{bmatrix}.
\]

For compact nilpotent $p$-adic Lie groups we have some additional structure, due to the implications of the Kirillov orbit method. For instance on $\mathbb{H}_d$, as we mentioned before, the irreducible unitary representations of $\mathbb{G}$ are indexed by $\widehat{\mathbb{H}}_d$, which we think of as the sub-tree of $\widehat{\mathbb{Z}}_p^{2d+1}$ given by
\[
\widehat{\mathbb{H}}_d = \left\{ \xi \in \widehat{\mathbb{Z}}_p^{2d+1} : (\xi_1, \xi_2) \in \mathbb{Q}_p^{2d} / p^{\vartheta(\xi_3)} \mathbb{Z}_p^{2d} \right\},
\]
and these representations can be realized in matrix form through the coefficients
\[
(\pi_\xi(\mathbf{x}))_{hh'} = e^{2\pi i \{x \cdot \xi + \xi_3 h' \cdot u\}_p} \mathbb{1}_{h'-h + p^{-\vartheta(\xi_3)}\mathbb{Z}_p^d}(x_1).
\]

Thus, the character of the representation is given by
\[
\chi_{\pi_\xi}(\mathbf{x}) = \sum_{h \in \mathbb{Z}_p / p^{-\vartheta(\xi_3)} \mathbb{Z}_p} e^{2\pi i \{x \cdot \xi + \xi_3 h \cdot u\}_p} \mathbb{1}_{p^{-\vartheta(\xi_3)} \mathbb{Z}_p^d}(x_1)
= |\xi_3|_p^d e^{2\pi i \{x \cdot \xi\}_p} \mathbb{1}_{p^{-\vartheta(\xi_3)} \mathbb{Z}_p^d}(x_1)\mathbb{1}_{p^{-\vartheta(\xi_3)} \mathbb{Z}_p^d}(x_2),
\]
which also allows us to obtain the characters of the tensor products:
\[
\chi_{\pi_\eta \otimes \pi_\xi}(\mathbf{x}) = \chi_{\pi_\eta}(\mathbf{x})\chi_{\pi_\xi}(\mathbf{x})
= |\eta_3|_p^d |\xi_3|_p^d e^{2\pi i \{x \cdot (\xi + \eta)\}_p} \mathbb{1}_{p^{-\vartheta(\xi_3,\eta_3)} \mathbb{Z}_p^{2d}}(x_1, x_2).
\]
Several things we can be deduced from here:

\begin{enumerate}
    \item The unitary dual of the Heisenberg group can be parametrized by a certain subset on $\Q_p^{2d+1}$, or to be more precise, by a subset of a complete collection of representatives for $\widehat{\Z}_p^{2d+1} \cong \Q_p^{2d+1} / \Z_p^{2d+1}$. Such set would be, by the Kirillov orbit method, a complete collection of representatives of $\widehat{\Z}_p^{2d+1}/ \mathrm{Ad}^* (\mathbb{H}_d) \cong \widehat{\mathbb{H}}_d$.
    \item Consider the set $$\widehat{\mathbb{H}}_d^{\N, \, finite} := \{ f: \N \to \widehat{\mathbb{H}}_d \subseteq \widehat{\Z}_p^{2d+1} \, : \, f(n)=0 \,\, \text{for large enough} \, n  \},$$which is the set of finite sequences on $\widehat{\mathbb{H}}_d$. Define $S_\infty$ as the union of all symmetric groups $S_n$. Then $$\Gamma(\mathbb{H}_d):= \widehat{\mathbb{H}}_d^{\N, \, finite}/ S_\infty,$$that is the collection of all finite unordered sequences of elements of $\widehat{\mathbb{H}}_d$, can be identified with $\mathrm{Rep}(\mathbb{H}_d)$. A natural question would be, can we endow $\Gamma(\mathbb{H}_d)$ with some operation isomorphic to $\otimes$?

    \item Given $ \xi, \eta \in \widehat{\mathbb{H}}_d$, the tensor product $[\pi_\xi \otimes \pi_\eta]$ is \textbf{not equal} to $[\pi_{\xi + \eta}]$, but there is a certain relation between $\pi_\xi \otimes \pi_\eta$ and the representation associated to the equivalence class of $\xi + \eta$ in $\widehat{\Z}_p^{2d+1}/ \mathrm{Ad}^* (\mathbb{H}_d)$. Actually, since $(\mathrm{Rep}(\mathbb{H}_d), \otimes)$ is a commutative monoid, and every element of $\mathrm{Rep}(\mathbb{H}_d)$ is a finite direct sum of elements of $\widehat{\mathbb{H}}_d$, we can hope to find a suitably defined  operation $\boxtimes$ to make $(\mathrm{Rep}(\mathbb{H}_d), \otimes)$ and $(\Gamma(\mathbb{H}_d) , \boxtimes),$ isomorphic monoids. Let us agree first that \begin{align*}
        (\xi_1,...,&\xi_n,0,0,...) \boxtimes (\eta_1,...,\eta_m,0,0,...) \\ &=(\xi_1 \boxtimes(\eta_1+...+\eta_m) , \xi_2  \boxtimes(\eta_1+...+\eta_m),..., \xi_n \boxtimes(\eta_1+...+\eta_m),0,0,... ).
    \end{align*}Thus, to define “$\boxtimes$” we only need to define its action on $\widehat{\mathbb{H}}_d$, so we need to understand how the irreducible components of $\pi_\xi \otimes \pi_\eta$ are indexed by elements of $\widehat{\mathbb{H}}_d$. Using a simple Fourier series argument, one can check that
    \[
    \chi_{\pi_\eta \otimes \pi_\xi}(\mathbf{x}) = \sum_{\gamma \in p^{-N(\eta, \xi)} \mathbb{Z}_p^{2d} / p^{\vartheta(\xi_3+\eta_3)} \mathbb{Z}_p^{2d}}
    |\xi_3 + \eta_3|_p^{2d} e^{2\pi i \{ \mathbf{x} \cdot (\xi + \eta + (\gamma_1, \gamma_2, 1)) \}_p} \mathbb{1}_{p^{-\vartheta(\xi_3+\eta_3)} \mathbb{Z}_p^{2d}} (x_1, x_2).
    \]
    \item From the above Fourier series representation we can deduce that  $$[\pi_\xi \otimes \pi_\eta ] \cong \bigoplus_{j=1}^{| \xi_3 + \eta_3 |_p^d} \bigoplus_{\gamma \in p^{-N(\eta, \xi)} \mathbb{Z}_p^{2d} / p^{\vartheta(\xi_3+\eta_3)} \mathbb{Z}_p^{2d}}  [\pi_{\xi +  \eta + (\gamma_1 , \gamma_2, 1)}],  $$and in particular: \[
[\pi_\xi \otimes \pi_\eta ] \cong 
\begin{cases}
\bigoplus_{j=1}^{| \eta_3 |_p^d} [\pi_{\xi + \eta } ] , & \text{if } \, |\xi_3 + \eta_3|_p = |\xi_3|_p \, \,, \\
\bigoplus_{j=1}^{| \xi_3 |_p^d} [\pi_{\xi + \eta } ] , & \text{if } \, |\xi_3 + \eta_3|_p = |\eta_3|_p .
\end{cases}
\]    
\end{enumerate}

These observations motivate the following definitions. Their purpose is to be able to make $\Gamma( \mathbb{H}_d)$ a commutative monoid with the operation “$\boxtimes$”.

\begin{defi}[Product of representations]
\,
\begin{enumerate}
    \item Let $\gamma, \zeta \in \widehat{\mathbb{Z}}_p^d$. We say that $\gamma$ is congruent to $\zeta$ modulo $p^n$, and we write $\gamma \equiv \zeta \mod p^n$, if
    \[
    \gamma - \zeta \in  p^{-n} \mathbb{Z}_p^d.
    \]
    In other words, $\gamma, \zeta \in \widehat{\mathbb{Z}}_p^d$ are congruent if they have the same digits in their $p$-adic expansion, except maybe for the first $n$ digits. This defines, for every fixed $n$, an equivalence relation in $\widehat{\mathbb{Z}}_p^d$, and we denote the class of $\gamma$ in $\Q_p^d / p^{-n}\Z_p^d$ by $[\gamma]_n$.
    \item For $\xi, \eta \in \widehat{\mathbb{Z}}_p^{2d+1}$, let $N(\eta,\xi) \in \mathbb{N}_0$ be the natural number such that
    \[
    p^{N(\eta,\xi)} = \max \{ |\eta_3|_p, |\xi_3|_p \}.
    \]
    We define the function
    \[
    \boxtimes : \widehat{\mathbb{H}}_d \times \widehat{\mathbb{H}}_d \to \Gamma( \widehat{\mathbb{H}}_d)
    \]
    through the formula $\xi \boxtimes \eta :=$
    \[
    \begin{cases} \big([ \xi_1 + \eta_1 ]_{N(\eta,\xi)}, [ \xi_2 + \eta_2 ]_{N(\eta,\xi)}, \xi_3 + \eta_3 \big), & \text{if } \, |\xi_3 + \eta_3|_p = \max \{ |\xi_3|_p, |\eta_3|_p \}  \, \,, \\
\bigoplus_{\gamma \in p^{-N(\eta, \xi)} \mathbb{Z}_p^{2d} / p^{\vartheta(\xi_3+\eta_3)} \mathbb{Z}_p^{2d}} \big([ \xi_1 + \eta_1 + \gamma_1 ]_{N(\eta,\xi)}, [ \xi_2 + \eta_2 + \gamma_2 ]_{N(\eta,\xi)}, \xi_3 + \eta_3 \big) , & \text{if } \, |\xi_3 + \eta_3|_p<\max \{ |\xi_3|_p, |\eta_3|_p \} .
\end{cases} 
    \]
\end{enumerate}
\end{defi}

The above definition might seem a little bit complicated, but it is meant to simplify the notation for the difference operators in our setting. With the above notation we actually can write    

\[
 \sigma( \xi \otimes \eta ) \cong 
\begin{cases}
|\eta_3|_p^d \sigma( \xi \boxtimes \eta) :=  \bigoplus_{j=1}^{| \eta_3 |_p^d}  \sigma([\pi_{\xi + \eta } ]) \cong \sigma(\xi \boxtimes \eta) \otimes  \mathrm{I}_{d_\eta} , & \text{if } \, |\xi_3 + \eta_3|_p = |\xi_3|_p \, \,, \\
|\xi_3|_p^d \sigma(\xi \boxtimes \eta ) := \bigoplus_{j=1}^{| \xi_3 |_p^d}  \sigma ([\pi_{\xi + \eta } ]) \cong  \mathrm{I}_{d_\xi} \otimes \sigma(\xi \boxtimes \eta )  , & \text{if } \, |\xi_3 + \eta_3|_p = |\eta_3|_p ,
\end{cases}
\]and for the case when $| \xi_3 + \eta_3|_p < \max \{|\xi_3|_p, | \eta |_p \} $    
$$\sigma(\xi \otimes \eta ) = |\xi_3 + \eta_3|_p^d \sigma(\xi \boxtimes \eta) := \bigoplus_{j=1}^{| \xi_3 + \eta_3 |_p^d} \bigoplus_{\gamma \in p^{-N(\eta, \xi)} \mathbb{Z}_p^{2d} / p^{\vartheta(\xi_3+\eta_3)} \mathbb{Z}_p^{2d}}  \sigma([\pi_{\xi +  \eta + (\gamma_1 , \gamma_2, 1)}]). $$

\begin{rem}
If the matrices are indexed by $h,h' \in \mathbb{Z}_p^d / p^{-\vartheta(\xi_3)}\mathbb{Z}_p^d$, the block $hh'$ of $\pi_\eta \otimes \pi_\xi$ is
\[
[\pi_\eta \otimes \pi_\xi(\mathbf{x})]_{hh'} := \pi_\eta(\mathbf{x})_{hh'} \pi_\xi(\mathbf{x}),
\]
which implies for the symbol of a linear operator $T$ the following: 
\[
[\pi_\eta \otimes \pi_\xi(\mathbf{x})\sigma_T(\mathbf{x},\eta \otimes \xi)]_{hh'} = \sum_{k \in \mathbb{Z}_p^d /  p^{- \vartheta(\eta_3)}\mathbb{Z}_p^d} [\pi_\eta \otimes \pi_\xi(\mathbf{x})]_{hk} [\sigma_T(\mathbf{x},\eta \otimes \xi)]_{kh'} = T(\pi_\eta \otimes \pi_\xi(\mathbf{x})).
\]
\end{rem}

\begin{rem}
The Kronecker product of matrices corresponds to the abstract tensor product of linear mappings. Specifically, if $V,W,X,Y$ have bases $\{v_i\}, \{w_j\}, \{x_k\}, \{y_\ell\}$, and $A,B$ represent linear transformations $S: V\to X$, $T: W\to Y$, then $A\otimes B$ represents $S\otimes T : V\otimes W \to X \otimes Y$ with respect to the natural bases.
\end{rem}

\subsection{Lower bounds for the representations} A crucial information we need here is a lower bound for $\| \pi (\mathbf{x}) - \mathrm{I}_{d_\pi}\|_{HS}$. The simplest form of this estimate occurs when the group is abelian, so the representations are simply characters. For the particular case of the group $\Q_p^d$ the following proposition holds.   

\begin{pro}\label{pronormAbelian}
    Let $\xi , \mathbf{x}\in \Q_p^d$ be $p$-adic numbers, and assume that $\xi \cdot \mathbf{x} \in \Q_p \setminus \Z_p.$ Then the following estimate holds:
    $$\| \xi \|_p^{-1} \| \mathbf{x}\|_p^{-1} \lesssim |e^{2 \pi i \{\xi \cdot \mathbf{x}\}_p} -1|.$$
    
    \end{pro}

\begin{proof}
Let $l$ be the natural number such that 
\[
p^l = \|\mathbf{x}\|_p \, \|\xi\|_p>1.
\]Then it is easy to check that $p^l \xi \cdot \mathbf{x} \in \Z_p$ so that $e^{2 \pi i\{p^l \xi \cdot \mathbf{x} \}_p} =1$ and  $e^{2 \pi i\{ \xi \cdot \mathbf{x} \}_p} =1$ must be a $p^l$-rooth of unity. We write $$\{ \xi \cdot \mathbf{x} \}_p = \frac{k}{p^l}, \quad \text{for some} \, \, 1 \leq k <p^l.$$ 
   Finally, just notice how \begin{align*}
       |e^{2 \pi i\{p^l \xi \cdot \mathbf{x} \}_p} -1|&= \sqrt{(\cos (2 \pi\{ \xi \cdot \mathbf{x} \}_p -1 )^2 + \sin^2 (2 \pi\{ \xi \cdot \mathbf{x} \}_p)} \\ &= \sqrt{2 (1- \cos(2 \pi\{ \xi \cdot \mathbf{x} \}_p))} \\ &= \sqrt{4 \sin^2 ( \pi\{ \xi \cdot \mathbf{x} \}_p)} =2  \sin( \pi\{ \xi \cdot \mathbf{x} \}_p),
   \end{align*}so as an application of Jordan's inequality we get $$|e^{2 \pi i\{p^l \xi \cdot \mathbf{x} \}_p} -1| =2  \sin( \pi\{ \xi \cdot \mathbf{x} \}_p) \gtrsim p^{-l}.$$
\end{proof}

Proposition \ref{pronormAbelian} may be regarded as the abelian counterpart 
of the lower bound we seek for $\|\pi_\xi(\mathbf{x}) - I_{d_\xi}\|_{HS}.$ However, at present the author is unable to provide a proof that holds 
for arbitrary compact Vilenkin groups. For this reason, it becomes necessary
to impose an additional hypothesis. This assumption is natural and not unduly restrictive, and it is known to hold for several examples  of compact nilpotent $p$-adic Lie groups. Our conjecture here is that the same holds for any compact nilpotent group, and maybe even foe general compact $p$-adic Lie groups. Nevertheless, given the wide structural diversity of 
compact Vilenkin groups, one cannot guarantee the required hypothesis is satisfied in 
full generality, and it becomes necessary to assume it as a given.

The additional condition we are talking about can be interpreted as a constraint on the 
multiplicity of the eigenvalue $1$ in a unitary irreducible representation; 
more precisely, it limits the number of eigenvalues equal to $1$ that such 
a representation may possess.

\begin{defi}\label{defiLowerBound}
    Let $G$ be a compact group. We say that $G$ \textbf{\emph{has the lower bound property}} if, given any $[\pi] \in \widehat{G}$, there is a constant $C>0$ such that $$\# \{ j \, : \, 1 \leq j \leq d_\pi, \,\, \lambda_j (\pi(x)) \neq 1 \} \geq C d_\pi, $$for any $x \in G$. Here $\lambda_1 (\pi(x)), ..., \lambda_{d_\pi} (\pi(x))$ are the eigenvalues of the matrix $\pi(x)$.
\end{defi}

To understand better the above hypothesis, we will check in the next sub-sections some particular examples where it is known to hold. But first, let us show the kind of estimate one can obtain from Definition \ref{defiLowerBound}, which will be crucial for all the equivalence of norms we are about to prove.

\begin{lema}\label{lemaLowerBoundVilenkin}
    Let $G$ be a compact vilenkin group, and assume the lower bound property for $G$. Then, for any $[\pi] \in \widehat{G}$, and any $x \in G$ such that $\pi(x) \neq \mathrm{I}_{d_\pi}$, it holds \[
\|\pi(x) - \mathrm{I}_{d_\pi} \|_{HS} \gtrsim d_\pi^{1/2} |\pi|_{\widehat{G}}^{-1} |x|_{\mathcal{G}}^{-1}.
\]In particular, if $\G$ is a compact $p$-adic Lie group we have: \[
\|\pi(\mathbf{x}) - \mathrm{I}_{d_\pi} \|_{HS} \gtrsim d_\pi^{1/2} \|\pi\|^{-d}_p \|x\|^{-d}_p.
\]
\end{lema}

\begin{proof}
Take $x \in G$ with $|x|_{\mathcal{G}} = |G/G_k|^{-1}$, and the $[\pi] \in \widehat{G}$ with $|[\pi]|_{\widehat{G}} = |G/G_{n}| \geq |G/G_{k+1}|$. Thus, since $\pi (x)$ is an unitary matrix, we can find a basis of the representation space $\mathcal{H}_\pi$ such that $\pi(x)$ is a diagonal matrix. Now let $\lambda_1 (\pi(x)), ..., \lambda_{d_\pi} (\pi(x))$ be the eigenvalues of $\pi (x)$. These are all roots of unity and, since $\pi$ descends to an unitary representation of $G_k / G_n$, we must have $$\pi(x)^{|G_k / G_n|} = I_{d_\pi} \, \, \implies \, \lambda_j (\pi(x))^{|G_k / G_n|} = 1.$$In other words, any $\lambda_j (\pi (x)) \neq 1$ is a $|G_k / G_n|$-root of unity, so we can write  $$\lambda_j (\pi (x)) = e^{2 \pi i \frac{l}{|G_k / G_n|}},   \, \, \text{for some} \, \, 1 \leq l \leq |G_k / G_n| -1. $$ 
This and Jordan's inequality on $[0, \pi)$ leads us to the estimate 

\begin{align*}
    |\lambda_j (\pi (x)) -1| &= \sqrt{( \cos \big(2 \pi  \frac{l}{|G_k / G_n|}\big) - 1)^2 + \big(\sin (2 \pi  \frac{l}{|G_k / G_n|} \big))^2} \\ &= \sqrt{2(1 -\cos \big(2 \pi  \frac{l}{|G_k / G_n|}\big))} \\& = \sqrt{4 \sin^2 \big( \pi  \frac{l}{|G_k / G_n|}\big)} = 2|\sin \big(\pi  \frac{l}{|G_k / G_n|}\big)| \\ & \gtrsim \frac{4}{\pi} \min \big\{\pi  \frac{l}{|G_k / G_n|} , \pi-    \pi \frac{l}{|G_k / G_n|}\big)\big\} \\ & = 4  \min \big\{  \frac{l}{|G_k / G_n|} ,    \frac{|G_k / G_n| - l}{|G_k / G_n|} \big\} \geq \frac{1}{|G_k / G_n|}. 
\end{align*}
Using this and the lower bound property from Definition \ref{defiLowerBound}, we get \begin{align*}
    \| \pi (x) - \mathrm{I}_{d_\pi} \|_{HS}^2 & = \sum_{1 \leq j \leq d_\pi} |\lambda_j (\pi (x)) -1|^2 = \sum_{\{ j \, : \, 1 \leq j \leq d_\pi, \,\, \lambda_j (\pi(x)) \neq 1 \}} |\lambda_j (\pi (x)) -1|^2 \\ & \gtrsim \sum_{\{ j \, : \, 1 \leq j \leq d_\pi, \,\, \lambda_j (\pi(x)) \neq 1 \}} \frac{1}{|G_k / G_n|^2} \\ & = \#\{ j \, : \, 1 \leq j \leq d_\pi, \,\, \lambda_j (\pi(x)) \neq 1 \} \cdot \frac{1}{|G_k / G_n|^2} \gtrsim d_\pi |G_k / G_n|^{-2}. 
\end{align*} 
To conclude our argument, just notice how $$|G_k / G_n|^{-1} = \frac{|G/G_k|}{|G/G_n|} = |\pi|_{\widehat{G}}^{-1} |x|_{\mathcal{G}}^{-1}.$$
This concludes the proof.
\end{proof}

\subsection{Example 1: $\mathbb{H}_d$} Start by recalling the realization for the matrix representations of $[\pi_\xi]$:

\[
\pi_\xi(\mathbf{x})_{hh'} = e^{2\pi i \{ \xi_1 \cdot x_1 + \xi_2 \cdot x_2 + \xi_3(x_3 + x_2 \cdot h') \}_p} \mathbb{1}_{h'-h + p^{-\vartheta(\xi_3)} \mathbb{Z}_p^d}(x_1).
\]
It is important to observe how, given any $\mathbf{x} \in \mathbb{H}_d $ and any column of $\pi_\xi$, there is only one element of the column that can be non-zero, thus we have two possibilities for the entries of $\pi_\xi - \mathrm{I}_{d_\xi}$. Either $\pi_\xi - \mathrm{I}_{d_\xi}$ is a diagonal matrix when $x_1 \in p^{-\vartheta(\xi_3)}\Z_p^d$, or only one of the $\pi_xi(\mathbf{x})_{hh'}$ is non-zero and the diagonal is simply constant $-1$, when $x_1 \in \Z_p^d \setminus p^{-\vartheta(\xi_3)}\Z_p^d$. We can compute
\[
\|\pi_\xi(\mathbf{x}) - \mathrm{I}_{d_\xi}\|^2_{HS} =
\begin{cases}
2|\xi_3|_p^d, & \text{if } x_1 \in \mathbb{Z}^d_p \setminus p^{-\vartheta(\xi_3)}\mathbb{Z}^d_p, \\[6pt]
\displaystyle \sum_{h' \in \mathbb{Z}^d_p / p^{-\vartheta(\xi_3)} \mathbb{Z}^d_p}
\Big| e^{2\pi i \{\xi \cdot \mathbf{x} + \xi_3 h' \cdot x_2\}_p} - 1 \Big|^2, & 
\text{if } x_1 \in p^{-\vartheta(\xi_3)}\mathbb{Z}^d_p.
\end{cases}
\]

Here the roots of the unity $e^{2\pi i \{\xi \cdot \mathbf{x} + \xi_3 h' \cdot x_2\}_p}$ cannot be all one, and actually \textbf{at most one of the eigenvalues of $\pi_\xi (\mathbf{x})$ can be equal to 1}. So in conclusion we have the estimate

$$\# \{ h \in \mathbb{Z}^d_p / p^{-\vartheta(\xi_3)} \mathbb{Z}^d_p \, :  \,\, \lambda_h (\pi_\xi (\mathbf{x})) \neq 1 \}= d_\xi -1 \geq \big( 1 -\frac{1}{p^{2d+1}} \big)  d_\xi, $$
leading us to the conclusion

\begin{align*}
\sum_{ h' \in \mathbb{Z}^d_p / p^{-\vartheta(\xi_3)} \mathbb{Z}^d_p } 
\left| e^{2\pi i \{\xi \cdot \mathbf{x} + \xi_3 h' \cdot x_2\}_p} - 1 \right|^2 
&\geq \sum_{h' \in \mathbb{Z}^d_p / |\xi_3|_p \mathbb{Z}^d_p}
4 \|\xi\|_p^{-2(2d+1)} \|\mathbf{x}\|_p^{-2(2d+1)} \\
&\geq  4 \|\xi\|_p^{-2(2d+1)} \|\mathbf{x}\|_p^{-2(2d+1)} (|\xi_3|_p^d - 1) \\
&\geq \|\xi\|_p^{-2(2d+1)} \|\mathbf{x}\|_p^{-2(2d+1)} |\xi_3|_p^d .
\end{align*}

\subsection{Example 2: The Engel group $\mathcal{B}_4$}Let $p>3$ be a prime number. Let us denote by $\mathfrak{b}_4$ the $\mathbb{Z}_p$--Lie
algebra generated by $X_1,\dots,X_4$ with the commutation relations
\[
[X_1,X_2]=X_3, \qquad [X_1,X_3]=X_4,
\]
and decomposition
\[
\mathfrak{b}_4=\bigoplus_{j=1}^{3}\mathcal{V}_j,
\]
where
\[
\mathcal{V}_1=\mathrm{span}_{\mathbb{Z}_p}\{X_1,X_2\},\qquad
\mathcal{V}_2=\mathrm{span}_{\mathbb{Z}_p}\{X_3\},\qquad
\mathcal{V}_3=\mathrm{span}_{\mathbb{Z}_p}\{X_4\}.
\]

We call the $\mathbb{Z}_p$--Lie algebra $\mathfrak{b}_4$ the \emph{4--dimensional Engel algebra},
and its exponential image, which we denote by $\mathcal{B}_4$, is called the
\emph{Engel group over the $p$--adic integers}.

Let us consider the realization of $\mathfrak{b}_4$ as the matrix algebra
\[
\mathfrak{b}_4(\mathbb{Z}_p)
=
\left\{
\begin{pmatrix}
0 & x_1 & 0 & x_4-\frac12 x_1\!\left(x_3-\frac12 x_1x_2\right)-\frac16 x_1^2x_2 \\
0 & 0 & x_1 & x_3-\frac12 x_1x_2 \\
0 & 0 & 0 & x_2 \\
0 & 0 & 0 & 0
\end{pmatrix}
\in \mathcal{M}_4(\mathbb{Z}_p)
\;:\;
x\in\mathbb{Z}_p^{4}
\right\}.
\]

With this realization, and using the usual matrix exponential map, we can think
of $\mathcal{B}_4(\mathbb{Z}_p)$ as the matrix group
\[
\mathcal{B}_4(\mathbb{Z}_p)
=
\left\{
\begin{pmatrix}
1 & x_1 & \tfrac12 x_1^2 & x_4 \\
0 & 1 & x_1 & x_3 \\
0 & 0 & 1 & x_2 \\
0 & 0 & 0 & 1
\end{pmatrix}
\in \mathcal{M}_4(\mathbb{Z}_p)
\;:\;
x\in\mathbb{Z}_p^{4}
\right\},
\]
which is analytically isomorphic to $\mathbb{Z}_p^4$ with the operation
\[
\mathbf{x}\star\mathbf{y}
=
\bigl(
x_1+y_1,\;
x_2+y_2,\;
x_3+y_3+x_1y_2,\;
x_4+y_4+x_1y_3+\tfrac12 x_1^2y_2
\bigr).
\]

The exponential map transforms sub--ideals of the Lie algebra $\mathfrak{b}_4$
to subgroups of $\mathcal{B}_4\simeq(\mathfrak{b}_4,\star)$, which can be endowed
with the sequence of subgroups
\[
J_n := \bigl(\mathfrak{b}_4(p^n\mathbb{Z}_p),\star\bigr),
\]
where
\[
\mathfrak{b}_4(p^n\mathbb{Z}_p)
=
p^n\mathbb{Z}_pX_1
+
p^n\mathbb{Z}_pX_2
+
p^n\mathbb{Z}_pX_3
+
p^n\mathbb{Z}_pX_4 .
\]

Thus $\mathcal{B}_4$ is a compact Vilenkin group together with the sequence of
compact open subgroups
\[
G_n := \mathcal{B}_4(p^n\mathbb{Z}_p)
= \exp\!\bigl(\mathfrak{b}_4(p^n\mathbb{Z}_p)\bigr),
\qquad n\in\mathbb{N}_0 .
\]

Notice that the sequence $\mathcal{G}=\{G_n\}_{n\in\mathbb{N}_0}$ forms a basis
of neighbourhoods at the identity, so the group is metrizable. We can endow it
with the natural ultrametric
\[
\rho_{\mathcal{G}}(\mathbf{x},\mathbf{y})=
\begin{cases}
0, & \mathbf{x}=\mathbf{y},\\
|G_n|=p^{-4n}, & \mathbf{x}\star\mathbf{y}^{-1}\in G_n\setminus G_{n+1}.
\end{cases}
\]

Nevertheless, instead of this ultrametric we will use the $p$--adic norm
\[
\|\mathbf{x}\|_p
:=
\max_{1\le j\le4} |x_j |_p .
\]

Notice that
\[
\|\mathbf{x}\|_p^{4}=|\mathbf{x}|_{\mathcal{G}},
\qquad \text{for any } \mathbf{x}\in\mathcal{B}_4 .
\]

Let us denote by $\widehat{\mathcal{B}_4}$ the unitary dual of $\mathcal{B}_4$,
i.e., the collection of all unitary irreducible representations of $\mathcal{B}_4$.
Then we can identify $\widehat{\mathcal{B}_4}$ with the following subset of
$\widehat{\mathbb{Z}_p^4}\cong \mathbb{Q}_p^4/\mathbb{Z}_p^4$:
\[
\widehat{\mathcal{B}_4}
:=
\left\{
\xi \in \widehat{\mathbb{Z}_p^4}
\;:\;
1 \le |\xi_4|_p < |\xi_3|_p
\ \wedge\
(\xi_1,\xi_2,\xi_3)\in \widehat{\mathbb{H}_1},
\ \text{or}\ 
|\xi_3|_p = 1
\ \wedge\
\xi_1\in \mathbb{Q}_p/p^{\nu(\xi_4)}\mathbb{Z}_p
\right\}.
\]

Moreover, each non–trivial representation $[\pi_\xi]$ is equivalent to one
representation $[\pi_\xi]\in\widehat{\mathcal{B}_4}$ which can be realized in the
following finite dimensional subspace $\mathcal{H}_\xi$ of $L^2(\mathbb{Z}_p)$:
\[
\mathcal{H}_\xi
:=
\mathrm{span}_{\mathbb{C}}
\left\{
\|(\xi_3,\xi_4)\|_p^{1/2}
\,\mathbf{1}_{\,h+p^{-\vartheta(\xi_3,\xi_4)}\mathbb{Z}_p}
\;:\;
h\in \mathbb{Z}_p/p^{-\vartheta(\xi_3,\xi_4)}\mathbb{Z}_p
\right\}.
\]

The dimension of this space is
\[
d_\xi := \dim_{\mathbb{C}}(\mathcal{H}_\xi)
=
\max\{|\xi_3|_p,|\xi_4|_p\}.
\]

The representation acts on functions $\varphi\in\mathcal{H}_\xi$
according to the formula
\[
(\pi_\xi(\mathbf{x})\varphi)(u)
:=
e^{2\pi i
\left\{
\xi_1x_1
+\xi_2x_2
+\xi_3(x_3+ux_2)
+\xi_4\!\left(x_4+ux_3+\tfrac{u^2}{2}x_2\right)
\right\}_p}
\,
\varphi(u+x_1).
\]

With this explicit realization of $[\pi_\xi]\in\widehat{\mathcal{B}_4}$,
and with the natural choice of basis for each representation space,
the associated matrix coefficients $(\pi_\xi)_{hh'}$ are given by
\[
(\pi_\xi(\mathbf{x}))_{hh'}
=
e^{2\pi i
\left\{
\xi\cdot \mathbf{x}
+
(\xi_3x_2+\xi_4x_3)h'
+
\xi_4x_2\frac{(h')^2}{2}
\right\}_p}
\,
\mathbf{1}_{\,h'-h+p^{- \vartheta (\xi_3,\xi_4)}\mathbb{Z}_p}(x_1).
\]

In a similar situation as for the heisenberg group we obtain 
\[
\|\pi_\xi(\mathbf{x}) - \mathrm{I}_{d_\xi}\|^2_{HS} =
\begin{cases}
2\| (\xi_3 , \xi_4) \|_p, & \text{if } x_1 \in \mathbb{Z}_p \setminus p^{-\vartheta(\xi_3, \xi_4)}\mathbb{Z}_p, \\[6pt]
\displaystyle \sum_{h' \in \mathbb{Z}_p / p^{-\vartheta(\xi_3, \xi_4)} \mathbb{Z}_p}
\Big| e^{2\pi i
\left\{
\xi\cdot \mathbf{x}
+
(\xi_3x_2+\xi_4x_3)h'
+
\xi_4x_2\frac{(h')^2}{2}
\right\}_p} - 1 \Big|^2, & 
\text{if } x_1 \in p^{-\vartheta(\xi_3, \xi_4)}\mathbb{Z}_p.
\end{cases}
\]

Here the roots of the unity $e^{2\pi i
\left\{
\xi\cdot \mathbf{x}
+
(\xi_3x_2+\xi_4x_3)h'
+
\xi_4x_2\frac{(h')^2}{2}
\right\}_p}$ cannot be all one at the same time, and actually \textbf{at most one of the eigenvalues of $\pi_\xi (\mathbf{x})$ can be equal to 1}. So in conclusion we have the estimate

$$\# \{ h \in \mathbb{Z}_p / p^{-\vartheta(\xi_3, \xi_4)} \mathbb{Z}_p \, :  \,\, \lambda_h (\pi_\xi (\mathbf{x})) \neq 1 \}= d_\xi -1 \geq \big( 1 -\frac{1}{p^{4}} \big)  d_\xi, $$
leading us to the conclusion

\begin{align*}
\sum_{ h' \in \mathbb{Z}_p / p^{-\vartheta(\xi_3, \xi_4)} \mathbb{Z}_p } 
\left| e^{2\pi i
\left\{
\xi\cdot \mathbf{x}
+
(\xi_3x_2+\xi_4x_3)h'
+
\xi_4x_2\frac{(h')^2}{2}
\right\}_p} - 1 \right|^2 
&\geq \sum_{h' \in \mathbb{Z}_p / \|(\xi_3, \xi_4) \|_p \mathbb{Z}_p}
4 \|\xi\|_p^{-2(4)} \|\mathbf{x}\|_p^{-2(4)} \\
&\geq  4 \|\xi\|_p^{-2(4)} \|\mathbf{x}\|_p^{-2(4)} (\|(\xi_3, \xi_4) \|_p - 1) \\
&\geq \|\xi\|_p^{-2(4)} \|\mathbf{x}\|_p^{-2(4)} \|(\xi_3, \xi_4 ) \|_p .
\end{align*}

\begin{rem}
    We have included only two examples of compact nilpotent \(p\)-adic Lie groups exhibiting the lower bound property, solely for the sake of brevity. For additional examples, the reader is referred to \cite{VelasquezRodriguez2025}. We conjecture that \emph{every compact nilpotent \(p\)-adic Lie group satisfies this property}. While more general \(p\)-adic groups may fail to do so, though I have no idea about it, in the nilpotent case it is expected that the optimal constant \(C>0\) appearing in Definition~\ref{defiLowerBound} is given by
\[
C = 1 - \frac{1}{p^d} \, , \quad d = \mathrm{dim}_{\mathbb{Q}_p}(\mathbb{G}).
\]
\end{rem}
\subsection{Interpolation theorems} Finally, to conclude our preliminaries, we recall here some classical interpolation theorems we will be using along the paper. The first one is an interpolation argument between weighted $L^r$-spaces, the second is a weighted version of Marcinkiewicz interpolation theorem. See \cite{Stein1971, Stein1958} for more details.

\begin{teo}\label{teoSteinWeiss}
Let $(X,\lambda)$ and $(Y,\rho)$ be $\sigma$-finite measure spaces.

Let $w_0,w_1 : X \to (0,\infty)$ and $v_0,v_1 : Y \to (0,\infty)$ be measurable
weight functions.

For $j=0,1$, define the weighted measures
\[
d\mu_j = w_j\, d\lambda,
\qquad
d\nu_j = v_j\, d\rho.
\]

Let $1 \le r_0,r_1,s_0,s_1 \le \infty$.  Assume that $T$ is a linear operator such that
\[
T :
L^{r_0}(X,\mu_0)
\longrightarrow
L^{s_0}(Y,\nu_0), \quad 
T :
L^{r_1}(X,\mu_1)
\longrightarrow
L^{s_1}(Y,\nu_1),
\]
are bounded mappings, and that there exist constants $M_0,M_1>0$ satisfying
\[
\|Tf\|_{L^{s_j}(Y,\nu_j)}
\le
M_j
\|f\|_{L^{r_j}(X,\mu_j)},
\qquad j=0,1.
\]

Let $0<\theta<1$ and define
\[
\frac1{r_\theta}
=
\frac{1-\theta}{r_0}
+
\frac{\theta}{r_1},
\qquad
\frac1{s_\theta}
=
\frac{1-\theta}{s_0}
+
\frac{\theta}{s_1}.
\]

Define the interpolated weights
\[
w_\theta
=
w_0^{\,1-\theta}
w_1^{\,\theta},
\qquad
v_\theta
=
v_0^{\,1-\theta}
v_1^{\,\theta},
\]
and the corresponding measures
\[
d\mu_\theta = w_\theta\, d\lambda,
\qquad
d\nu_\theta = v_\theta\, d\rho.
\]

Then $T$ extends uniquely to a bounded operator $T :
L^{r_\theta}(X,\mu_\theta)
\rightarrow
L^{s_\theta}(Y,\nu_\theta),$ and satisfies the estimate
\[
\|Tf\|_{L^{s_\theta}(Y,\nu_\theta)}
\le
M_0^{\,1-\theta}
M_1^{\,\theta}
\,
\|f\|_{L^{r_\theta}(X,\mu_\theta)}.
\]
\end{teo}

\begin{teo}[Weighted Marcinkiewicz Interpolation]\label{teoMarcinkiewicz}
Let $(X,\mu)$ be a measure space and let $w_0, w_1$ be nonnegative measurable weights on $X$. 
Let $T$ be a sublinear operator defined on simple functions such that:

\begin{enumerate}
    \item[\textnormal{(i)}] (\textbf{Weak type $(1,1)$}) There exists $C_0>0$ such that
    \[
    \|Tf\|_{L^{1,\infty}(w_0\, d\mu)}
    \le C_0 \|f\|_{L^{1}(w_0\, d\mu)}.
    \]
    
    \item[\textnormal{(ii)}] (\textbf{Strong type $(2,2)$}) There exists $C_1>0$ such that
    \[
    \|Tf\|_{L^{2}(w_1\, d\mu)}
    \le C_1 \|f\|_{L^{2}(w_1\, d\mu)}.
    \]
\end{enumerate}

Then for every $1<r<2$ and $\theta \in (0,1)$ satisfying
\[
\frac{1}{r} = \frac{1-\theta}{1} + \frac{\theta}{2},
\]
if we define the interpolated weight
\[
w_\theta = w_0^{\,1-\theta} w_1^{\,\theta},
\]
the operator $T$ extends to a bounded operator on $L^r(w_\theta\, d\mu)$ and there exists a constant $C_p>0$ such that
\[
\|Tf\|_{L^{r}(w_\theta\, d\mu)}
\le C_p \, C_0^{\,1-\theta} C_1^{\,\theta}
\|f\|_{L^{r}(w_\theta\, d\mu)}.
\]
\end{teo}

\section{Multipliers on $L^2_\alpha(\mathbb{G})$}
As we promised, the first step in our journey is to prove a convenient equivalence of norms. This is the content of the Lemma \ref{lemma3.1}.

\begin{lema}\label{lemma3.1}
Let $\G$ be a compact $p$-adic Lie group with dimension $d= \mathrm{dim}_{\Q_p} (\G)$, and let $\alpha > 0$. Assume $\G$ has the lower bound property from Definition \ref{defiLowerBound}. Then:
\begin{enumerate}
    \item[(i)] For all $\mathbf{x} \in \mathbb{G}$ and $\alpha > 0$ it holds:
    \[
    \|\mathbf{x}\|_{p}^\alpha \asymp I_\alpha(\mathbf{x}) := 
    \sum_{[\eta] \in \widehat{\G}} 
    d_\eta \|\eta\|_p^{-(\alpha+d)}  \, \|\eta(\mathbf{x}) - \mathrm{I}_{d_\eta} \|^2_{HS} .
    \]

    \item[(ii)] For $\alpha > 0$ and $f \in L^2_\alpha(\mathbb{G})$ it holds:
    \[
    \| f \|^2_{L^2_\alpha(\mathbb{G})} \asymp 
    \sum_{[\xi] \in \widehat{\mathbb{G}}} 
    \sum_{[\eta] \in \widehat{\mathbb{G}}}
    d_\xi d_\eta \, 
    \|\eta\|_p^{-(\alpha+d)} \, 
    \|\Delta_{\eta} \hat{f}(\xi)\|^2_{HS}.
    \]

    \item[(iii)] For $0 < \alpha < d$ and $f \in L^2(\mathbb{G})$ there exists $C>0$ such that
    \[
    \|f\|_{H^{- \alpha/2}_2(\mathbb{G})} \leq C \|f\|_{L^2_\alpha(\mathbb{G})}.
    \]
\end{enumerate}
\end{lema}

 \begin{proof}
\,
 
\textit{\textbf{(i)}} We begin by fixing $\mathbf{x} \in \mathbb{G}$, so that 
\[
\|\mathbf{x}\|_{p}^\alpha = p^{-\alpha k}.
\] 
Then we can check how $\eta (\mathbf{x}) = \mathrm{I}_{d_\eta}$ for $\|\eta \|_p \leq \|\mathbf{x}\|^{-1}_p$.  
So, in one hand we have
\begin{align*}
I_\alpha(\mathbf{x}) 
&= \sum_{\substack{[\eta] \in \widehat{\mathbb{G}} \\ \|\eta\|_p \geq p \|\mathbf{x}\|^{-1}_p}}
d_\eta \, \|\eta(\mathbf{x}) - \mathrm{I}_{d_\eta} \|^2_{HS} \, \|\eta\|_p^{-(\alpha+d)} \\
&\leq \sum_{\substack{[\eta] \in \widehat{\mathbb{G}} \\ \|\eta\|_p \geq p \|\mathbf{x}\|^{-1}_p}}
d_\eta (2d_\eta ) \|\eta\|_p^{-(\alpha+d)} \\
&= 2 \sum_{n=k+1}^\infty 
p^{-n(\alpha+d)}
\sum_{\substack{[\eta] \in \widehat{\mathbb{G}} \\ \|\eta\|_p = p^n}} 
d_\eta^2 = 2 p^{-\alpha} \frac{
1 - p^{-d}}{1-p^{- \alpha}} \|\mathbf{x}\|^\alpha_p \lesssim \|\mathbf{x}\|^\alpha_p.
\end{align*}

On the other hand, for $\|\eta\|_p \geq p \|\mathbf{x}\|_p^{-1}$ we can use the lower bound property and Lemma \ref{lemaLowerBoundVilenkin} to deduce:
\begin{align*}
I_\alpha(\mathbf{x}) 
&= \sum_{\substack{[\eta] \in \widehat{\mathbb{G}} \\ \|\eta \|_p \geq p \|\mathbf{x}\|_p^{-1}}}
d_\eta \|\eta\|_p^{-(\alpha+d)} \|\eta(\mathbf{x}) - \mathrm{I}_{ d_\eta}\|_{HS}^2   \\
&\gtrsim \sum_{\substack{ [\eta] \in \widehat{\mathbb{G}} \\ \| \eta \|_p \geq p \|\mathbf{x}\|_p^{-1}}}
d_\eta \|\eta \|_p^{-(\alpha+d)} \left( d_\eta \| \eta \|_p^{-2d} \|\mathbf{x}\|_p^{-2d}  \right)  \\
&\geq \sum_{\substack{ [\eta] \in \widehat{\mathbb{G}} \\ \|\eta\|_p = p \|\mathbf{x}\|_p^{-1}}}
d_\eta^2 \|\eta\|_p^{-(\alpha+d)}  \left( \|\eta\|_p^{-2d} \|\mathbf{x}\|_p^{-2d} \right) \\
&= (1 - p^{-(2d+1)}) p^{-\alpha} p^{-2(2d+1)} \|\mathbf{x}\|^\alpha_p \gtrsim \| \mathbf{x} \|_p^\alpha.
\end{align*}

\medskip

\noindent (ii) Let $f \in L^2(\mathbb{G})$ and let us denote $q^\eta_{hh'}(\mathbf{x}) := \eta (\mathbf{x})_{hh'} - \delta_{hh'} .$ Then
\begin{align*}
\int_{\mathbb{G}} |f(\mathbf{x})|^2 |\overline{q^\eta_{h'h}(\mathbf{x})}|^2 \, d\mathbf{x}
&= \sum_{[\xi] \in \widehat{\mathbb{G}}} 
d_\xi  \, \| \mathcal{F}_{\G} [ f\big( \overline{\eta(\mathbf{x})_{h'h}} - \delta_{hh'} \big)](\xi)\|_{HS}^2 \\
&= \sum_{[\xi] \in \widehat{\mathbb{G}}} 
d_\xi  \, \|\Delta_{\overline{q^\eta_{h'h}}} \hat{f}(\xi)\|_{HS}^2 \\
&= \sum_{[\xi] \in \widehat{\mathbb{G}}} 
d_\xi \, \| [\Delta_\eta \hat{f}(\xi)]_{hh'}\|_{HS}^2 .
\end{align*}

In this way,
\begin{align*}
\int_{\mathbb{G}} |f(\mathbf{x})|^2 \|\mathbf{x}\|_p^{\alpha} \, d\mathbf{x}
&\asymp \sum_{[\xi] \in \widehat{\mathbb{G}}} \sum_{[\eta] \in \widehat{\mathbb{G}}}
\sum_{1 \leq  h,h' \leq d_\eta}
d_\xi  d_\eta \, \|\eta\|_p^{-(\alpha+d)} 
\| [\Delta_\eta \hat{f}(\xi)]_{hh'} \|_{HS}^2 \\
&= \sum_{[\xi] \in \widehat{\mathbb{G}}} \sum_{[\eta] \in \widehat{\mathbb{G}}}
d_\xi d_\eta  \, \|\eta\|_p^{-(\alpha+d)} 
\|\Delta_\eta \hat{f}(\xi)\|_{HS}^2 \\
&= \sum_{[\xi] \in \widehat{\mathbb{G}}} \sum_{[\eta] \in \widehat{\mathbb{G}}}
d_\xi d_\eta\, 
\|\Delta^{\frac{\alpha+d}{2}}_\eta \hat{f}(\xi)\|_{HS}^2 .
\end{align*}

This concludes the proof.

\textbf{(iii)} Let $f \in \mathcal{D}(\mathbb{G})$, let us say $f(\mathbf{x}) = \sum_{g \in \mathbb{G} / \mathbb{G}_n} f(\mathbf{g}) \mathbb{1}_{\mathbf{g} \smallstar \mathbb{G}_n}(\mathbf{x}),$ so that
\[
\widehat{f}(\xi) = \sum_{\mathbf{g} \in \mathbb{G} / \mathbb{G}_n} f(\mathbf{g}) \pi^*_\xi(\mathbf{g}) |\mathbb{G} / \mathbb{G}_n|^{-1} \delta_{\mathbb{G}_n^\perp}(\xi).
\]
We can see that
\[
\|\widehat{f}(\xi)\|^2_{HS} \lesssim \sum_{\mathbf{g} \in \mathbb{G} / \mathbb{G}_n} |f(\mathbf{g})|^2 \|\xi^*(\mathbf{g})\|^2_{HS} |\mathbb{G} / \mathbb{G}_n|^{-2} \delta_{\mathbb{G}_n^\perp}(\xi)
\leq \sum_{\mathbf{g} \in G / \mathbb{G}_n} |f(\mathbf{g})|^2 |\mathbb{G} / \mathbb{G}_n|^{-2} d_\xi \, \delta_{\mathbb{G}_n^\perp}(\xi),
\]
and consequently
\[
\|f\|^2_{H^{-(\alpha/2)}_2(\mathbb{G})} 
= \sum_{[\xi] \in \widehat{\mathbb{G}}} d_\xi \|\xi\|^{-\alpha}_p \|\widehat{f}(\xi)\|^2_{HS}
\lesssim \sum_{\mathbf{g} \in \mathbb{G} / \mathbb{G}_n} |f(\mathbf{g})|^2 |\mathbb{G} / \mathbb{G}_n|^{-2}
\sum_{k \leq n} \sum_{[\xi] \in \widehat{G}(k)} d_\xi^2 \|\xi\|^{-\alpha}_p.
\]
Thus
\[
\|f\|^2_{H^{-(\alpha/2)}_2(\mathbb{G})} \lesssim 
\sum_{\mathbf{g} \in \mathbb{G} / \mathbb{G}_n} |f(\mathbf{g})|^2 |\mathbb{G} / \mathbb{G}_n|^{-2} 
\sum_{k \leq n} |\mathbb{G} / \G_k|^{1-\alpha/d}
\lesssim \sum_{\mathbf{g} \in \mathbb{G} / \mathbb{G}_n} |f(\mathbf{g})|^2 |\mathbb{G} / \mathbb{G}_n|^{-(\alpha/d+1)}.
\]

On the other hand,
\[
\|f\|^2_{L^2_\alpha(\mathbb{G})} 
= \int_{\mathbb{G}} |f(\mathbf{x})|^2 \|\mathbf{x}\|^\alpha_p \, d\mathbf{x}
= \sum_{\mathbf{g} \in \G / \mathbb{G}_n} |f(\mathbf{g})|^2 \int_{\mathbf{g} \smallstar \mathbb{G}_n} \|\mathbf{x}\|^\alpha_p \, d\mathbf{x}.
\]
Hence
\[
\|f\|^2_{L^2_\alpha(\mathbb{G})}
= \sum_{\mathbf{g} \in \mathbb{G} / \mathbb{G}_n, \mathbf{g} \neq \mathbf{e}} |f(\mathbf{g})|^2 \|\mathbf{g}\|^\alpha_p \int_{\mathbf{g} \smallstar \mathbb{G}_n} d\mathbf{x}
+ |f(\mathbf{e})|^2 \int_{\mathbb{G}_n} \|\mathbf{x}\|^\alpha_p \, d\mathbf{x}.
\]
But we know that $\int_{\mathbb{G}_n} \|\mathbf{x}\|^\alpha_p \, d\mathbf{x} \asymp |\mathbb{G} / \mathbb{G}_n|^{-(\alpha/d+1)},$ so, using the equalities $$\| \mathbf{g} \|_p^\alpha \geq p^{n \alpha} = |\G / \G_n|^{\alpha/d}, \quad \int_{\mathbf{g} \smallstar \mathbb{G}_n} d\mathbf{x} = |\G / \G_n|^{-1}, $$
we get in conclusion $\|f\|_{H^{-\alpha/2}_2(\mathbb{G})} \lesssim \|f\|_{L^2_\alpha(\mathbb{G})}.$

\end{proof}

We conclude this section with the proof of Theorem \ref{teomultL2v1}. As we anticipated, after the equivalence of norms given in Lemma \ref{lemma3.1}, the desired multiplier theorem follows as a simple consequence of the product rule for the difference operators.

\begin{proof}[Proof of Theorem \ref{teomultL2v1}]
    Just notice how the difference operators $\Delta_\eta$ have the following product rule:
\begin{align*}
\Delta_\eta (\sigma\widehat{ f})(\xi) 
&= \sigma(\eta \otimes \xi)\,\hat{f}(\eta \otimes \xi) 
   - \sigma(\mathrm{I}_{d_\eta} \otimes \xi)\,\hat{f}(\mathrm{I}_{d_\eta} \otimes \xi) \\
&= \sigma(\eta \otimes \xi)\,\Delta_\eta \hat{f}(\xi) 
   + \Delta_\eta \sigma(\xi)\,\hat{f}(\mathrm{I}_{d_\eta} \otimes \xi) \\
&= \sigma(\eta \otimes \xi)\,\Delta_\eta \hat{f}(\xi) 
   + \Delta_\eta \sigma(\xi)\,(\mathrm{I}_{d_\eta} \otimes \hat{f}(\xi)) \\ &= \Delta_\eta \sigma (\xi) \widehat{f}(\eta \otimes \xi) + \sigma (\mathrm{I}_{d_\eta} \otimes \xi ) \Delta_\eta \widehat{f} (\xi ).
\end{align*}

So we actually have the following estimate:
\begin{align*}
\|\Delta_\eta (\sigma\widehat{ f})(\xi)\|^2_{HS} 
&\lesssim 
\|\sigma(\eta \otimes \xi)\,\Delta_\eta \hat{f}(\xi)\|^2_{HS} 
+ \|\Delta_\eta \sigma(\xi)(\mathrm{I}_{d_\eta} \otimes \hat{f}(\xi))\|^2_{HS} \\
&\leq \|\sigma(\eta \otimes \xi)\|^2_{op}\,\|\Delta_\eta \hat{f}(\xi)\|^2_{HS} 
+ d_\eta \,\|\Delta_\eta \sigma(\xi)\|^2_{op}\,\|\hat{f}(\xi)\|^2_{HS}.
\end{align*}

Now we need to apply Lemma 3.1. First:
\begin{align*}
\quad & \|T_\sigma f \|^2_{L^2_\alpha(\mathbb{G})} 
\asymp \sum_{[\xi]\in \widehat{\mathbb{G}}}
\sum_{[\eta]\in \widehat{\mathbb{G}}}
d_\xi d_\eta \,\|\eta\|_p^{-(\alpha+d)}\,
\|\Delta_\eta (\sigma \widehat{f})(\xi)\|^2_{HS} \\
&\lesssim \sum_{[\xi]\in \widehat{\mathbb{G}}}
\sum_{[\eta]\in \widehat{\mathbb{G}}}
d_\xi d_\eta \,\|\eta\|_p^{-(\alpha+d)}\,
\Big( \|\sigma(\eta \otimes \xi)\|^2_{op}\,\|\Delta_\eta \hat{f}(\xi)\|^2_{HS} 
+ d_\eta \,\|\Delta_\eta \sigma(\xi)\|^2_{op}\,\|\hat{f}(\xi)\|^2_{HS} \Big) \\
&\lesssim \sup_{[\xi]\in \widehat{\mathbb{G}}} \|\sigma(\xi)\|^2_{op}\,\|f\|^2_{L^2_\alpha(\mathbb{G})}
+ \sum_{[\xi]\in \widehat{\mathbb{G}}}
\sum_{[\eta]\in \widehat{\mathbb{G}}}
d_\xi d_\eta^2 \,\|\eta\|^{-(\alpha+d)}_p\,
\|\Delta_\eta \sigma(\xi)\|^2_{op}\,\|\hat{f}(\xi)\|^2_{HS}.
\end{align*}

For the right-hand side, notice how
\begin{align*}
&\sum_{[\xi]\in \widehat{\mathbb{G}}}
\sum_{\substack{[\eta]\in \widehat{\mathbb{G}} \\ \|\eta\|_p < \|\xi\|_p}}
d_\xi d_\eta^2\,\|\eta\|_p^{-(\alpha+d)}\,
\|\Delta_\eta \sigma(\xi)\|^2_{op}\,\|\hat{f}(\xi)\|^2_{HS} \\
&= \sum_{[\xi]\in \widehat{\mathbb{G}}}
d_\xi
\left( \sum_{\substack{[\eta]\in \widehat{\mathbb{G}}\\ \|\eta\|_p<\|\xi\|_p}}
d_\eta^2\,
\|\Delta^{(\alpha+d)/2}_\eta \sigma(\xi)\|^2_{op}\right)\,
\|\hat{f}(\xi)\|^2_{HS} \\
&\lesssim \sum_{[\xi]\in \widehat{\mathbb{G}}}
d_\xi \,\|\xi\|^{-(\alpha+d)}_p
\left( \sum_{\substack{[\eta]\in \widehat{\mathbb{G}}\\ \|\eta\|_p<\|\xi\|_p}}
d_\eta^2\right)\,
\|\hat{f}(\xi)\|^2_{HS} \\
&\leq \sum_{[\xi]\in \widehat{\mathbb{G}}}
d_\xi \,\|\xi\|^{-\alpha}_p\,\|\hat{f}(\xi)\|^2_{HS} = \|f\|^2_{H^{-\alpha/2}(\mathbb{G})},
\end{align*}
so, applying Lemma \ref{lemma3.1} we can conclude
\begin{align*}
\sum_{[\xi]\in \widehat{\mathbb{G}}}
\sum_{\substack{[\eta]\in \widehat{\mathbb{G}}\\ \|\eta\|_p<\|\xi\|_p}}
d_\xi d_\eta^2 \,\|\eta\|_p^{-(\alpha+d)}\,
\|\Delta_\eta \sigma(\xi)\|^2_{op}\,\|\hat{f}(\xi)\|^2_{HS}
\lesssim \|f\|^2_{L^2_\alpha(\mathbb{G})}.
\end{align*}
To conclude the proof, just apply again Lemma \ref{lemma3.1}:

\begin{align*}
\sum_{[\xi]\in \widehat{\mathbb{G}}}
\sum_{\substack{[\eta]\in \widehat{\mathbb{G}}\\ \|\eta\|_p \geq \|\xi\|_p}}
d_\xi d_\eta^2 \,\|\eta\|_p^{-(\alpha+d)}\,
\|\Delta_\eta \sigma(\xi)\|^2_{op}\,\|\hat{f}(\xi)\|^2_{HS}
\lesssim \| f \|_{H^{\alpha/2}_2 (\G)}^2 \lesssim  \|f\|^2_{L^2_\alpha(\mathbb{G})}.
\end{align*}
\end{proof}

\subsection{One more example: $\G_{5,2}$}

To conclude this section, let us present another example of a nilpotent group $\G$ where our theorems apply. Let $p > 2$ be a prime number. The group $\mathbb{G}_{5,2}(\mathbb{Z}_p)$, or simply $\mathbb{G}_{5,2}$ for brevity, is defined as $\mathbb{Z}_p^5$ equipped with the non-commutative operation
\[
\mathbf{x} \star \mathbf{y} := \left(
x_1 + y_1,\,
x_2 + y_2,\,
x_3 + y_3,\,
x_4 + y_4 + x_1 y_2,\,
x_5 + y_5 + x_1 y_3
\right),
\]where the inverse element $\mathbf{x}^{-1}$ of $\mathbf{x}$ with respect to this operation is given by
\[
\mathbf{x}^{-1} = \left(
- x_1,\,
- x_2,\,
- x_3,\,
- x_4 + x_1 x_2,\,
- x_5 + x_1 x_3
\right).
\]

We can identify this group with the exponential image of the $\mathbb{Z}_p$-Lie algebra $\mathfrak{g}_{5,2}$ defined by the commutation relations
\[
[X_1, X_2] = X_4, \qquad [X_1, X_3] = X_5.
\]

The unitary dual $\widehat{\mathbb{G}}_{5,2}$ of $\mathbb{G}_{5,2}$ can be identified with the following subset of $\widehat{\mathbb{Z}_p^5}$:
\[
\widehat{\mathbb{G}}_{5,2}
=
\left\{
\xi \in \widehat{\mathbb{Z}_p^5} \;:\;
\begin{aligned}
&|\xi_4|_p > |\xi_5|_p,\; (\xi_1,\xi_2) \in \mathbb{Q}_p^2 / p^{\vartheta(\xi_4)}\mathbb{Z}_p, \\
&\text{or } |\xi_4|_p \leq |\xi_5|_p,\; (\xi_1,\xi_3) \in \mathbb{Q}_p^2 / p^{\vartheta(\xi_5)}\mathbb{Z}_p
\end{aligned}
\right\}.
\]

Moreover, we can write
\[
\widehat{\mathbb{G}}_{5,2} = A_1 \cup A_2 \cup A_3,
\]
where
\begin{align*}
A_1 &:= \left\{ \xi \in \widehat{\mathbb{Z}_p^5} : \|(\xi_4,\xi_5)\|_p = 1 \right\}, \\
A_2 &:= \left\{ \xi \in \widehat{\mathbb{Z}_p^5} :
\|(\xi_4,\xi_5)\|_p > 1,\;
|\xi_4|_p > |\xi_5|_p,\;
(\xi_1,\xi_2) \in \mathbb{Q}_p^2 / p^{\vartheta(\xi_4)}\mathbb{Z}_p
\right\}, \\
A_3 &:= \left\{ \xi \in \widehat{\mathbb{Z}_p^5} :
\|(\xi_4,\xi_5)\|_p > 1,\;
|\xi_4|_p \leq |\xi_5|_p,\;
(\xi_1,\xi_3) \in \mathbb{Q}_p^2 / p^{\vartheta(\xi_5)}\mathbb{Z}_p
\right\}.
\end{align*}
See \cite{VelasquezRodriguez2025} for more details. For this group, each unitary irreducible representation can be realized in the finite-dimensional Hilbert space
\[
\mathcal{H}_\xi
:=
\mathrm{span}_{\mathbb{C}}
\left\{
\|(\xi_4,\xi_5)\|_p^{1/2}
\, \1_{h + p^{\vartheta(\xi_4,\xi_5)}\mathbb{Z}_p}
:\;
h \in \mathbb{Z}_p / p^{-\vartheta(\xi_4,\xi_5)}\mathbb{Z}_p
\right\},
\]
with
\[
d_\xi := \dim_{\mathbb{C}}(\mathcal{H}_\xi) = \|(\xi_4,\xi_5)\|_p.
\]

The representation acts according to the formula
\[
\pi_\xi(\mathbf{x})\varphi(u)
:=
e^{2\pi i \{\xi \cdot \mathbf{x} + (\xi_4 x_2 + \xi_5 x_3)u\}_p}
\, \varphi(u + x_1),
\]and with a natural choice of basis for $\mathcal{H}_\xi$, the associated matrix coefficients are given by
\[
\pi_\xi(\mathbf{x})_{hh'}
=
e^{2\pi i \{\xi \cdot \mathbf{x} + (\xi_4 x_2 + \xi_5 x_3)h'\}_p}
\, \1_{h'-h + p^{-\vartheta(\xi_4,\xi_5)}\mathbb{Z}_p}(x_1),
\]
and the associated characters are
\[
\chi_{\pi_\xi}(\mathbf{x})
=
\|(\xi_4,\xi_5)\|_p
\, e^{2\pi i \{\xi \cdot \mathbf{x}\}_p}
\, \1_{p^{-\vartheta(\xi_4,\xi_5)}\mathbb{Z}_p}(x_1)
\, \1_{\mathbb{Z}_p}(\xi_4 x_2 + \xi_5 x_3).
\]
Similar to $\mathbb{H}_d$ and $\mathcal{B}_4$, given any $\mathbf{x} \in \G_{5,1} $ and any column of $\pi_\xi$, there is only one element of the column that can be non-zero, thus we have two possibilities for the entries of $\pi_\xi - \mathrm{I}_{d_\xi}$. Either $\pi_\xi - \mathrm{I}_{d_\xi}$ is a diagonal matrix when $x_1 \in p^{-\vartheta(\xi_4, \xi_5)}\Z_p$, or only one of the $\pi_xi(\mathbf{x})_{hh'}$ is non-zero and the diagonal is simply constant $-1$, when $x_1 \in \Z_p \setminus p^{-\vartheta(\xi_4, \xi_5e)}\Z_p$. We can compute
\[
\|\pi_\xi(\mathbf{x}) - \mathrm{I}_{d_\xi}\|^2_{HS} =
\begin{cases}
2\| (\xi_4 , \xi_5)\|_p, & \text{if } x_1 \in \mathbb{Z}_p \setminus p^{-\vartheta(\xi_4, \xi_5)}\mathbb{Z}_p, \\[6pt]
\displaystyle \sum_{h' \in \mathbb{Z}_p / p^{-\vartheta(\xi_4, \xi_5)} \mathbb{Z}_p}
\Big| e^{2\pi i \{\xi \cdot \mathbf{x} + (\xi_4 x_2 + \xi_5 x_3)h'\}_p} - 1 \Big|^2, & 
\text{if } x_1 \in p^{-\vartheta(\xi_4, \xi_5)}\mathbb{Z}_p.
\end{cases}
\]

Here the roots of the unity $e^{2\pi i \{\xi \cdot \mathbf{x} + (\xi_4 x_2 + \xi_5 x_3)h'\}_p}$ cannot be all one, and actually \textbf{at most one of the eigenvalues of $\pi_\xi (\mathbf{x})$ can be equal to 1}. So in conclusion we have the estimate

$$\# \{ h \in \mathbb{Z}_p / p^{-\vartheta(\xi_4, \xi_5)} \mathbb{Z}_p \, :  \,\, \lambda_h (\pi_\xi (\mathbf{x})) \neq 1 \}= d_\xi -1 \geq \big( 1 -\frac{1}{p^{5}} \big)  d_\xi, $$
leading us to the conclusion

\begin{align*}
\sum_{ h' \in \mathbb{Z}_p / p^{-\vartheta(\xi_4, \xi_5)} \mathbb{Z}_p } 
\left| e^{2\pi i \{\xi \cdot \mathbf{x} + (\xi_4 x_2 + \xi_5 x_3)h'\}_p} - 1 \right|^2 
&\gtrsim \sum_{h' \in \mathbb{Z}_p / p^{\vartheta (\xi_4 , \xi_5)} \mathbb{Z}_p}
\|\xi\|_p^{-2(5)} \|\mathbf{x}\|_p^{-2(5)} \\
&\geq   \|\xi\|_p^{-2(5)} \|\mathbf{x}\|_p^{-2(5)} (\|(\xi_4, \xi_5 )\|_p - 1) \\
&\gtrsim \|\xi\|_p^{-2(5)} \|\mathbf{x}\|_p^{-2(5)} \| (\xi_4 , \xi_5) \|_p.
\end{align*}

\section{Multipliers on $L^r_\alpha (G)$, $1 <r< \infty$.}

In this section we collect the proofs of our main results in the general abstract setting of compact Vilenkin groups. Considering that we are dealing with weighted Lebesgue spaces, in what follows we denote
by $\mu_{\alpha}$ the measure defined as
\[
\mu_{\alpha}(A):=\int_{A}|x|_{\mathcal{G}}^{\alpha}\,dx,
\]
so that $L^{r}_{\alpha}(G):=L^{r}(G,\mu_{\alpha})$ for $1\leq r<\infty$. We collect some properties of this
measure in the following proposition. See \cite{Onneweer1985} for a detailed proof.

\begin{pro}\label{promeasurewighted}
    Let $\alpha>-1$, $x\in G$ and $n\in\mathbb{N}_{0}$.
\begin{itemize}
    \item[(i)] $\mu_{\alpha}(G_{k})\asymp |G_{k}|^{\alpha+1}$.
    \item[(ii)] If $\alpha\leq 0$ and $(x \smallstar G_{k})^{*}:=x \smallstar G_{k}\setminus\{e\}$, then
    \[
    \mu_{\alpha}(g \smallstar G_{k})\lesssim |G_{k}|\inf\{|y|^{\alpha}:y\in x\smallstar G^{*}_{k}\}.
    \]
    \item[(iii)] $\mu_{\alpha}(x \smallstar G_{k})\lesssim \mu_{\alpha}(x \smallstar G_{k+1})$.
    \item[(iv)] $\mu_{\alpha}(G_{k})\asymp \mu_{\alpha}(G_{k}\setminus G_{k+1})$.
\end{itemize}
\end{pro}

Next we need to make use of the following two propositions. The first is a Calderón–Zygmund
decomposition for functions in $L^{1}_{\alpha}$, and the second is duality theorem for multipliers on
$L^{r}_{\alpha}$–spaces. The proof of both propositions can be found in \cite{Onneweer1985}, where we refer the reader
for more details.

\begin{pro}\label{proCalderonZygmund}
    Let $-1<\alpha\leq 0$, let $\varphi\in L^{1}_{\alpha}(G)$, and let $\gamma>0$ be given. Then there exist functions $\{\varphi_{j}\}_{j\in\mathbb{N}_{0}}$ such that
\begin{itemize}
    \item[(i)] $\varphi=\sum_{j}\varphi_{j}$,
    \item[(ii)] $\varphi_{j}\in L^{1}_{\alpha}(G)$ for each $j\geq 0$,
    \item[(iii)] $\sum_{j}\|\varphi_{j}\|_{L^{1}_{\alpha}(G)}\lesssim \|\varphi\|_{L^{1}_{\alpha}(G)}$,
    \item[(iv)] $|\varphi_{0}(x)|\lesssim \gamma$,
    \item[(v)] There exist disjoint sets $I_{j}=g_{j} \smallstar G_{m(j)}=G_{m(j)}\smallstar g_{j}$ such that $\mathrm{supp}(\varphi_{j})\subset I_{j}$ for $j\in\mathbb{N}$, and we use the notation $D_{\gamma}:=\bigcup_{n\in\mathbb{N}_{0}}I_{j}$,
    \item[(vi)] $\sum_{j}\mu_{\alpha}(I_{j})\leq \gamma^{-1}\|\varphi\|_{L^{1}_{\alpha}(G)}$,
    \item[(vii)] $\displaystyle\int_{I_{j}}\varphi_{j}(x)\,dx=0$ for $j\in\mathbb{N}$.
\end{itemize}

\end{pro}

\begin{pro}\label{prointerpolationDual}
    Let $\sigma\in L^{\infty}_{\mathrm{op}}(\widehat{G})$, $1<r<\infty$, $-1<\alpha<r-1$. Assume that there exists $C>0$ so that for all $f\in\mathcal{D}(G)$ we have
\[
\|T_{\sigma}f\|_{L^{r}_{\alpha}(G)}\leq C\|f\|_{L^{r}_{\alpha}(G)}.
\]
Then we have for all $f\in\mathcal{D}(G)$ and with the same constant $C$
\[
\|T_{\sigma}f\|_{L^{r'}_{(1-r')\alpha}(G)}\leq C\|f\|_{L^{r'}_{(1-r')\alpha}(G)}.
\]
Thus $T_{\sigma}$ is a multiplier on $L^{r}_{\alpha}(G)$ if and only if $T_{\sigma}$ is a multiplier on
$L^{r'}_{(1-r')\alpha}(G)$.
\end{pro}

In order to proceed, it is time to prove how we can use condition $H(t)$ to prove the boundedness
of a Fourier multiplier on $L^{r}_{\alpha}(G)$. The sketch of the idea is the following:

\begin{itemize}
    \item[(i)] The boundedness of our operators comes from an interpolation argument.
    \item[(ii)] The argument is: if $T_\sigma$ is bounded on $L^2_\alpha (G)$ and bounded from $L^{1,w}_\alpha (G)$ to $L^{1}_\alpha (G)$, then it has to be bounded on $L^r_\alpha (G)$. This is a consequence of the interpolation Theorem \ref{teoMarcinkiewicz}.   
    \item[(iii)] For the boundedness from $L^{1,w}_\alpha (G)$ to $L^{1}_\alpha (G)$ we use condition $H(t)$ and the Calderon-Zygmund decomposition given in Proposition \ref{proCalderonZygmund}.
    \item[(iv)] For the boundedness on $L^2_\alpha (G)$ we use Theorem \ref{teomultL2v1}. 
    \item[(v)]  The Stein-Weiss interpolation given in Theorem \ref{teoSteinWeiss} proves the boundedness on $L^r_\alpha (G)$ for $1<r<2$, provided the boundedness on the endpoint spaces. 
    \item[(vi)] The standard duality argument given in Proposition \ref{prointerpolationDual} proves the boundedness for the case $2<r<\infty$.    
\end{itemize}

Our next goal is to derive the weak $(1,1)$ estimate by means of condition $H(t)$ and the Calderón--Zygmund decomposition. This will be accomplished in Lemma \ref{lemaweak11}. We begin by introducing notation that will make the subsequent arguments more transparent.
\begin{rem}
    Let $\sigma\in L^{\infty}_{\mathrm{op}}(G)$.  and for $k\in\mathbb{Z}$ let $\sigma_{k}:=\widehat{\epsilon_k }\sigma$. We are using the same notation as in Definition \ref{conditionH(t)}.  Since $\mathcal{D}(G)$ equals the space of trigonometric polynomials on $G$, then $T_{\sigma_{k}}f=T_{\sigma}f$ for $k$ large enough. Thus to prove the boundedness of $T_{\sigma}$ on $L^{r}_{\alpha}$–spaces or weak $L^{r}_{\alpha}$–spaces is equivalent to prove the uniform boundedness of the operators $T_{\sigma_{k}}$. This is a fact that we will exploit along our exposition. 
\end{rem}
Now we are ready for our lemma: 

\begin{lema}\label{lemaweak11}
    Let $\sigma\in L^{\infty}_{\mathrm{op}}(G)$ and assume that $\sigma$ satisfies the condition $H(t)$ for some $t\geq 1$. If $T_{\sigma}$ is of type $(2,2)$ on $L^{2}_{\alpha}(G)$ for some $\alpha$ with $-\frac{1}{t'}<\alpha\leq 0$, then $T_{\sigma}$ is of weak type $(1,1)$ on $L^{1}_{\alpha}(G)$.
\end{lema}

\begin{proof}
Take any $\varphi \in L_{\alpha}^{1}(G)$. Fix $\gamma>0$ and apply Proposition 4.2 to get a Calderon-Zygmund decomposition for $\varphi$. Let us write
\[
\varphi=\varphi_{0}+\sum_{j \in \mathbb{N}} \varphi_{j}=\varphi_{0}+\psi.
\]
Then
\[
\{x \in G:|T_{\sigma_{k}} \varphi(x)|>\gamma\} \subset\{x \in G:|T_{\sigma_{k}} \varphi_{0}(x)|>\gamma / 2\} \cup\{|T_{\sigma_{k}} \psi(x)|>\gamma / 2\}:=E_{\gamma} \cup F_{\gamma}.
\]
For $E_{\gamma}$, using the boundedness of $T_\sigma$ on $L^2_\alpha (G)$, we have
\begin{align*}
\mu_{\alpha}\left(E_{\gamma}\right)&=\mu_{\alpha}\left(\left\{x \in G:2 \gamma^{-1}\left|T_{\sigma_{k}} \varphi_{0}(x)\right|>1\right\}\right)
\leq 4 \gamma^{-2}\left\|T_{\sigma_{k}} \varphi_{0}\right\|_{L_{\alpha}^{2}(G)}^{2} \\ & \lesssim 4 \gamma^{-2}\left\|\varphi_{0}\right\|_{L_{\alpha}^{2}(G)}^{2} \lesssim \gamma^{-1}\|\varphi_{0}\|_{L_{\alpha}^{1}(G)}.
\end{align*}

Next we observe that, using the set $D_\gamma$ defined in Proposition \ref{proCalderonZygmund},
\begin{align*}
\mu_{\alpha}\left(F_{\gamma}\right)&=\mu_{\alpha}\left(F_{\gamma} \cap D_{\gamma}\right)+\mu_{\alpha}\left(F_{\gamma} \backslash D_{\gamma}\right)
\\ &\leq \mu_{\alpha}\left(D_{\gamma}\right)+2 \gamma^{-1} \int_{G \backslash D_{\gamma}}\left|T_{\sigma_{k}} \psi(x)\right| |x|^{\alpha} d x
\leq \gamma^{-1}\|\varphi\|_{L_{\alpha}^{1}(G)}+2 \gamma^{-1} \int_{G \backslash D_{\gamma}}\left|T_{\sigma_{k}} \psi(x)\right||x|^{\alpha} d x,
\end{align*}
where in the above estimation we used the properties of the Calderon--Zygmund decomposition from Proposition \ref{proCalderonZygmund} and the hypothesis $2 \gamma^{-1} \left|T_{\sigma_{k}} \psi(x)\right|>1$.

Recall now the sequence $\{ g_j\}_{j \in \N}$ such that $I_j = g_j \smallstar G_{m(j)}$. For the previous integral we have:
\[
\begin{aligned}
\int_{G \backslash D_{\gamma}}|T_{\sigma_{k}} \psi(x)|&|x|^{\alpha} d x 
= \int_{G \backslash D_{\gamma}} \left|\sum_{j \in \mathbb{N}} \int_{I_{j}} \mathcal{F}_{G}^{-1} [\sigma_{k}](x \smallstar y^{-1}) \varphi_{j}(y) d y \right|\, |x|^{\alpha} d x \\
&= \int_{G \backslash D_{\gamma}} \left|\sum_{j \in \mathbb{N}} \int_{I_{j}} \left(\mathcal{F}_{G}^{-1}[\sigma_{k}](x \smallstar y^{-1})-\mathcal{F}_{G}^{-1}[\sigma_{k}](x  \smallstar g_j^{-1})\right)\varphi_{j}(y) d y \right|\, |x|^{\alpha} d x \\ &\leq  \int_{G \backslash D_{\gamma}} \sum_{j \in \mathbb{N}} \int_{I_{j}} \left| \mathcal{F}_{G}^{-1}[\sigma_{k}](x \smallstar y^{-1})-\mathcal{F}_{G}^{-1}[\sigma_{k}](x  \smallstar g_j^{-1})\right| | \varphi_{j}(y)|  d y \, |x|^{\alpha} d x \\ & = \sum_{j \in \mathbb{N}}   \int_{I_{j}} | \varphi_{j}(y)| \int_{G \backslash D_{\gamma}} \left| \mathcal{F}_{G}^{-1}[\sigma_{k}](x \smallstar y^{-1})-\mathcal{F}_{G}^{-1}[\sigma_{k}](x  \smallstar g_j^{-1} )\right|   |x|^{\alpha} d x dy \\ & \leq \sum_{j \in \mathbb{N}}   \int_{I_{j}} | \varphi_{j}(y)| \int_{G \backslash I_j}  \left| \mathcal{F}_{G}^{-1}[\sigma_{k}](x \smallstar y^{-1})-\mathcal{F}_{G}^{-1}[\sigma_{k}](x \smallstar g_j^{-1} )\right|   |x|^{\alpha} d x dy \\ & = \sum_{j \in \mathbb{N}}   \int_{G_{m(j)}} | \varphi_{j}(y \smallstar g_j)| \int_{G \backslash G_{m(j)}}  \left| \mathcal{F}_{G}^{-1}[\sigma_{k}](x \smallstar y^{-1})-\mathcal{F}_{G}^{-1}[\sigma_{k}](x )\right|   |x \smallstar g_j|^{\alpha} d x dy.
\end{aligned}
\]

Now we have to estimate the inside integral in the above expression, let us denote it by $K_{j,\alpha}$. If $\alpha=0$ and the condition $H(1)$ holds then
\[
\int_{G \backslash G_{m(j)}}\left|\mathcal{F}_{G}^{-1}[\sigma_{k}](x \smallstar y^{-1})-\mathcal{F}_{G}^{-1}[\sigma_{k}](x)\right| |x \smallstar g_{j}|^{\alpha} d x \lesssim 1,
\]
so we can conclude
\[
\int_{G \backslash D_{\gamma}}\left|T_{\sigma_{k}} \psi(x)\right| |x|^{\alpha} d x \lesssim \sum_{j \in \mathbb{N}} \int_{G_{m(j)}}\left|\varphi_{j}\left(y \smallstar g_{j}\right)\right| d y \lesssim \|\varphi\|_{L^{1}(G)}.
\]

If $\alpha<0$, let $k(j)$ be the natural number such that $|g_j|=|G/G_{k(j)}|^{-1}$. Let us define the sets $$R_j^1 := \{ x \in G \setminus G_{m(j)} \, : \, x \smallstar g_j \in G_{k(j) +1} \}, \quad R_j^2 := \{ x \in G \setminus G_{m(j)} \, : \, x \smallstar g_j \in G \setminus G_{k(j) +1} \}.$$First of all, for $x\in R_j^2$, we either have $$|x \smallstar g_{j}| = |x|\geq|g_j|, \quad \text{or} \quad |x \smallstar g_{j}|= |g_j|>|x|.$$In any case $|x \smallstar g_{j}|^\alpha \leq |g_j|^\alpha = |y \smallstar g_j|^\alpha$. For $R_j^2$, notice how $R_j^2 \subset G_{k(j)} \setminus G_{k(j)+1},$ therefore \begin{align*}
    \int_{R_j^2} & \left|  \mathcal{F}_{G}^{-1}[\sigma_{k}](x \smallstar y^{-1})-\mathcal{F}_{G}^{-1}[\sigma_{k}](x)\right|  |x \smallstar g_{j}|^{\alpha} d x  \leq  \\ & \Big( \int_{G \setminus G_{m(j)}}\left|\mathcal{F}_{G}^{-1}[\sigma_{k}](x \smallstar y^{-1})-\mathcal{F}_{G}^{-1}[\sigma_{k}](x)\right|^t d x\Big)^{1/t} \Big( \int_{G_{k(j)} \setminus G_{k(j)+1}} |x \smallstar g_{j}|^{t' \alpha} dx  \Big)^{1/t'} .
\end{align*}
Recall how, by Proposition \ref{promeasurewighted}, for $\alpha t' > -1$ $$\Big( \int_{G_{k(j)} \setminus G_{k(j)+1}} |x \smallstar g_{j}|^{t' \alpha} dx  \Big)^{1/t'} \lesssim  |G_{k(j)}|^{1/t'}\inf\{|y|^{\alpha}:y\in g_j \smallstar G^{*}_{k(j)}\} ,$$
and, assuming here the condition $H(t)$ holds, $t>1$, we have:
\begin{align*}
    \int_{R_j^1} \left|\mathcal{F}_{G}^{-1}[\sigma_{k}](x \smallstar y^{-1}) -\mathcal{F}_{G}^{-1}[\sigma_{k}](x) \right| &d x  \leq \int_{G \setminus G_{m(j)}} \left|\mathcal{F}_{G}^{-1}[\sigma_{k}](x \smallstar y^{-1})-\mathcal{F}_{G}^{-1}[\sigma_{k}](x)\right|d x \\ & \leq \sum_{k=0}^{m(j) -1 } \Big( \int_{G_k \setminus G_{k+1}} \left|\mathcal{F}_{G}^{-1}[\sigma_{k}](x \smallstar y^{-1})-\mathcal{F}_{G}^{-1}[\sigma_{k}](x)\right|^td x \Big)^{1/t} |G_k|^{1/t'} \\ &\lesssim \sum_{k=0}^{m(j) -1 } |G_k|^{- \varepsilon - \frac{1}{t'}} |G_{m(j)}|^{\varepsilon} |G_k|^{1/t'} \lesssim 1.
\end{align*}

In total we have:

\begin{align*}
K_{j,\alpha} &\lesssim |y \smallstar g_j|^\alpha \int_{R_j^1} \left|\mathcal{F}_{G}^{-1}[\sigma_{k}](x \smallstar y^{-1})-\mathcal{F}_{G}^{-1}[\sigma_{k}](x)\right|d x \\  &+ \left(\int_{G_{k(j)} \backslash G_{k(j)+1}} \left|\mathcal{F}_{G}^{-1}[\sigma_{k}](x \smallstar y^{-1})-\mathcal{F}_{G}^{-1}[\sigma_{k}](x)\right|^{t} d x \right) \left(\int_{G_{k(j)} \backslash G_{k(j)+1}}|x \smallstar g_{j}|^{\alpha t^{\prime}} d x\right)^{1 / t^{\prime}}  \\ &\lesssim |y \smallstar g_j|^\alpha + |G_{k(j)}|^{- \varepsilon - \frac{1}{t'}} |G_{m(j)}|^{\varepsilon} \inf\{|y|^{\alpha}:y\in g_j \smallstar G^{*}_{k(j)}\} \lesssim |y \smallstar g_j|^\alpha.
\end{align*}

Summing up:
\[
\int_{G \backslash D_{\gamma}}\left|T_{\sigma_{k}} \psi(x)\right| |x|^{\alpha} d x \lesssim \sum_{j \in \mathbb{N}} \int_{G_{m(j)}}\left|\varphi_{j}\left(y \smallstar g_{j}\right)\right| |y \smallstar g_j|^\alpha d y
\lesssim \|\varphi\|_{L_{\alpha}^{1}(G)}.
\]
This concludes the proof. 
\end{proof}

Along the last steps of the above proof, we tacitly established the following corollary.

\begin{coro}\label{coroalternativeestimate}
    The same conclusion holds if we replace condition $H(t)$ by
    \begin{equation}\label{eqaternativeLrproof}
        \sup\Big\{\left(\int_{G \setminus G_{l}} \left|\mathcal{F}_{G}^{-1}[ \sigma_{k}](x \smallstar y^{-1})-\mathcal{F}_{G}^{-1}[\sigma_{k}](x) \right|^{t} d x \right)^{1/t} \, : \, y \in G_l \Big\} \leq C .
    \end{equation}
\end{coro}

\begin{proof}
Following the notation of the previous proof, notice how we reached estimate
    \begin{align*}
K_{j,\alpha} &\lesssim |y \smallstar g_j|^\alpha \int_{R_j^1} \left|\mathcal{F}_{G}^{-1}[\sigma_{k}](x \smallstar y^{-1})-\mathcal{F}_{G}^{-1}[\sigma_{k}(x)]\right|d x \\  &+ \left(\int_{R_j^2} \left|\mathcal{F}_{G}^{-1}[\sigma_{k}](x \smallstar y^{-1})- \mathcal{F}_{G}^{-1} [\sigma_{k}](x)\right|^{t} d x \right)^{1/t} \left(\int_{G_{k(j)} \backslash G_{k(j)+1}}|x \smallstar g_{j}|^{\alpha t^{\prime}} d x\right)^{1 / t^{\prime}}  \\ &\lesssim |y \smallstar g_j|^\alpha \Big(  \int_{R_j^1} \left|\mathcal{F}_{G}^{-1}[\sigma_{k}](x \smallstar y^{-1})- \mathcal{F}_{G}^{-1}[ \sigma_{k}](x)\right|^t d x \Big)^{1/t} \\  &+ \left(\int_{R_j^2} \left|\mathcal{F}_{G}^{-1}[\sigma_{k}](x \smallstar y^{-1})- \mathcal{F}_{G}^{-1}[\sigma_{k}](x)\right|^{t} d x \right)^{1/t}|G_{k(j)}|^{1/t'} \inf\{|y|^{\alpha}:y\in g_j \smallstar G^{*}_{k(j)}\} \\ & \lesssim \left(\int_{G \setminus G_{m(j)}} \left|\mathcal{F}_{G}^{-1}[\sigma_{k}](x \smallstar y^{-1})- \mathcal{F}_{G}^{-1}[\sigma_{k}](x)\right|^{t} d x \right)^{1/t} |y \smallstar g_j|^\alpha \lesssim |y \smallstar g_j|^\alpha.
\end{align*}
This concludes the proof.
\end{proof}

\begin{rem}
    Notice how condition $H(t)$ implies the condition of the above corollary. 
\end{rem}
As a corollary of Lemma \ref{lemaweak11}, by using standard interpolation arguments, we obtain the following result:

\begin{teo}
Let $\sigma \in L^\infty_{\mathrm{op}}(G)$. Assume that $\sigma$ satisfies condition $H(t)$ for some $t \geq 1$, and that $T_\sigma$
is bounded on $L^2_{\alpha_0}(G)$ for some $-1/t' < \alpha_0 < 1/t'$. Then the Fourier multiplier $T_\sigma$ is bounded on
$L^r_\alpha(G)$ for $1 < r < \infty$ and $-|\alpha_0|\leq \alpha \leq (r-1) \cdot |\alpha_0|$. If the condition $H(1)$ holds then $T_\sigma$ is bounded
on $L^r(G)$ for $1 < r < \infty$.
\end{teo}

\begin{proof}

\quad

\textbf{Case 1: $1<r<2$, $- |\alpha_0| \leq \alpha \leq 0$.}  For $\sigma \in
L^\infty_{\mathrm{op}}(\widehat{G})$ we have by hypothesis hat $T_\sigma$ is $L^2_{\alpha_0}$-bounded, so by Stein’s interpolation theorem and Proposition \ref{prointerpolationDual}, 
$T_\sigma$ is a multiplier on $L^2_\alpha(G)$ for $-|\alpha_0|\leq \alpha \leq |\alpha_0|$. If additionally $\sigma$ satisfies condition $H(t)$, then
by Lemma \ref{lemaweak11} $T_\sigma$ is of weak type $(1,1)$ on $L^1_\alpha(G)$ for $-|\alpha_0| \leq  \alpha \leq 0$. Therefore, by Stein–Weiss interpolation, $T_\sigma$ is bounded on $L^r_\alpha(G)$ for $-|\alpha_0| \leq  \alpha \leq 0$.

\textbf{Case 2: $2<r<\infty$, $0 \leq \alpha \leq (r-1) \cdot|\alpha_0|$.} Follows directly from Proposition \ref{prointerpolationDual} and the \textbf{Case 1}. The reason is that $1<r'<2$, so that $$T_\sigma \in \mathcal{L}(L^{r'}_{-\alpha}(G)) \iff T_\sigma \in \mathcal{L}(L^{r}_{(r-1)\alpha}(G)).$$

\textbf{Case 3: $1<r<2$, $0 \leq \alpha \leq (r-1) \cdot  |\alpha_0|$.} For the case $0 < \alpha \leq (r-1) \cdot |\alpha_0|$, choose $r_0$ such that 
\[
1 < r_0 < \frac{r|\alpha_0| - 2\alpha}{|\alpha_0| - |\alpha|} \leq r <2, \qquad 
\alpha_1 = \frac{2\alpha - r_0\alpha}{r-r_0} , \qquad 
\theta = \frac{2r - 2r_0}{2r - rr_0}.
\]If $\sigma$ satisfies condition $H(t)$, then in particular it satisfies condition $H(1)$, so we can deduce $T_\sigma \in \mathcal{L}(L^r (G))$ from \textbf{Case 1}. Also, since $0 <\frac{r}{2} \alpha_1< \alpha_1 \leq |\alpha_0|$, we know that $T_\sigma$ is bounded on $L^2_{( \frac{r}{2} \alpha_1)}(G)$, and since 
\[
\frac{1}{r} = \frac{1-\theta}{r_0} + \frac{\theta}{2}, \qquad 
\alpha = \frac{ r }{2}\alpha_1\theta,
\]
it follows from Theorem \ref{teoSteinWeiss} that $T_\sigma$ is bounded on $L^r_\alpha(G)$.

\textbf{Case 4: $2<r<\infty$, $-|\alpha_0| \leq \alpha \leq 0$.} Once again we use Proposition \ref{prointerpolationDual}. From \textbf{Case 3} we get $$T_\sigma \in \mathcal{L}(L^{r'}_{(1-r') \alpha }(G)) \iff T_\sigma \in \mathcal{L}(L^{r}_{ \alpha }(G)).$$ 

This concludes the proof. 
\end{proof}

\section{Multipliers on $L^r_\alpha (\G)$, $1<r<\infty$, $-d \leq \alpha \leq (r-1)d$.}
Now it is finally time to prove our results on compact $p$-adic Lie groups. For starters, let us point to the fact that the crucial estimate we need for the $L^r_\alpha$-boundedness, given in Corollary \ref{coroalternativeestimate}, follows easily from the equivalence
\begin{equation}\label{eqequimultiplier}
    \|\mathbf{x}\|_p^\alpha
 \asymp
\Tilde{I}_\alpha(\mathbf{x})
:=
\sum_{\eta \in \widehat{\G},\,
\|\eta\|_p \le \|\mathbf{y}\|_p^{-1}}
d_\eta 
\|\eta\|_p^{-(\alpha+d)}
\,
\|\pi_\eta(\mathbf{x})-\mathrm{I}_{d_\eta}\|_{HS}^2,
\qquad
\|\mathbf{x}\|_p > \|\mathbf{y}\|_p.
\end{equation}

Using this fact we intend now to prove how the hypothesis in Theorem \ref{TeoMultiplierLr} implies the necessary conditions to apply Theorem \ref{teoL2impliesLr}.
\begin{proof}[Proof of Theorem \ref{TeoMultiplierLr}]
We simply show how the hypothesis of Lemma \ref{lemaweak11} and Corollary \ref{coroalternativeestimate} hold. To simplify our notation along the proof, let us write
\[
\check{\sigma}_{k,\mathbf{y}}(\mathbf{x})
:=
\mathcal{F}^{-1}_{\G}[\sigma_k](\mathbf{x}\smallstar \mathbf{y}^{-1})
-
\mathcal{F}^{-1}_{\G}[\sigma_k](\mathbf{x}),
\] 

\textbf{Part 1 :}First, apply the Cauchy-Schwartz inequality:

\begin{align*}
    \int_{\G \setminus \G (p^\ell \Z_p)}|\check{\sigma}_{k,\mathbf{y}}(\mathbf{x})| d \mathbf{x} & \leq \Big(\int_{\G \setminus \G (p^\ell \Z_p)}|\check{\sigma}_{k,\mathbf{y}}(\mathbf{x})|^2 \| \mathbf{x}\|_p^{2 \alpha_0} d \mathbf{x} \Big)^{1/2} \Big( \int_{\G \setminus \G (p^\ell \Z_p)}\| \mathbf{x} \|_p^{-2 \alpha_0} d \mathbf{x} \Big)^{1/2}.
\end{align*}In one hand 
\begin{align*}
    \Big( \int_{\G \setminus \G (p^\ell \Z_p)}\| \mathbf{x} \|_p^{-2 \alpha_0} d \mathbf{x} \Big)^{1/2} = \Big( \sum_{k=0}^{ \ell -1} p^{-k(d- 2\alpha_0)}(1-p^{-d}) \Big)^{1/2} \lesssim p^{\ell(\alpha_0 - \frac{d}{2})} \leq \| \mathbf{y} \|_p^{ \frac{d}{2} - \alpha_0},
\end{align*}for any $\mathbf{y} \in \G (p^\ell \Z_p)$.  On the other hand, with the equivalence given in Equation \ref{eqequimultiplier} we obtain: 
\[
\int_{\G \setminus \G(p^\ell \mathbb{Z}_p)}
|\check{\sigma}_{k,\mathbf{y}}(\mathbf{x})|^2
\|\mathbf{x}\|_p^{2 \alpha_0 }
\,d\mathbf{x}
\lesssim
\int_{\G}
|\check{\sigma}_{k,\mathbf{y}}(\mathbf{x})|^2
\left(
\sum_{\|\eta\|_p \le \|\mathbf{y}\|_p^{-1}}
d_\eta 
\|\eta\|_p^{-(2 \alpha_0 +d)}
\|\pi_\eta(\mathbf{x})-\mathrm{I}_{d_\eta} \|_{HS}^2
\right)
\,d\mathbf{x}
\]
\[
=
\sum_{\|\xi\|_p > \|\mathbf{y}\|_p^{-1}}
\sum_{\|\eta\|_p \le \|\mathbf{y}\|_p^{-1}}
d_\xi d_\eta 
\|\eta\|_p^{-(2 \alpha_0 +d)}
\,
\|\Delta_\eta(\widehat{f}_{\mathbf{y}}\sigma_k)(\xi)\|_{HS}^2,
\]
where we used the notation
\[
\widehat{f}_{\mathbf{y}}(\xi)
:=
\pi_\xi(\mathbf{y}^{-1})-\mathrm{I}_{d_\xi},
\]
so that
\[
\Delta_\eta \widehat{f}_{\mathbf{y}}(\xi)
=
\pi_\eta\otimes\pi_\xi(\mathbf{y}^{-1})
-
\mathrm{I}_{d_\eta} \otimes\pi_\xi(\mathbf{y}^{-1})
=
0_{d_\eta\times d_\xi},
\qquad
\text{for } \|\eta\|_p \le \|\mathbf{y}\|_p^{-1}.
\]

By the product rule,
\[
\Delta_\eta(\widehat{f}_{\mathbf{y}}\sigma_k)(\xi)
=
\widehat{f}_{\mathbf{y}}(\eta\otimes\xi)\Delta_\eta\sigma_k(\xi)
+
\Delta_\eta\widehat{f}_{\mathbf{y}}(\xi)
(\mathrm{I}_{d_\eta}\otimes\sigma_k(\xi)).
\]
Since $\Delta_\eta\widehat{f}_{\mathbf{y}}(\xi)=0$ for
$\|\eta\|_p \le \|\mathbf{y}\|_p^{-1}$, we obtain
\[
\Delta_\eta(\widehat{f}_{\mathbf{y}}\sigma_k)(\xi)
=
\widehat{f}_{\mathbf{y}}(\eta\otimes\xi)\Delta_\eta\sigma_k(\xi),
\]
and therefore
\[
\|\Delta_\eta(\widehat{f}_{\mathbf{y}}\sigma_k)(\xi)\|_{HS}
\le
2\|\Delta_\eta\sigma_k(\xi)\|_{HS}.
\]

Summing up, for $\mathbf{y} \in \G(p^\ell \Z_p )$
\begin{align*}    
\int_{\G \setminus \G(p^\ell \mathbb{Z}_p)}
|\check{\sigma}_{k,\mathbf{y}}(\mathbf{x})|^2
\|\mathbf{x}\|_p^{2 \alpha_0}
\,d\mathbf{x}
&\lesssim
\sum_{\|\xi\|_p > \|\mathbf{y}\|_p^{-1}}
\sum_{\|\eta\|_p \le \|\mathbf{y}\|_p^{-1}}
d_\xi d_\eta 
\|\eta\|_p^{-(2\alpha_0+d)}
\|\Delta_\eta\sigma_k(\xi)\|_{HS}^2
\\&\leq  \sum_{\|\xi\|_p > \|\mathbf{y}\|_p^{-1}}
\sum_{\|\eta\|_p \le \|\mathbf{y}\|_p^{-1}}
d_\xi^2 d_\eta^2 
\|\Delta_\eta^{\alpha_0 + \frac{d}{2}}\sigma_k(\xi)\|_{op}^2 \\ & \lesssim \sum_{\|\xi\|_p > \|\mathbf{y}\|_p^{-1}}
\sum_{\|\eta\|_p \le \|\mathbf{y}\|_p^{-1}}
d_\xi^2 d_\eta^2 
\|\xi \|_p^{-2\alpha_0 -d} \\ &\lesssim \|\mathbf{y}\|_p^{-d} \sum_{\|\xi\|_p > \|\mathbf{y}\|_p^{-1}} d_\xi^2  
\|\xi \|_p^{-2\alpha_0 -d}  \lesssim \|\mathbf{y}\|_p^{2 \alpha_0 -d} .
\end{align*}
This concludes the proof.

\textbf{Part 2:} Assume $1<t<2$. First, apply the H{\"o}lder inequality:

\begin{align*}
    \Big( \int_{\G \setminus \G (p^\ell \Z_p)}|\check{\sigma}_{k,\mathbf{y}}(\mathbf{x})|^t d \mathbf{x} \Big)^{1/t} & \leq \Big(\int_{\G \setminus \G (p^\ell \Z_p)}|\check{\sigma}_{k,\mathbf{y}}(\mathbf{x})|^2 \| \mathbf{x}\|_p^{2 \alpha_0} d \mathbf{x} \Big)^{1/2} \Big( \int_{\G \setminus \G (p^\ell \Z_p)}\| \mathbf{x} \|_p^{-\frac{2t \alpha_0}{2-t}} d \mathbf{x} \Big)^{\frac{2-t}{2t}}.
\end{align*}In one hand 
\begin{align*}
    \Big( \int_{\G \setminus \G (p^\ell \Z_p)}\| \mathbf{x} \|_p^{-\frac{2t \alpha_0}{2-t}} d \mathbf{x} \Big)^{\frac{2-t}{2t}} = \Big( \sum_{k=0}^{ \ell -1} p^{-k(d- \frac{2t \alpha_0}{2-t})}(1-p^{-d}) \Big)^{\frac{2-t}{2t}} \lesssim p^{\ell(\alpha_0 - d(\frac{1}{t}- \frac{1}{2}))} \leq \| \mathbf{y} \|_p^{ d(\frac{1}{t}- \frac{1}{2}) - \alpha_0},
\end{align*}for any $\mathbf{y} \in \G (p^\ell \Z_p)$.  On the other hand, with the equivalence given in Equation \ref{eqequimultiplier} we obtain again: 
\[
\int_{\G \setminus \G(p^\ell \mathbb{Z}_p)}
|\check{\sigma}_{k,\mathbf{y}}(\mathbf{x})|^2
\|\mathbf{x}\|_p^{2 \alpha_0 }
\,d\mathbf{x}
\lesssim
\int_{\G}
|\check{\sigma}_{k,\mathbf{y}}(\mathbf{x})|^2
\left(
\sum_{\|\eta\|_p \le \|\mathbf{y}\|_p^{-1}}
d_\eta 
\|\eta\|_p^{-(2 \alpha_0 +d)}
\|\pi_\eta(\mathbf{x})-\mathrm{I}_{d_\eta} \|_{HS}^2
\right)
\,d\mathbf{x}
\]
\[
=
\sum_{\|\xi\|_p > \|\mathbf{y}\|_p^{-1}}
\sum_{\|\eta\|_p \le \|\mathbf{y}\|_p^{-1}}
d_\xi d_\eta 
\|\eta\|_p^{-(2 \alpha_0 +d)}
\,
\|\Delta_\eta(\widehat{f}_{\mathbf{y}}\sigma_k)(\xi)\|_{HS}^2,
\]
so for $\mathbf{y} \in \G(p^\ell \Z_p )$, let us say $\| \mathbf{y} \|_p= p^{-\vartheta(\mathbf{y})}$
\begin{align*}    
\int_{\G \setminus \G(p^\ell \mathbb{Z}_p)}
|\check{\sigma}_{k,\mathbf{y}}(\mathbf{x})|^2
\|\mathbf{x}\|_p^{2 \alpha_0}
\,d\mathbf{x}
&\lesssim
\sum_{\|\xi\|_p > \|\mathbf{y}\|_p^{-1}}
\sum_{\|\eta\|_p \le \|\mathbf{y}\|_p^{-1}}
d_\xi d_\eta 
\|\eta\|_p^{-(2\alpha_0+d)}
\|\Delta_\eta\sigma_k(\xi)\|_{HS}^2
\\&\leq  \sum_{\|\xi\|_p > \|\mathbf{y}\|_p^{-1}}
\sum_{\|\eta\|_p \le \|\mathbf{y}\|_p^{-1}}
d_\xi^2 d_\eta^2 
\|\Delta_\eta^{\alpha_0 + \frac{d}{2}}\sigma_k(\xi)\|_{op}^2 \\ & \lesssim \sum_{\|\xi\|_p > \|\mathbf{y}\|_p^{-1}}
\sum_{\|\eta\|_p \le \|\mathbf{y}\|_p^{-1}}
d_\xi^2 d_\eta^2 \|\xi\|_p^{-2\alpha_0-2d(
    \frac{3}{2}-\frac{1}{t})} \\ &\lesssim \|\mathbf{y}\|_p^{-d} \sum_{\|\xi\|_p > \|\mathbf{y}\|_p^{-1}} d_\xi^2  
\|\xi \|_p^{-2\alpha_0-2d(
    \frac{3}{2}-\frac{1}{t})} \\ &= \|\mathbf{y}\|_p^{-d}  \sum_{k= \vartheta(\mathbf{y})+1}^{\infty}  p^{k(-2\alpha_0-2d(
    \frac{3}{2}-\frac{1}{t}))} \sum_{\| \xi \|_p = p^k} 
    d_\xi^2 \\ & \leq \|\mathbf{y}\|_p^{-d}  \sum_{k= \vartheta(\mathbf{y})+1}^{\infty}  p^{k(d -2\alpha_0-2d(
    \frac{3}{2}-\frac{1}{t}))} \\ & = \|\mathbf{y}\|_p^{-d}  \sum_{k= \vartheta(\mathbf{y})+1}^{\infty}  p^{k(-2\alpha_0-2d(
    1-\frac{1}{t}))} \lesssim  \|\mathbf{y}\|_p^{2 \alpha_0 +2d(\frac{1}{2}-\frac{1}{t})}  = \|\mathbf{y}\|_p^{2 \alpha_0 -2d(\frac{1}{t} - \frac{1}{2})} .
\end{align*}
This concludes the proof.\end{proof}

We conclude this section with the proof of Theorem \ref{teolittlewood}. We write in the context of compact $p$-adic Lie groups, but the arguments works the same for general compact Vilenkin groups.

\begin{rem}
By a standard duality argument, see for instance \cite[p.p.\ 105, 5.3.1]{Stein1971}, it is only 
necessary to prove
\[
\|Sf\|_{L^r_{\alpha}(G)} \lesssim \|f\|_{L^r_{\alpha}(G)}.
\]
\end{rem}

\begin{proof}[Proof of Theorem \ref{teolittlewood}.]
The proof is just an application of the properties of the Rademacher
functions, more precisely Khintchine inequality. For $n \in \mathbb{N}_0$ denote by
\[
r_n(s) = \operatorname{sgn}(\sin(2^n \pi s))
\]
the $n$-th Rademacher function. Let us define the Fourier multiplier $T_{\sigma_s}$ via the symbol
\[
\sigma_s(\xi)
:=
\sum_{n \in \mathbb{N}_0}
r_n(s)\,\delta_{\widehat{\G}(n)}(\xi)\,\mathrm{I}_{d_\xi}
=
\sum_{n \in \mathbb{N}_0}
r_n(s)\big(\widehat{\epsilon}_{\G_{n+1}}(\xi)
-
\widehat{\epsilon}_{\G_n}(\xi)\big).
\]

First we claim that $\sigma_t$ satisfies the hypothesis of Theorem \ref{teomultL2v1}.
To see this just compute $\Delta_\eta \sigma_t$:
\[
\Delta_\eta \sigma_s(\xi)
=
\sum_{n \in \mathbb{N}_0}
r_n(s)\big(
\Delta_\eta \widehat{\epsilon}_{\G_{n+1}}(\xi)
-
\Delta_\eta \widehat{\epsilon}_{\G_n}(\xi)
\big),
\]
where
\[
\Delta_\eta \widehat{\epsilon}_{\G_n}(\xi)
=
\widehat{\epsilon}_{\G_n}(\eta \otimes \xi)
-
\widehat{\epsilon}_{\G_n}(\mathrm{I}_{d_\eta} \otimes \xi)
=
\widehat{\epsilon}_{\G_n}(\eta \otimes \xi)
-
\mathrm{I}_{d_\eta} \otimes
\widehat{\epsilon}_{\G_n}(\xi).
\]

Since
\[
\widehat{\epsilon}_{\G_n}(\eta \otimes \xi)
=
\int_{\G_n}
\frac{1}{|\G_n|}
\,\eta(x) \otimes \xi(x)\,dx
=
\mathrm{I}_{d_\eta} \otimes
\delta_{\G_n^{\bot}}(\xi)\,\mathrm{I}_{d_\xi},
\quad \text{for } \| \eta \|_p < \| \xi \|_p,
\]
we see that $\Delta_\eta \widehat{\epsilon_{\G_n}}(\xi)$ is identically
zero. Consequently, by Theorem \ref{teomultL2v1} and Lemma \ref{prointerpolationDual}, $T_{\sigma_s}$ is bounded on
$L^2_\alpha(\G)$ for $-d < \alpha < d.$ Now we claim that $\sigma_s$ satisfies the hypothesis of Theorem \ref{teoL2impliesLr}.
Indeed, we write
\[
\mathcal{F}^{-1}_{\G}[\sigma_s](\mathbf{x} \smallstar \mathbf{y}^{-1})
-
\mathcal{F}^{-1}_{\G}[\sigma_s](\mathbf{x})
=
\sum_{n \in \mathbb{N}_0}
r_n(s)
\big(
\epsilon_{\G_{n+1}}(\mathbf{x} \smallstar \mathbf{y}^{-1})
-
\epsilon_{\G_{n+1}}(\mathbf{x})
\big)
+
\sum_{n \in \mathbb{N}_0}
r_n(s)
\big(
\epsilon_{\G_n}(\mathbf{x} \smallstar \mathbf{y}^{-1})
-
\epsilon_{\G_n}(\mathbf{x})
\big).
\]

Since $\|\mathbf{x} \smallstar \mathbf{y}^{-1}\|_p = \|\mathbf{x}\|_p$ whenever
$\|\mathbf{x}\|_p > \|\mathbf{y}\|_p$, we obtain in any case
\[
\mathcal{F}^{-1}_{\G}[\sigma_s](\mathbf{x} \smallstar \mathbf{y}^{-1})
-
\mathcal{F}^{-1}_{\G}[\sigma_s](\mathbf{x})
=
0,
\quad
\text{for } x \in \G \setminus \G_\ell,\; y \in \G_\ell:= \G \cap (\mathrm{I}_m + \mathcal{M}(p^\ell \Z_p)).
\]
Hence $\sigma_s$ satisfies condition $H(t)$ for $1 \le t < \infty$.
Thus, by Theorem \ref{teoL2impliesLr}, there exists a constant $C_r$ depending only on
$r$ such that
\[
\|T_{\sigma_s} f\|_{L^r_\alpha(\G)}^r
\le
C_r
\|f\|_{L^r_\alpha(\G)}^r,
\qquad
f \in \mathcal{D}(\G).
\]

Integrating the left-hand side with respect to $s \in [0,1]$ we obtain
\[
\|Sf\|_{L^r_\alpha(\G)}
\lesssim
\|f\|_{L^r_\alpha(\G)},
\]
where
\[
Sf(\mathbf{x})
=
\left(
\sum_{n \in \mathbb{N}_0}
\left|
\sum_{\|\pi\|_p = p^n}
d_\xi 
\operatorname{Tr}
\big(
\pi(\mathbf{x})\widehat{f}(\pi)
\big)
\right|^2
\right)^{1/2}.
\]

More generally, for compact Vilenkin groups we would write \[
Sf(x)
=
\left(
\sum_{n \in \mathbb{N}_0}
\left|
\sum_{[\pi] \in \widehat{G}(n)}
d_\pi
\operatorname{Tr}
\big(
\pi(x)\widehat{f}(\pi)
\big)
\right|^2
\right)^{1/2}.
\]
This concludes the proof.
\end{proof}

\section{Final remarks}

\subsection{Functions of the VT operator}

Let us start by defining radial functions on the unitary dual.

\begin{defi}\normalfont
A symbol \(\sigma : \widehat{\G} \to \bigcup_{[\pi] \in \widehat{\G}} \mathcal{L}(\mathcal{H}_\pi)\) is said to be \emph{radial} if there exists a function
\[
\varphi : \{|\G / \G_n|^{1/d} : n \in \mathbb{N}_0\} \to \mathbb{C}
\]
such that
\[
\sigma(\pi) = \varphi\big(\|\pi\|_p\big) \mathrm{I}_{d_\pi}, \quad \text{for all } \pi \in \widehat{G}.
\]

Equivalently, a radial function is constant on the ultrametric spheres
\[\widehat{\G}(n):=
\{\pi \in \widehat{\G} : \|\pi\|_p = |\G / \G_n|^{1/d}\}, \quad n \in \mathbb{N}_0.
\]
\end{defi}

The Littlewood-Paley decomposition in Theorem \ref{teolittlewood} implies that bounded radial mappings in the unitary dual define $L^r_\alpha$-bounded operators. To be more precise, remember how a pseudo-differential operator $T_\sigma$ on $\G$ is a linear operator with the form $$T_\sigma f(\mathbf{x}) := \sum_{n \in \N_0 } \sum_{ [ \pi] \in \widehat{\G}, \, \| \pi \|_p = |\G / \G_n|^{1/d} } d_\pi \, \mathrm{Tr}[\pi(\mathbf{x}) \sigma(\mathbf{x}, \xi) \widehat{f}(\pi)], \quad f \in L^2(\mathbb{G}),$$where the mapping $$\sigma: \G \times \widehat{\G} \to  \bigcup_{[\pi] \in \widehat{\G}} \mathcal{L}(\mathcal{H}_\pi), \quad (\mathbf{x}, \pi) \mapsto \sigma(\mathbf{x}, \pi) \in \mathcal{L}(\mathcal{H}_\pi),$$is called the symbol of the operator. If the symbol is radial, that is $$\sigma(\mathbf{x}, \pi) = \sigma(\mathbf{x}, \| \pi \|_p) \mathrm{I}_{d_\pi},$$then clearly \begin{align*}
    \| T_\sigma f \|_{L^r_\alpha(\G)} &\lesssim \| S(T_\sigma f)\|_{L^r_\alpha(\G)}
= \left( \int_{\G}
\left(
\sum_{n \in \mathbb{N}_0}
\left|
\sum_{\|\pi\|_p = p^n}
d_\xi 
\operatorname{Tr}
\big(
\pi(\mathbf{x})\sigma(\mathbf{x}, \| \pi \|_p) \widehat{f}(\pi)
\big)
\right|^2
\right)^{r/2} \| \mathbf{x}\|_p^\alpha d \mathbf{x} \right)^{1/r}  \\ &= \left( \int_{\G}
\left(
\sum_{n \in \mathbb{N}_0}|\sigma(\mathbf{x}, \| \pi \|_p) |^2
\left|
\sum_{\|\pi\|_p = p^n}
d_\xi 
\operatorname{Tr}
\big(
\pi(\mathbf{x})\widehat{f}(\pi)
\big)
\right|^2
\right)^{r/2} \| \mathbf{x}\|_p^\alpha d \mathbf{x} \right)^{1/r}\\ & \leq \sup_{(\mathbf{x},\pi) \in \G \times \widehat{\G}}  |\sigma(\mathbf{x}, \| \pi \|_p) | \| S f \|_{L^r_\alpha (\G)} ,
\end{align*}
proving that $$\sup_{(\mathbf{x},\pi) \in \G \times \widehat{\G}}  |\sigma(\mathbf{x}, \| \pi \|_p) | < \infty \implies T_\sigma \in \mathcal{L}( L^r_\alpha(\G)).$$ 
In particular, the Vladimirov-Taibleson operator from Definition \ref{defivladimirov} has a radial symbol, so for any function $G$ bounded on $[0, \infty)$ the operator $G(D^\alpha)$ defines a bounded operator on $L^r_\alpha (\G)$. In particular this is true for the inverse operators $\mathbb{D}^{-\alpha}$, proving how the negative powers of the VT operator are $L^r_\alpha$-bounded.

\subsection{Spaces adapted to the Vladimirov--Taibleson sub-Laplacian} For the rest of this section, let $\G$ be a compact nilpotent $p$-adic Lie group, with nilpotency class $\mathscr{N}(\G) <p $. We want to study now the weighted spaces associated to the weight $$w_\alpha (\mathbf{x}) = \| \mathbf{x} \|_{p, sub}^\alpha : = \|(x_1,...,x_\kappa) \|_p^\alpha,$$where we are choosing a basis $\{X_1,...,X_\kappa,...,X_d \}$ for the $\Z_p$-Lie algebra $\g$, and $X_1,...,X_\kappa$ is a basis for $\g/[\g,\g]$. For this weight function, we have the estimate $$\| \mathbf{x} \|_{p,sub}^\alpha \asymp \sum_{\{ \eta \in \widehat{\Z}_p^d \, : \, \eta_{\kappa+1} = ...= \eta_d = 1 \}} |e^{2 \pi i \{ \mathbf{x} \cdot \eta \}_p}-1|^2 \| \eta \|_p^{-(\alpha + \kappa)},\quad \kappa:= \mathrm{dim}_{\Q_p} (\g/[\g , \g]),$$
so we also have the following equivalence of norms $$\| f \|_{L^2_{\alpha, sub}(\G)}:= \Big(\int_{\G} |f(\mathbf{x})|^2 \| \mathbf{x} \|^\alpha_{p,sub} d \mathbf{x}\Big)^{1/2} \asymp \Big( \sum_{\xi \in \widehat{\G}} \sum_{\{ \eta \in \widehat{\Z}_p^d \, : \, \eta_{\kappa+1} = ...= \eta_d = 1 \}} d_\xi \| \eta \|_p^{-(\alpha + \kappa)}\| \Delta_\eta \widehat{f}(\xi) \|_{HS}^2 \Big)^{1/2}.$$
With similar techniques to the ones used for our main results, we can prove the following theorem: 

\begin{teo}\label{teosublaplacianspaces}
Let $\sigma \in L^\infty (\widehat{\G})$ be a symbol, al let $1 <t <2$ be a real number.
\begin{enumerate}
    \item Assume there is a $C = C(\sigma,\alpha) > 0$ such that for $\|\eta\|_p < \|\xi\|_p$,
    \[
    \left\| \Delta_{\eta}^{\frac{\alpha + \kappa}{2}} \sigma(\xi) \right\|_{\mathrm{op}} \leq C_2 \|\xi\|_p^\frac{-(\alpha + \kappa)}{2}, \quad \forall \xi \in \widehat{\mathbb{G}}.
    \] Then $T_\sigma$ extends to a bounded operator on $L^2_{\alpha, sub}(\mathbb{G})$.
    \item If there is a certain $\alpha_0> \frac{d}{2}$ and a certain constant  $C = C(\sigma,\alpha_0) > 0$ such that for $\|\eta\|_p < \|\xi\|_p$,
    \[
    \left\| \Delta_{\eta}^{\alpha_0 +\frac{ \kappa}{2}} \sigma(\xi) \right\|_{\mathrm{op}} \leq C_2 \|\xi\|_p^{-(\alpha_0+\frac{ d}{2})}, \quad \forall \xi \in \widehat{\mathbb{G}},
    \]Then $T_\sigma$ extends to a bounded operator on $L^r(\G)$ with $1 <r<\infty$.
    \item If there is a certain $\alpha_0> d \big( \frac{1}{t} - \frac{1}{2} \big)$ and a certain constant  $C = C(\sigma,\alpha_0,t) > 0$ such that for $\|\eta\|_p < \|\xi\|_p$,
    \[
    \left\| \Delta_{\eta}^{\alpha_0 +\frac{ d}{2}} \sigma(\xi) \right\|_{\mathrm{op}} \leq C_2 \|\xi\|_p^{-(\alpha_0+\frac{d}{2} + \kappa (
    \frac{1}{t}+\frac{1}{2}))}, \quad \forall \xi \in \widehat{\mathbb{G}},
    \]Then $T_\sigma$ extends to a bounded operator on $L^r_{\alpha, sub} (\G)$ with $1 <r<\infty$ and any $$ -d \leq \alpha \leq (r-1) d.$$
\end{enumerate}
\end{teo}

\begin{proof}

\quad

\textbf{Part 1:} Similar to the proof of Theorem \ref{teomultL2v1}, using the product rule we get to the estimation 
\begin{align*}
\quad & \|T_\sigma f \|^2_{L^2_{\alpha. sub}(\mathbb{G})} 
\lesssim \sup_{[\xi]\in \widehat{\mathbb{G}}} \|\sigma(\xi)\|^2_{op}\,\|f\|^2_{L^2_{\alpha, sub}(\mathbb{G})}
+ \sum_{[\xi]\in \widehat{\mathbb{G}}}
\sum_{\{ \eta \in \widehat{\Z}_p^d \, : \, \eta_{\kappa+1} = ...= \eta_d = 1 \}}
d_\xi \,\| \eta \|_p^{-(\alpha + \kappa)}
\|\Delta_\eta \sigma(\xi)\|^2_{op}\,\|\hat{f}(\xi)\|^2_{HS} .
\end{align*}
For the right hand side, notice how \begin{align*}
    \sum_{[\xi]\in \widehat{\mathbb{G}}}
\sum_{\{ \| \eta \|_p < \| \xi \|_p  \, : \, \eta_{\kappa+1} = ...= \eta_d = 1 \}}&
d_\xi \,\| \eta \|_p^{-(\alpha + \kappa)}
\|\Delta_\eta \sigma(\xi)\|^2_{op}\,\|\hat{f}(\xi)\|^2_{HS} \\ &= \sum_{[\xi]\in \widehat{\mathbb{G}}}
\sum_{\{ \| \eta \|_p < \| \xi \|_p  \, : \, \eta_{\kappa+1} = ...= \eta_d = 1 \}}
d_\xi \,
\|\Delta_\eta^{\frac{\alpha + \kappa}{2}} \sigma(\xi)\|^2_{op}\,\|\hat{f}(\xi)\|^2_{HS} \\ &\lesssim \sum_{[\xi]\in \widehat{\mathbb{G}}}
\sum_{\{ \| \eta \|_p < \| \xi \|_p  \, : \, \eta_{\kappa+1} = ...= \eta_d = 1 \}}
d_\xi \,
\| \xi  \|_p^{-(\alpha + \kappa)} \,\|\hat{f}(\xi)\|^2_{HS}
\\& \lesssim \sum_{[\xi]\in \widehat{\mathbb{G}}} d_\xi \,
\| \xi  \|_p^{- \alpha } \,\|\hat{f}(\xi)\|^2_{HS} \lesssim \| f \|_{L^2_{\alpha, sub}(\G)}^2.
\end{align*}

\textbf{Part 2 :}First, apply the Cauchy-Schwartz inequality:

\begin{align*}
    \int_{\G \setminus \G (p^\ell \Z_p)}|\check{\sigma}_{k,\mathbf{y}}(\mathbf{x})| d \mathbf{x} & \leq \Big(\int_{\G \setminus \G (p^\ell \Z_p)}|\check{\sigma}_{k,\mathbf{y}}(\mathbf{x})|^2 \| \mathbf{x}\|_{p, sub}^{2 \alpha_0} d \mathbf{x} \Big)^{1/2} \Big( \int_{\G \setminus \G (p^\ell \Z_p)}\| \mathbf{x} \|_{p, sub}^{-2 \alpha_0} d \mathbf{x} \Big)^{1/2}.
\end{align*}In one hand 
\begin{align*}
    \Big( \int_{\G \setminus \G (p^\ell \Z_p)}\| \mathbf{x} \|_{p, sub}^{-2 \alpha_0} d \mathbf{x} \Big)^{1/2} \leq \Big( \int_{(\Z_p^\kappa \setminus p^\ell \Z_p^\kappa) \times \Z_p^{d- \kappa}}\| \mathbf{x} \|_{p, sub}^{-2 \alpha_0} d \mathbf{x} \Big)^{1/2}  \lesssim p^{\ell(\alpha_0 - \frac{\kappa}{2})} \leq \| \mathbf{y} \|_p^{ \frac{\kappa}{2} - \alpha_0},
\end{align*}for any $\mathbf{y} \in \G (p^\ell \Z_p)$.  On the other hand, we know that 

$$\| \mathbf{x} \|_{p,sub}^\alpha \asymp \sum_{\{ \| \eta \|_p \leq \| y \|_p^{-1}  \, : \, \eta_{\kappa+1} = ...= \eta_d = 1 \}} |e^{2 \pi i \{ \mathbf{x} \cdot \eta \}_p}-1|^2 \| \eta \|_p^{-(\alpha + \kappa)},\quad \mathbf{y} \in  \G(p^\ell \Z_p), \, \mathbf{x} \in \G \setminus  \G(p^\ell \Z_p),$$
and therefore 
\begin{align*}    
\int_{\G \setminus \G(p^\ell \mathbb{Z}_p)}
|\check{\sigma}_{k,\mathbf{y}}(\mathbf{x})|^2
\|\mathbf{x}\|_{p, sub}^{2 \alpha_0}
\,d\mathbf{x}
&\lesssim
\sum_{\|\xi\|_p > \|\mathbf{y}\|_p^{-1}}
\sum_{\|\eta\|_p \le \|\mathbf{y}\|_p^{-1}}
d_\xi  
\|\eta\|_p^{-(2\alpha_0+\kappa)}
\|\Delta_\eta\sigma_k(\xi)\|_{HS}^2
\\&\leq  \sum_{\|\xi\|_p > \|\mathbf{y}\|_p^{-1}}
\sum_{\|\eta\|_p \le \|\mathbf{y}\|_p^{-1}}
d_\xi^2 
\|\Delta_\eta^{\alpha_0 + \frac{\kappa}{2}}\sigma_k(\xi)\|_{op}^2 \\ & \lesssim \sum_{\|\xi\|_p > \|\mathbf{y}\|_p^{-1}}
\sum_{\|\eta\|_p \le \|\mathbf{y}\|_p^{-1}}
d_\xi^2  
\|\xi \|_p^{-2\alpha_0 -d} \\ &\lesssim \|\mathbf{y}\|_p^{-\kappa} \sum_{\|\xi\|_p > \|\mathbf{y}\|_p^{-1}} d_\xi^2  
\|\xi \|_p^{-2\alpha_0 -d}  \lesssim \|\mathbf{y}\|_p^{2 \alpha_0 -\kappa} .
\end{align*}

\textbf{Part 3:} Assume $1<t<2$. Similarly, applying the H{\"o}lder inequality:

\begin{align*}
    \Big( \int_{\G \setminus \G (p^\ell \Z_p)}|\check{\sigma}_{k,\mathbf{y}}(\mathbf{x})|^t d \mathbf{x} \Big)^{1/t} & \leq \Big(\int_{\G \setminus \G (p^\ell \Z_p)}|\check{\sigma}_{k,\mathbf{y}}(\mathbf{x})|^2 \| \mathbf{x}\|_p^{2 \alpha_0} d \mathbf{x} \Big)^{1/2} \Big( \int_{\G \setminus \G (p^\ell \Z_p)}\| \mathbf{x} \|_p^{-\frac{2t \alpha_0}{2-t}} d \mathbf{x} \Big)^{\frac{2-t}{2t}}.
\end{align*}In one hand 
\begin{align*}
    \Big( \int_{\G \setminus \G (p^\ell \Z_p)}\| \mathbf{x} \|_p^{-\frac{2t \alpha_0}{2-t}} d \mathbf{x} \Big)^{\frac{2-t}{2t}} \lesssim p^{\ell(\alpha_0 - \kappa(\frac{1}{t}- \frac{1}{2}))} \leq \| \mathbf{y} \|_p^{ \kappa(\frac{1}{t}- \frac{1}{2}) - \alpha_0},
\end{align*}for any $\mathbf{y} \in \G (p^\ell \Z_p)$. On the other,  
\begin{align*}    
\int_{\G \setminus \G(p^\ell \mathbb{Z}_p)}
|\check{\sigma}_{k,\mathbf{y}}(\mathbf{x})|^2
\|\mathbf{x}\|_{p, sub}^{2 \alpha_0}
\,d\mathbf{x}
&\lesssim
\sum_{\|\xi\|_p > \|\mathbf{y}\|_p^{-1}}
\sum_{\|\eta\|_p \le \|\mathbf{y}\|_p^{-1}}
d_\xi  
\|\eta\|_p^{-(2\alpha_0+\kappa)}
\|\Delta_\eta\sigma_k(\xi)\|_{HS}^2
\\&\leq  \sum_{\|\xi\|_p > \|\mathbf{y}\|_p^{-1}}
\sum_{\|\eta\|_p \le \|\mathbf{y}\|_p^{-1}}
d_\xi^2 
\|\Delta_\eta^{\alpha_0 + \frac{\kappa}{2}}\sigma_k(\xi)\|_{op}^2 \\ & \lesssim \sum_{\|\xi\|_p > \|\mathbf{y}\|_p^{-1}}
\sum_{\|\eta\|_p \le \|\mathbf{y}\|_p^{-1}}
d_\xi^2  
\|\xi \|_p^{-2\alpha_0 -d - \kappa(\frac{1}{t} + \frac{1}{2})} \\ &\lesssim \|\mathbf{y}\|_p^{-\kappa} \sum_{\|\xi\|_p > \|\mathbf{y}\|_p^{-1}} d_\xi^2  
\|\xi \|_p^{-2\alpha_0 -d- \kappa(\frac{1}{t} + \frac{1}{2})}  \lesssim \|\mathbf{y}\|_p^{2 \alpha_0 -\kappa(\frac{1}{t} - \frac{1}{2})} .
\end{align*}

This concludes the proof.

\end{proof}

There are several advantages to the use of the above difference operators. For starters, they do not change the dimension of the representation space, and they are also easier to calculate. This comes in handy when studying the inverses of the following kind of operators.

\begin{defi}[Vladimirov sub-Laplacian]\label{defusublap}
Let $\mathbb{K}$ be a non-Archimedean local field with ring of integers $\mathcal{O}_{\mathbb{K}}$, prime ideal $\mathfrak{p} = p \mathcal{O}_{\mathbb{K}}$, and residue field $\mathbb{F}_q = \mathcal{O}_{\mathbb{K}} / p \mathcal{O}_{\mathbb{K}}$. Let $\mathfrak{g} = \operatorname{span}_{\mathcal{O}_{\mathbb{K}}}\{X_1,\dots,X_d\}$ be a nilpotent $\mathcal{O}_{\mathbb{K}}$-Lie algebra, and let $\mathbb{G}$ be the exponential image of $\mathfrak{g}$, so that $\mathbb{G}$ is a compact nilpotent $\mathbb{K}$-Lie group. Let $\{X_1,...,X_\kappa \}$ be a basis for $\g/[\g,\g]$.

We use the symbol $\D^{\alpha}_{sub}$ to denote the \emph{Vladimirov--Taibleson sub-Laplacian}, $\alpha>0$. defined by
\[
D_{sub}^{\alpha} f(x)
:=
\frac{1 - q^{\alpha}}{1 - q^{-(\alpha+\kappa)}}
\int_{\mathcal{O}_{\mathbb{K}}}
\frac{
f\big(\mathbf{x} \star \exp(x_1X_1 +...+x_\kappa X_\kappa)^{-1}\big) - f(\mathbf{x})
}{
\|(x_1,...,x_\kappa)\|_{\mathbb{K}}^{\alpha+\kappa}
}\, dx_1...dx_\kappa,
\, \,\;\; f \in \mathcal{D}(\mathbb{G}).
\]

Here $\mathcal{D}(\mathbb{G})$ denotes the space of smooth functions on $\mathbb{G}$, that is, the collection of locally constant functions with a fixed index of local constancy.
\end{defi}

In the particular case of the Heisenberg group we can compute 
\begin{align*}
    &\sigma_{D^\alpha_{sub}} (\xi)_{hh'}  = D^\alpha_{sub}(\pi_\xi(\mathbf{x}))_{hh'} |_{\mathbf{x}=\mathbf{e}} \\&= \frac{1-p^\alpha}{1-p^{-(\alpha + 2d)}}\int_{\Z_p^{2d}} \frac{e^{2\pi i \{x_1 \cdot \xi_1 + x_2 \cdot \xi_2 + \xi_3 h' \cdot x_2\}_p} \mathbb{1}_{h'-h + p^{-\vartheta(\xi_3)}\mathbb{Z}_p^d}(x_1) - \delta_{hh'}}{\|(x_1, x_2)\|_p^{\alpha + 2d}}dx_1 dx_2\\&=C_\alpha  \int_{\Z_p^{2d}} \frac{e^{2\pi i \{(x_1+u_1) \cdot \xi_1 + (x_2+u_2) \cdot (\xi_2 + \xi_3 h')\}_p} \mathbb{1}_{h'-h}(u_1+x_1) - e^{2\pi i \{u_1 \cdot \xi_1 + u_2 \cdot (\xi_2 + \xi_3 h')\}_p}\mathbb{1}_{h'-h }(u_1)}{\|(x_1, x_2)\|_p^{\alpha + 2d}}dx_1 dx_2 \big|_{u_1=u_2=0} \\ &= D^\alpha_{\Z_p^{2d}} (e^{2\pi i \{u_1 \cdot \xi_1 + u_2 \cdot (\xi_2 + \xi_3 h')\}_p}\mathbb{1}_{h'-h + p^{-\vartheta(\xi_3)}\mathbb{Z}_p^d}(u_1))|_{u_1=u_2=0}.
\end{align*}
Using the $\Z_p^d$-Fourier transform we can write $$\mathbb{1}_{h'-h + p^{-\vartheta(\xi_3)}\mathbb{Z}_p^d}(u_1) = \sum_{\| \tau \|_p \leq |\xi_3|_p} C_{hh'}(\tau) e^{2 \pi i \{ u_1 \cdot \tau  \}_p },$$and therefore \begin{align*}
    &D^\alpha_{\Z_p^{2d}} (e^{2\pi i \{u_1 \cdot \xi_1 + u_2 \cdot (\xi_2 + \xi_3 h')\}_p}\mathbb{1}_{h'-h + p^{-\vartheta(\xi_3)}\mathbb{Z}_p^d}(u_1)) \\ &= \sum_{\| \tau \|_p \leq |\xi_3|_p}  \Big( \|(\tau + \xi_1 , \xi_2 + \xi_3 h') \|_p^\alpha - \frac{1 - p^{-d}}{1 - p^{- (\alpha + d)}} \Big) C_{hh'}(\tau) e^{2 \pi i \{ u_1 \cdot (\tau+ \xi_1)+ u_2 \cdot ( \xi_2+ \xi_3 h' )  \}_p } 
\end{align*}
From this it is easy to deduce that $D^\alpha_{sub}$ has a non-trivial kernel, composed by all the central functions on $\mathbb{H}_d$: 
$$\mathrm{Ker}(D^\alpha_{sub}) = \mathrm{Span}_\C \{ e^{2 \pi i \{\lambda x_3 \}_p} \, : \, \lambda  \in \widehat{\Z}_p \},$$so the symbol $\sigma_{D^\alpha_{sub}} (\xi)$ has a one-dimensional kernel on each representation space. This means that $D^\alpha_{sub}$ is not invertible modulo constant functions, that is, there is no fundamental solution for $D^\alpha_{sub}$. And, even if we think on its inverse modulo kernel, it is still not clear if the resulting operator would be $L^r$-bounded. So, we consider instead the closely related operator $$\mathscr{L}^\alpha := D^\alpha_{sub} + \partial^\alpha_{X_3},$$where $$\partial^\alpha_{X_3}f(\mathbf{x}):= \frac{1-p^\alpha}{1-p^{-(\alpha + 1)}}\int_{\Z_p} \frac{f(\mathbf{x} + (0,0,t)) - f(\mathbf{x})}{|t|_p^{\alpha + 1}} dt,$$which is an alternative form of Vladimirov Laplacian. This operator is invertible (modulo constant functions) and its associated symbol satisfies for $\| \eta \|_p < \| \xi \|_p$: $$\Delta_\eta \sigma_{\mathscr{L}^\alpha} (\xi) = \mathbf{0}_{d_\xi}, \quad \eta_{\kappa+1}...=\eta_d =1,$$so we have the same conclusion for the inverse powers;   $$\Delta_\eta \sigma_{(\mathscr{L}^\alpha)^{- \beta}} (\xi) = \mathbf{0}_{d_\xi}, \quad \eta_{\kappa+1}...=\eta_d =1,$$showing how $(\mathscr{L}^\alpha)^{-\beta}$ defines a bounded operator on $L^r_{\alpha, sub} (\G)$ as a consequence of Theorem \ref{teosublaplacianspa}. More generally, any bounded function of $\mathscr{L}^\alpha$ defines a bounded operator on $L^r_{\alpha, sub}(\G)$.

\nocite{*}
\bibliographystyle{acm}
\bibliography{main}
\Addresses

\end{document}